\documentclass[10pt]{amsart}
\usepackage{graphicx}
 \usepackage{caption}
\usepackage{psfrag}
\usepackage{float}
\usepackage{subfigure}
\usepackage{psfrag}
\usepackage{epsfig}
\usepackage{epstopdf}
\usepackage{amsmath}
\usepackage{amscd}
 \def\Ta{{\mathcal T}}
\def\ix{{\epsilon}}
\def\Bn{ B_N}
\def\norma{{\|}}
\def\td{{\tau_1}}
\def\esse{{S_0}}
\def\PaPi{{K}}\def\val{{\vartheta}}

\def\er{{\mathtt r}}
\def\ro{{\varepsilon^2}}
\def\pluto{{\mathtt j}}
\def\art{{\bar \Theta}}
\usepackage{amsthm,amsfonts,amsxtra,amssymb,amscd}
\usepackage[all]{xy}
\xyoption{2cell}
\usepackage{enumerate}
\def\Gama{{\Gamma_S}}

\def\GA{{\mathcal A}}
\def\Co{{\mathbb C}}
\usepackage{pdfsync}
\def\pd#1#2{\dfrac{\partial#1}{\partial#2}}

\font\teneufm=eufm10
\font\seveneufm=eufm7
\font\fiveeufm=eufm5
\newfam\eufmfam
\textfont\eufmfam=\teneufm
\scriptfont\eufmfam=\seveneufm
\scriptscriptfont\eufmfam=\fiveeufm
\def\eufm#1{{\fam\eufmfam\relax#1}}
\def\rro{{\varepsilon^2}}

\def\Vgot{{\eufm V}}

\def\vgot{{\eufm v}}

\def\C{{\mathbb C}}

\def\teta{\theta}

\theoremstyle{plain}
\newtheorem*{lemma*}{Lemma}
\newtheorem{lemma}[subsection]{Lemma}
\newtheorem*{theorem*}{Theorem}
\newtheorem{theorem}{Theorem}
\newtheorem*{proposition*}{Proposition}
\newtheorem{proposition}[subsection]{Proposition}
\newtheorem*{corollary*}{Corollary}

\newtheorem{step}{KAM Step}
\newtheorem{corollary}[subsection]{Corollary}
\theoremstyle{definition}
\newtheorem*{definition*}{Definition}

\newtheorem{definition}[subsection]{Definition}
\newtheorem*{example*}{Example}
\newtheorem{example}[subsection]{Example}
\theoremstyle{remark}
\newtheorem*{remark*}{Remark}\newtheorem*{proof*}{proof}
\newtheorem{remark}[subsection]{Remark}

\def\Q{{\mathbb Q}}
\def\T{{\mathbb T}}
\def\frec#1{{\stackrel{#1}\rightarrow}}

\def\N{{\mathbb N}}

\def\cc{{\mathtt c}}
\def\CC{{\mathtt C}}
\newcommand{\be}{\begin{equation}}
\newcommand{\ee}{\end{equation}}
\newcommand{\om}{\omega}
\newcommand{\e}{\varepsilon}

\newcommand{\al}{\alpha}

\renewcommand{\a }{\alpha }
\renewcommand{\b }{\beta }
\newcommand{\s }{\sigma }
\newcommand{\ii }{{\rm i} }
\renewcommand{\d }{\delta }

\newcommand{\g }{\gamma}
\renewcommand{\l }{\lambda }

\renewcommand{\o }{\omega }
\def\ome{{\omega}}
\def\Ome{{\Omega}}
\renewcommand{\O }{\mathcal O }
\def\Pappa{{{K}}}
\newcommand{\Z}{\mathbb{Z}}
\newcommand{\R}{\mathbb{R}}

\newcommand{\n}{\mathtt j}\def\ome{{\omega}}
\def\Ome{{\Omega}}
\renewcommand{\O }{\mathcal O }

\renewcommand\c{{\bar \Theta}}
\def\Art{{\mathtt C_0}}

 \title[A KAM algorithm for the NLS]{A KAM algorithm for the resonant  non--linear Schr\"odinger  equation }
\author{ M. Procesi*, \and 
 C. Procesi{**}. }
 \thanks{ *Universit\`a di Roma, La Sapienza supported by ERC under FP7, {**}  Universit\`a di Roma, La Sapienza }
\begin{document}

\begin{abstract}
We prove, by applying a KAM algorithm, existence of large families of stable and unstable    quasi periodic solutions  for the NLS in any number of independent frequencies. The main tools are the existence of a non-degenerate integrable normal form proved in \cite{PP} and \cite{PP1} and a suitable generalization of the quasi-T\"oplitz functions introduced in \cite{PX}
 \end{abstract}\maketitle
 
\tableofcontents

\section{Introduction}

 The present paper is devoted to the  construction of  stable and unstable   quasi--periodic solutions for the completely resonant cubic NLS equation on a torus $\T^d$:

\begin{equation}\label{scE}iu_t-\Delta u=\mathtt k |u|^{2}u +\partial_{\bar u}G(|u|^2).\end{equation}
Here $u:= u(t,\varphi)$, $\varphi\in \T^d$, $\Delta$ is the Laplace operator, $G(a)$ is a real analytic function whose Taylor series starts from degree $3$ and $\mathtt k=\pm 1$.   
   
  Our results are obtained by exploiting the Hamiltonian structure of equation \eqref{scE} and  applying  a KAM algorithm.  As is well known such algorithms require strong {\em non-degeneracy} conditions  which  are not always valid, even for finite dimensional systems and, when valid, are generally proved by  performing on the Hamiltonian a few steps of Birkhoff  normal form.   This is done in \cite{PP} (where the results are for for a larger class of NLS with non--linearity $|u|^{2q}u$)  and \cite{PP1} in which we  exhibit  and study the normal forms   for classes of  completely resonant non--linear Schr\"odinger  equations. 
  
 Let us give a brief overview of the main difficulties in proving existence and  stability of small quasi-periodic solutions for PDEs. One starts with a Hamiltonian PDE which has an elliptic fixed point at $u=0$, and wishes to prove that some of the solutions of the nonlinear equation stay {\em close} to the linear solutions for all time.  One has to deal with two classes of problems: 
 
 the {\em resonances}, namely the equation linearized at $u=0$ does not have quasi-periodic solutions so we are dealing with a singular perturbation problem;
 
  the {\em small divisors},  namely the equation linearized at $u=0$ is described by an operator whose inverse is unbounded so that in order to find small solutions one needs to use some {\em Generalized Implicit Function Theorem}.
   
   There are two main approaches to these problems:\quad 1. Use a combination of  Lyapunov--Schmidt reduction techniques  and a Nash--Moser algorithm to solve the small divisor problem. This is the so--called {\em Craig--Wayne--Bourgain} approach, see \cite{CW} ,\cite{Bo2} and  for a recent generalization also \cite{BBhe}.\quad
   2. Use a combination of Birkhoff normal form and a KAM algorithm, see for instance \cite{KP}, \cite{BB}.
   
   In both cases one usually studies simplified models, namely parameter families of PDEs  with the parameters chosen in such a way as to avoid resonances. Even under this simplifying hypotheses the problems related to the small divisors are in general quite complicated. Essentially, in order to perform a quadratic iteration scheme to prove the existence of quasi-periodic solutions, one needs some control on the operator  (which we denote by $L(u)$) describing the equation linearized on an approximate solution $u(x,t)$  (and not only  at $u=0$). In the Nash--Moser scheme one requires very weak hypotheses, in order to ensure that one may define a left inverse for $L(u)$ with some control on the {\em loss of regularity}.
   In a KAM scheme instead one imposes  lower bounds on  the {\em eigenvalues} of  $L(u)$ and on {\em  their differences}, this allows to prove a stronger result namely the NLS operator  linearized at a quasi-periodic solution can be diagonalized by an analytic time dependent change of variables.  Note that these last hypotheses imply a very good control on the loss of regularity of $L^{-1}$.  The Nash-Moser approach combined with reducibility arguments (together with some novel ideas from pseudo-differential calculus) was used in \cite{BBM} in order to prove existence and stability of quasi-periodic solutions for a class of fully-nonlinear perturbations of the KdV equation.

  For PDEs in dimension $d>1$, where the eigenvalues of $L(0)$ are clearly multiple, the Nash--Moser algorithm is more readily applicable, see for instance \cite{Bo2}.     KAM results for PDEs in dimension $d>1$ are few and relatively recent, see for instance \cite{GY}, and  in particular the paper \cite{EK} which studies an NLS with external parameters.
  Note that not only one needs to impose that the eigenvalues are different but one must give a lower bound on the difference. In the case of the NLS of  \cite{EK} this requires a subtle analysis and the introduction of the class of  T\"oplitz-Lipschitz functions (see also \cite{PX}).
   
   In the case of equation \eqref{scE}, before attempting to study the small divisor problem one must deal with the resonances, since there are no external parameters and the only freedom is in the choice of the  {\em initial data}. 
   
   In the case of \eqref{scE} in dimension one this  problem is avoided by just performing a step of Birkhoff normal form then applying a KAM algorithm (see \cite{KP} or  \cite{BB}). This is due to the fact that the  NLS equation after one step of Birkhoff normal form is integrable and non-degenerate. Unfortunately this very strong property holds only for the cubic NLS in dimension one, indeed for $d>1$ the non-integrability of the NLS normal form has been exploited (see for instance \cite{KT} and \cite{KG}) to construct diffusive orbits.
   In order to overcome this problem Bourgain proposed the idea of choosing the initial data wisely. More precisely 
   one looks for  a set  $S\subset \Z^d$, the  tangential sites, such that the Birkhoff normal form Hamiltonian admits quasi-periodic solutions which excite only the modes $\mathtt j\in S$. Then, by choosing $S$ appropriately, one may prove existence of true solutions nearby. 
   
    This idea was used in \cite{Bo2}  to  prove the existence of quasi--periodic solutions with two frequencies for the cubic NLS in dimension two. This strategy was generalized by Wang in \cite{W09},\cite{W11} to study the NLS on a torus $\T^d$ and prove existence of quasi periodic solutions. 
    
    A similar idea was exploited in \cite{GP2} and \cite{GP3} to look for   ``wave packet'' periodic solutions (i.e. periodic solutions which at leading order excite an arbitrarily large number of ``tangential sites'') of the cubic NLS in any dimension both in the case of periodic and Dirichlet boundary conditions.  All the previous papers only deal with the existence of quasi-periodic solutions and not  the linear stability and  reducibility of the normal form. Note that the results  by Wang on the NLS imply existence of quasi-periodic solutions for equation \eqref{scE} and indeed her  approach to the resonance problem is parallel in various ways  to the one of \cite{PP}. Her approach is through the Nash--Moser method and hence does not prove  reducibility results as explained before. Note however that \cite{W11} covers a larger class of cases i.e. non-cubic NLS equations with explicit dependence on the spatial variable. 
     
     In the context of KAM theory and normal form, we mention the result of \cite{G} for the NLS in dimension one with the nonlinearity $|u|^4u$.  
     
    A  strategy similar to the one used in this paper  is proposed   by Geng You and Xu in \cite{GYX}, to study the cubic NLS in dimension two. In that paper the authors show that one may give constraints on the tangential sites so that the normal form is non--integrable (i.e. it depends explicitly on the angle variables) but block diagonal with blocks of dimension $2$.    They apply this result to perform a KAM algorithm and prove existence (but not stability) of quasi--periodic solutions. 
    We also mention the paper \cite{MP}, which studies the non-local NLS and the beam equation both for periodic and Dirichlet boundary conditions. 
   \smallskip
   
   The present paper is the last of a series of three papers in which we have developed a strategy aimed at the construction of large families of stable and unstable quasi-periodic solutions for the cubic NLS \eqref{scE} in any dimension.

In the first paper  \cite{PP}  we study the NLS equation after one step of Birkhoff normal form and give   ``genericity conditions''  on the {\em tangential sites} $S\subset \Z^d$ in order to make the normal form  as simple as possible. 
  
  The main results of \cite{PP} are formulated in Theorem \ref{teo1}, where  we prove that,  for  $|S|<\infty$ and for generic $S$, one can choose symplectic coordinates in which the normal form is integrable. 
  On the {tangential variables} the normal form is non-degenerate and the motion is quasi-periodic with frequency
  $\omega=\omega(\xi)$, where $\omega(\xi)$ is a diffeomorphism and $\xi\in \R^n$ ($n=|S|$)  are free parameters modulating the initial data.  Moreover, in the ``normal variables,'' the normal form is a block diagonal quadratic form, with blocks of dimension at most $d+1$. All blocks have constant coefficients. These infinitely many blocks are explicitly  described by a graph $\Gamma_S$ (cf. \S \ref{gra}) which contains all the combinatorial difficulties of the structure. This  combinatorial structure will influence the KAM--algorithm presented in this paper.
\smallskip

In \cite{PP1} we address the delicate question of  the {\em non-degeneracy} of the normal form deduced in \cite{PP}, we obtain precise positive results for the cubic case $q=1$. 

In this paper we address the question of constructing  quasi--periodic solutions and present a general solution.  We need to analyze three issues
\begin{enumerate}
\item The second Melnikov non--degeneracy condition. This we  prove by using the results of \cite{PP1}.
\item  The T\"oplitz--Lipschitz (cf. \cite{EK}) or quasi--T\"oplitz property of the perturbation. This    is done by generalizing the quasi--T\"oplitz  functions of \cite{PX} to this context; in particular we need to prove that the changes of variables that we perform to integrate the normal form do not destroy the quasi--T\"oplitz structure.
\item The KAM algorithm. This is a variation (with some complications) of a well established path; we follow closely the structure of \cite{BB} and of \cite{PX}.
\end{enumerate}
\smallskip

In all these steps we need to combine the analysis of \cite{PX} with the special structure of the graph $\Gamma_S$.  This is the source of most of the specific problems for the NLS which make this case particularly complex.\smallskip

The final result will be the construction, for any dimension $d$, of families of  linearly stable (and also elliptic for appropriate initial data) quasi--periodic solutions for the cubic NLS.

Following \cite{PP} for any $n\in \N$  we introduce the notion of {\em generic set of frequencies}  $S\subset \Z^d$ with $|S|=n$ (see Definition 3.of \cite{PP}). 
The conclusive result of this analysis is: 
\begin{theorem}
 For any $n\in \N$ and any generic  set of frequencies $S=\{\mathtt j_1,\dots \mathtt j_n\} \subset\Z^d$ 
 the NLS equation 
 \eqref{scE} 
 admits small-amplitude, analytic (both in $ t $ and $ \varphi
 $), quasi-periodic solutions of the form \begin{equation}\label{solutions} u(\varphi,t) =
 \sum_{j \in {S}} \sqrt{  \xi_j} e^{\ii (\omega^\infty(\xi)t+ j\cdot \varphi)} + o( \sqrt{\xi}
 ), \quad
  \omega^\infty_j (\xi) \stackrel{\xi \to 0}\approx |j|^2
  \end{equation}
 for all sufficiently small $\xi \in \R^n$ belonging to a  "Cantor-like" set of parameters
 with asymptotical density $ 1$ at $ \xi = 0 $. The term $ o( \sqrt{\xi} ) $ in \eqref{solutions}
  is small in an analytic norm. The equations \eqref{scE} linearized at these
 quasi-periodic solutions  are diagonable by an analytic time dependent change of variables. Finally, in a non empty open set of this Cantor set the solutions are also elliptic and linearly stable. 
 
 \end{theorem}

\smallskip

We prove this result  by verifying that the NLS Hamiltonian can be brought into a normal form which satisfies the properties of an abstract KAM Theorem, Theorem \ref{KAM}.

Most of the properties necessary for Theorem \ref{KAM}  have been verified for the NLS in Theorem 1 of \cite{PP} and in \cite{PP1}, here we  have to  prove the quasi--T\"oplitz property of the NLS, cf. \S \ref{tpnls}.\smallskip

It is possible to perform a KAM algorithm for any analytic NLS obtaining a weaker result. In this case one has the second Melnikov condition property  i) only in a finite {\em block form}. This will be discussed elsewhere.

 \subsubsection{The plan of the paper} The paper is divided into four parts. In  Part 1 we  recall all the properties of the normal form proved in \cite{PP} and \cite{PP1}  which will be needed.  In Part two we start by recalling the  geometric formalism  developed in \cite{PX}  and prove that this formalism is compatible with the structure of the graph $\Gamma_S$.  Having done this we proceed to define quasi-T\"oplitz functions in our context and prove their basic properties.  Parts 3 and 4 are devoted to the KAM algorithm.  In Part 3 we discuss the general properties of the type of algorithm that we shall apply to the NLS while in the final Part 4 we  verify that the NLS  satisfies all the properties of the class of Hamiltonians studied in Part 3. We can finally conclude that the KAM algorithm, applied to the Hamiltonian of the NLS starting from the normal form described in Part 1,  leads to a successful construction of a family of quasi--periodic solutions of the NLS parametrized by a set of positive measures of the parameters $\xi_i$,  actions of the initial excited frequencies.  We discuss also which solutions are stable or unstable.

\part{The normal form} 
\section{Summary of   results from \cite{PP}}
 \subsubsection{The Hamiltonian} In \cite{PP} we have studied   the NLS  on $\T^d$ as an infinite dimensional Hamiltonian system. After rescaling and passing to     Fourier representation\footnote{\label{ft}In fact one should work in a slightly more general setting where the torus is the quotient of $\R^d$ by any lattice $\Lambda$ of finite index in $\Z^d$ and write $u(t,\varphi):= \sum_{k\in  \Lambda^*} u_k(t) e^{\ii (k, \varphi)}$.}
\begin{equation}
u(t,\varphi):= \sum_{k\in \Z^d} u_k(t) e^{\ii (k, \varphi)}\ 
\end{equation} 
the Hamiltonian is (having normalized $\kappa$):
\begin{equation}\label{Ham}H:=\sum_{k\in \Z^d}|k|^2 u_k \bar u_k + \sum_{k_i\in \Z^d: \sum_{i=1}^{4}(-1)^i k_i=0}\hskip-30pt u_{k_1}\bar u_{k_2}u_{k_3}\bar u_{k_4}  . \end{equation}
The complex symplectic form is $ i \sum_{k}d u_k\wedge d \bar u_k$, on the scale of complex Hilbert spaces 
\begin{equation}\label{scale}
{\bf{\bar \ell}}^{(a,p)}:=\{ u=\{u_k \}_{k\in \Z^d}\;\big\vert\;|u_0|^2+\sum_{k\in \Z^d} |u_k|^2e^{2 a |k|} |k|^{2p}:=||u||_{a,p}^2 \le \infty \},
\end{equation}$$\ a>0,\ p>d/2.$$
 
We   systematically apply the fact that we have $d+1$ conserved  quantities:
the  $d$--vector {\em momentum} \ $\mathbb M$ and the scalar {\em mass}  $\mathbb L$:
$$\mathbb  M:=\sum_{k\in \Z^d} k |u_k|^2\,,\qquad \mathbb L:= \sum_{k\in \Z^d} |u_k|^2\,,$$ with 
\begin{equation}
\{\mathbb M,u_h\}=\ii h u_h,\ \{\mathbb M,\bar u_h\}=-\ii h \bar u_h,\ \{\mathbb L,u_h\}=\ii u_h,\ \{\mathbb L,\bar u_h\}=-\ii \bar u_h.\ 
\end{equation}\smallskip
The terms in equation \eqref{Ham} commute with $\mathbb L$. The conservation of momentum is expressed by the constraints $\sum_{i=1}^{4}(-1)^i k_i=0$.

 \subsubsection{Choice of the tangential sites\label{cts}} If in the Hamiltonian $H $ we remove all quartic terms which do not Poisson commute with the quadratic part, we obtain a simplified Hamiltonian  denoted $H_{Birk}$. This has the property that its Hamiltonian vector field is tangent to infinitely many subspaces obtained by setting some of the coordinates equal to  0 (cf. \cite{PP}, Proposition 1). On infinitely many of them furthermore the restricted system is  completely integrable, thus  the next step  consists in choosing such a subset $S$ which, for obvious reasons, is called of {\em tangential sites}. Without loss of generality one may assume that $S$ spans $\Z^d$ over $\Z$ (cf.   footnote  \ref{ft}).
\smallskip

With this remark in mind we  partition 
\begin{equation}\label{essec}
\Z^d= S\cup S^c,\quad S:=(\mathtt j_1,\ldots,\mathtt j_n)\end{equation}
where 
the elements of  $S$ play the role of {\em tangential sites} and of $S^c$ the {\em normal sites}. We divide $u\in \bar\ell^{a,p}$ in two components $u=(u_1,u_2)$, where $u_1$ has indexes in $S$ and $u_2$ in $S^c$.   The choice of $S$ is subject to several constraints which make it {\em generic} and which are fully discussed in \cite{PP} and finally refined in \cite{PP1}. Here we shall always assume that these constraints are  valid so we just refer to the results of these  two papers  in all the statements.\smallskip

We often use the map $\pi:\R^n\to \R^d,\, \,\pi(a_1,\ldots,a_n):=\sum_ia_i\mathtt j_i$, notice that $\pi$ maps $\Z^n$ to $\Z^d$, and set
\begin{equation}\label{essec1}
  \kappa:=\max_{j\in S} |\mathtt j| .
\end{equation} If we use on $\R^n$ the $L^1$ norm then $\kappa$ is also the norm of the map $\pi$.

 We apply a standard {\em semi-normal form} change of variables with generating function:
\begin{equation}\label{birkof}
  F_{Birk}= -\ii \sum_{\alpha,\beta\in (\Z^d)^\N: |\alpha|=|\beta|=2\,, |\alpha_2|+|\beta_2|\leq 2\atop {\sum_k (\alpha_k-\beta_k)k=0\,,\;\sum_k (\alpha_k-\beta_k)|k|^2\neq 0}} \hskip-10pt\binom{2}{\alpha}\binom{2}{\beta}\frac{u^\alpha\bar u^\beta}{\sum_k (\alpha_k-\beta_k)|k|^2}. 
\end{equation}
Here the notation $\alpha_2,\beta_2$ refers to the exponents for the variable $u_k,\bar u_k$ with $k\in S^c$. 
We use the operator notation $ad(F)$ for the operator $X\mapsto\{F,X\}$. 
The change of variables by $\Psi^{(1)}:=e^{ad( F_{Birk})}$  is well defined and analytic:    $   B_{\epsilon_0}\times B_{\epsilon_0} \to B_{2{\epsilon_0}} \times B_{2\epsilon_0}$,  for $\epsilon_0$ small enough, see \cite{PP}.   By construction $\Psi^{(1)}$  brings (\ref{Ham}) to  the form
 $H= H_{Birk} +P^{4}(u)+P^{6}(u)$ where $P^{4}(u)$ is of degree $4$ but at least cubic in $u_2$ while $P^{ 6 }(u)$ is analytic of degree at least $6$ in $u$, finally 
\begin{equation}\label{Ham2}H_{Birk}:=\sum_{k\in \Z^d}|k|^2 u_k \bar u_k +  \hskip-30pt\sum_{\alpha,\beta\in (\Z^d)^\N: |\alpha|=|\beta|=2\,,\, |\alpha_2|+|\beta_2|\leq 2 \atop {\sum_k (\alpha_k-\beta_k)k=0\,,\;\sum_k (\alpha_k-\beta_k)|k|^2=0}} \hskip-10pt\binom{2}{\alpha}\binom{2}{\beta}u^\alpha\bar u^\beta.
\end{equation} 

The three constraints in the second summand of the previous formula express the conservation of $\mathbb L$, $\mathbb M$ and  of the {\em quadratic energy} 
\begin{equation}
\label{kappabb}\mathbb K:= \sum_{k\in \Z^d}|k|^2 u_k \bar u_k .
\end{equation}
 \vskip10pt
 
 In order to perform perturbation theory from the  system given by the tangential sites it is convenient to switch to polar coordinates.
We set \begin{equation}
\label{chofv}u_k:= z_k \;{\rm for}\; k\in S^c\,,\quad u_{\mathtt j_i}:= \sqrt {\xi_i+y_i} e^{\ii x_i}= \sqrt {\xi_i}(1+\frac {y_i}{2 \xi_i }+\ldots  ) e^{\ii x_i}\;{\rm for}\;  i=1,\dots n,
\end{equation}  considering  the $\xi_i>0$ as parameters $|y_i|<\xi_i$ while $y,x,w:=(z,\bar z)$ are dynamical variables.   We  denote   by   $ {\bf \ell}^{(a,p)}:= {\bf\ell}^{(a,p)}_S $  the subspace of $\bf{\bar \ell}^{(a,p)}\times\bf{\bar \ell}^{(a,p)} $  of the sequences $u_i,\bar u_i$ with indices in $S^c$ and denote the coordinates  $w=(z,\bar z) $.
\begin{definition}\label{Aro}
Let $\mathfrak K\subset \R^n_+$ be a compact domain and let $0<c_1<c_2$  be such that 
$$ c_1^2= \min_{\mathfrak K}(\min_i\xi_i) \,,\quad c_2^2= \max_{\mathfrak K}(\max_i\xi_i) $$
It is convenient to choose as $\mathfrak K=\mathfrak H\times J$ a product  in polar coordinates of a compact domain $\mathfrak H$  in the unit sphere and some compact set  $J$ in the coordinate $\rho$.
We will consider  for all $\rho>0$ the {\em scaled  domain} $\rho \mathfrak K$ and notice that $\rho\mathfrak K=\mathfrak H\times \rho J$.

One can refer to such a domain as a {\em truncated cone}.
\end{definition}
We choose $\rho=\e^2$ and note that,  for all $r<  c_1  \e$, formula \eqref{chofv} is an analytic and symplectic change of variables $\Phi_\xi$ in  the  domain
\begin{equation}\label{domain}  D_{a,p}(s,r)= D(s,r):= 
 \{   x,y,w\,:\   x\in \T^n_s\,,\  |y|\le r^2\,,\  \|w\|_{a,p}\le r\}\subset \T^n_s\times\Co^n\times {\bf{\ell}}^{(a,p)}.
\end{equation} Here $\e>0$, $s>0$ and $ 0<r<\e c_1$ are auxiliary parameters. $\T^n_s$ denotes the compact subset of the complex torus $\T_{\Co}^n:=\Co^n/2\pi\Z^n$ where   $ x\in\Co^n,\ |$Im$(x)|\leq s$. Moreover if \begin{equation}
\label{bapa}\sqrt{2 n} c_2 \kappa^{p} e^{ ( s+ a\kappa)} \e  < {\epsilon_0}\,, \quad(\mbox{recall}\quad \kappa=\max(|\mathtt j_i|)\;)
\end{equation} the change of variables sends $D(s,r)\to B_{{\epsilon_0}}$ so we can apply it to our Hamiltonian.

We thus assume that the parameters $\e,\, r,\, s$  satisfy \eqref{bapa}.
Formula \eqref{chofv}  puts  in action angle variables  $(y;x)= (y_1,\dots, y_n;x_1,\dots, x_n) $ the tangential sites, close to the action $\xi= \xi_1,\dots, \xi_n$, which are parameters for the system.  

The  symplectic form is now $ dy \wedge dx + i \sum_{k\in S^c} dz_k\wedge d \bar z_k $.

We give degree $0$ to the angles $x$, $2$ to $y$ and $1$ to $w$.  
We use the degree only for handling dynamical variables, as follows. We develop  in  Taylor expansion, in particular since $y$ is small with respect to $\xi$ we develop $\sqrt{\xi_i+y_i}= \sqrt{\xi_i}(1+\frac{y_i}{2\xi_i}+\ldots)$ as a series in $\frac{y_i}{\xi_i}$.  
\smallskip

By abuse of notations we still call $H$ the composed Hamiltonian $H\circ \Psi^{(1)}\circ \Phi_\xi$.
\begin{definition}\label{puzza}
We define the {\em normal form} $\mathcal N$ which collects all the terms of $H_{Birk}$ of degree $\leq 2$ (dropping the constant terms). We then set $P= H-\mathcal N$.
\end{definition} Notice that the Hamiltonian $H_{Birk}$ is different from the corresponding one in \cite{PP} (in that paper we performed a full normal form transformation), however the resulting normal form $\mathcal N$ is the same since it collects only terms of degree less or equal to two in the variables $z= u_2$.

 \section{Functional setting}

Following \cite{Po} we study {\em regular} functions $F:\ro \mathfrak K\times D_{a,p}(s,r)\to \Co$, that is whose Hamiltonian vector field  $X_F(\cdot;\xi)$ is M-analytic from $D(s,r)\to \Co^n\times\Co^n\times\ell^{a,p}_S$. In the variables $\xi$ we require Lipschitz regularity.  Let us recall the  definitions of M-analytic and majorant norm and their properties proved in  \cite{BBP}.\smallskip

Let us consider the space  \begin{equation}\label{E}
V := \Co^n \times \Co^n \times  \ell^{a, p}_{S}
\end{equation}
with  $(s,r)$-weighted norm
\begin{equation}\label{normaEsr}
v =  (x,y,z,\bar z) \in V \, , \quad
\|v\|_V := \|v\|_{s,r}= \|v\|_{V,s,r}= \frac{|x|_\infty}{s} + \frac{|y|_1}{r^2}
 +\frac{\|z\|_{a,p}}{r}+\frac{\|\bar z\|_{a,p}}{r}
\end{equation}
where $ 0 < s< 1 $, $0 < r <c_1\e$  and
$ |x|_\infty := \max_{h =1, \ldots, n} |x_h| $,
$ |y|_1 := \sum_{h=1}^n |y_h| $.

For a vector field, i.e. a map $X:  D(s,r)\to V$, described by the formal Taylor expansion:
$$ X = \sum_{\nu,i,\a,\b} X_{\nu,i,\a,\b}^{(\vgot)} e^{\ii (\nu, x)} y^i z^\a \bar z^\b \partial_{\mathtt v}\,, \quad \mathtt v= x,y,z,\bar z$$   we define the {\em majorant} and its {\em norm}:
$$ 
MX := \sum_{\nu,i,\a,\b} |X_{\nu,i,\a,\b}^{(\vgot)}| e^{s|\nu|} y^i z^\a \bar z^\b \partial_{\mathtt v}\,, \quad \mathtt v= x,y,z,\bar z
$$
\begin{eqnarray}\label{normadueA}
|| X ||_{s,r} & := &
\sup_{(y,z, \bar z) \in D(s,r)} \| M X \|_{V} \,.  \end{eqnarray}
The different weights ensure that, if $\Vert X_F\Vert_{s,r}<\frac12$,  then $F$ generates a close--to--identity symplectic change of variables from $D(s/2,r/2)\to D(s,r)$, Proposition \ref{mpomn}.
\begin{remark}
The notion of $M$--analytic can be given in general for any map between separable Hilbert spaces with prescribed bases. It means that the {\em map} given in coordinates by the corresponding majorant functions is in fact analytic.  It is then easy to see that  composition of $M$--analytic maps is $M$--analytic with the corresponding estimate on norms.\end{remark}

In our algorithm we deal with functions which depend in a Lipschitz way on some parameters $\xi$ in a compact set $ \mathcal O \subseteq \ro \mathfrak K$ (Formula \eqref{Aro}). To handle this dependence we introduce weighted Lipschitz norms   for a map $X:  \O \times D(s,r)\to V$ setting: 
$$ \|X \|^{lip}_{s,r,\mathcal O}:=\sup_{\xi\neq \eta \in \mathcal O\,,\;(x,y,w)\in D(s,r)}\frac{\|X(\eta)-X(\xi)\|_{s,r}}{|\eta-\xi|},  $$
\begin{equation}
\label{weno} \Vert X\Vert_{s,r,\mathcal O}=\Vert X\Vert_{s,r} := \sup_{\mathcal O \times D(s,r)} \| M X \|_{V}\,,  \Vert X\Vert^\lambda_{s,r}= \|X\|_{s,r,\mathcal O}+\lambda  \|X_f\|^{lip}_{s,r,\mathcal O}
\end{equation}  where $\lambda $ is a parameter proportional to $|\mathcal O|$.  
Correspondingly for a parameter dependent sequence $f=\{f_m(\xi)\}_{m\in I}\,,$ here $I$ is any index set, we define:
\begin{equation}\label{seque}
 |f|_\infty:= \sup_{\xi\in \mathcal O} \sup_{m\in I}|f_m(\xi)|\,,
\quad |f|_\infty^{lip}:=  \sup_{\xi\neq \eta \in \mathcal O}\sup_{m\in I}\frac{|f_m(\xi)-f_m(\eta)|}{|\eta-\xi|_\infty}\,,
\end{equation} 
\begin{definition}
We define by $\mathcal H_{s,r,\mathcal O}=\mathcal H_{s,r}$ the space of regular analytic Hamiltonians depending on a parameter $\xi\in\mathcal O$ 
with 
the norm\footnote{in fact Hamiltonians should be considered up to scalar summands and then this is actually a norm}
\begin{equation}\label{normasr}
\| F\|^\lambda_{s,r} := \Vert X_F\Vert^\lambda_{s,r}<\infty.
\end{equation} 
\end{definition}
We denote by  $\mathbb I=\Z^n\times\N^n\times \N^{S^c}\times \N^{S^c}$  the indexing set of the monomials, that is  $k,i,\a,\b$ is associated to $e^{\ii(k,x)}y^i z^\a \bar z^\b$. For all $I\subset \mathbb{I} $ we define the projection $\Pi_I$ as the linear operator which acts as the identity on the monomials associated to $I$ and zero otherwise. In particular we define $ \Pi_{|k|< K}$ to be the projection relative to the set $I$ of  $k,i,\a,\b$ with $|k|<K$, same for $ \Pi_{|k|\geq K}$, similarly we define $\Pi^{(\ell)}$ as the projection on the $k,i,\a,\b$ with $2i +|\a|+|\b|= \ell$, same for $\Pi^{(\geq \ell)}$ and $\Pi^{(\leq\ell)}$.

The main properties of the majorant norm are contained in the following statements, proved in \cite{BBP1}, Lemma 2.10,\ 2.15,\  2.17.
 \begin{proposition} \label{mpomn}
Let $ H, K \in {\mathcal H}_{s,r} $. Then, for all $ r/2 \leq r' < r $,
$ s/2 \leq s' < s $, $\l'\leq \l$: 
\be\label{inclu} \norma X_H \norma_{s',r'}^{\l'} \leq 4  \norma X_H \norma_{s,r}^\l\,,
\ee
 \be\label{commXHK} \norma X_{\{H,K\}}\norma_{s',r'}^\l
=  \norma \,[ X_H, X_K]\, \norma_{s',r'}^\l \leq 2^{2n+3} \delta^{-1}
\norma X_H \norma_{s,r}^\l \norma X_K\norma_{s,r}^\l \end{equation} where $\d$ is
defined by:
\begin{equation}\label{diffusivumsui}
  \d :=  \min\Big\{ 1- \frac{s'}{s},   1-  \frac{r'}{r}  \Big\}\,.
\end{equation}

Let $ r/2 \leq r' < r $, $ s/2 \leq s' < s $, and
$ F \in \mathcal{H}_{s,r}$ with
\be\label{defeta}
\norma X_F \norma_{s,r}^\l < \d/ (2^{2n + 6 }e)
\end{equation} 
with $\d$ defined in \eqref{diffusivumsui}.
Then  the time $ 1$-Hamiltonian flow
$$
 \Phi^1_F := e^{{\rm ad}(F)}  : D(s', r') \to D(s,r)
 $$
is well defined, analytic, symplectic, and,
$ \forall H \in \mathcal{H}_{s,r} $, we have $H\circ \Phi^1_F  \in
\mathcal{H}_{s',r'}  $ and
\begin{equation}\label{iluvatar}
\norma  X_{H\circ \Phi^1_F}\norma_{s', r'}^\l\leq 2 \norma
X_H\norma_{s,r}^\l \,,\quad \norma  X_{H\circ \Phi^1_F}-X_H\norma_{s', r'}^\l\leq2 \norma  X_F \norma_{s,r}^\l\norma
X_H\norma_{s,r}^\l \,.
\end{equation}
For all $I\subset \mathbb I$  and,
$ \forall H \in \mathcal{H}_{s,r} $, we have
\begin{equation}\label{caligola}
 \norma \Pi_I X_H \norma^\l_{s,r} \leq \norma X_H \norma^\l_{s,r} \,.
\end{equation}
In particular we have the {\em smoothing estimates}: $ s' < s $,
\begin{equation}\label{smoothl}
\norma\Pi_{|k| \geq K}X_H \norma_{s',r}^\l\leq \frac{s}{s'} \, e^{-K(s-s')}\norma X_H
\norma_{s,r}^\l \, ,
\end{equation}
and the {\em degree estimates} 
\begin{equation}\label{degrl}
\norma X_{\Pi^{( l \geq d)}H} \norma_{s,r'}^\l\leq (\frac{r'}{r})^{d-2}\norma X_H
\norma_{s,r}^\l \, .
\end{equation}
\end{proposition}
\begin{remark}\label{diaH} 
For a diagonal quadratic Hamiltonian $F=\sum_m\val_m(\xi)z_m\bar z_m$  we have
$$ X_F=\\i(\sum_m\val_m( \xi)z_m\pd{}{ z_m}-\sum_m\val_m( \xi)\bar z_m \pd{}{\bar z_m})$$$$ MX_F= \sum_m|\val_m( \xi)|(z_m\pd{}{ z_m}+\bar z_m \pd{}{ \bar z_m}) ,\quad  \| MX_F\|_{s,r}^\l =|\val|_\infty+\lambda|\val|_\infty^{lip}.$$
\end{remark}
\begin{lemma}
For $c_1 \e >r >\e^3$, the perturbation $P$ of Definition \ref{puzza} is in $\mathcal H_{s,r}$  and satisfies the bounds
\begin{equation}
\label{bonls}\norma X_P\norma_{s,r}^\l \leq  C(\e r + \e^5 r^{-1})\,,   
\end{equation}  where $C$ does not depend on $r$ and depends on $\e,\lambda$ only through $\lambda/\e^2$.
\end{lemma}
\begin{proof}
This is item iv) of Theorem 1 of \cite{PP}. The fact that we are using the majorant norm only changes the constant and not the order of magnitude. 
\end{proof}
\section{The normal form}
We will work with many quadratic Hamiltonians in the variables $w$ (thought as a row vector). We  represent a quadratic form $\mathcal F$ by a matrix $F$ as
 \begin{equation}\label{represQ}
\mathcal F(w)= \frac 12 (w, wJF^t)=-\frac1 2  w F Jw^t\,,
\end{equation} where $J:=-\ii \{w^t,w\}$ is the standard matrix of the symplectic form  
which expresses the action by Poisson bracket.

\smallskip

By explicit computation, and under simple generiticity conditions, the normal form $\mathcal N$ of Definition \ref{puzza} is as follows:
\begin{equation}
\label{Pno}(\ome(\xi),y)+\sum_{k\in S^c} |k|^2 |z_k|^2 +{\mathcal Q}(\xi;x,w) \,,\quad \ome_i(\xi)= |\mathtt j_i|^2 -2\xi_i
\end{equation} 
 here    ${\mathcal Q}(\xi;x,w)$ is a quadratic Hamiltonian in the variables $w$ with coefficients trigonometric polynomials in $x$ given by Formula ($30$) of \cite{PP}:
  \begin{equation}
 \mathcal Q  (\xi,w)= 4\sum^*_{  1\leq i\neq j\leq m    \atop h,  k \in S^c}\sqrt{\xi_{i}\xi_{j}}e^{\ii  (x_{i}-x_{j})}z_{h}\bar z_{k } +
\end{equation} $$ + 2\sum^{**}_{ 1\leq i< j\leq m   \atop h,  k \in S^c }\sqrt{\xi_{i}\xi_{j}}e^{-\ii  (x_{i}+x_{j})}z_{h} z_{k } +
  2\sum^{**}_{ 1 \leq i<j\leq m    \atop h,  k \in S^c }\sqrt{\xi_{i}\xi_{j}}e^{\ii  (x_{i}+x_{j})}\bar z_{h}\bar  z_{k }.  $$  Here $\sum^*$ denotes that $  (h,  k,  v_i,  v_j)$ satisfy:
 $$ \{  (h,  k,  v_i,  v_j)\,  |\,    {h+v_i= k+v_j},  \  { |h|^2+|v_i|^2=| k|^2+|v_j|^2}\}. $$
  and $\sum^{**}$,   that  $  (h,  v_i,   k,  v_j)$ satisfy:
   $$ \{  (h,  v_i,  k,  v_j)\,  |\,    {h+k= v_i+v_j},  \  { |h|^2+|k|^2=| v_i|^2+|v_j|^2}\}. $$
Notice that in the sums  $  \sum^{**}$ each term appears twice.  \vskip10pt

This is   a very complicated infinite dimensional quadratic
Hamiltonian, by applying the results of \cite{PP}, we decompose this infinite dimensional
system into infinitely many decoupled finite dimensional systems corresponding to the connected components of a graph
(which is recalled in \S  \ref{gra}).      One of the main results of \cite{PP} is the construction of an explicit symplectic change of variables which reduces $\mathcal N$ to constant coefficients.
 
 Since this construction is needed in the following we recall   quickly   Theorem   \ref{teo1}   of \cite{PP} adapted to the case of the cubic NLS. In the cubic case   we also apply the more precise results of \cite{PP1}.
 \begin{theorem}\label{teo1}For all {\em generic} choices $S=\{\mathtt j_1,\dots,\mathtt j_n\}\in\nobreak  \Z^{nd}$ of the tangential sites,    there exists   a map 
$$S^c\ni k\to L(k)\in \Z^n \,,\quad |L(k)|\leq d+1$$  such that the analytic symplectic change of variables:
$$z_k= e^{-\ii (L(k),x)}z_k' ,\ y=y'+\sum_{k\in S^c}  L(k)  |z_k'|^2,\ x=x'. $$
$$\Psi: ( y', x)\times (z',\bar z') \to ( y, x)\times (z,\bar z) $$   from $ D(s,r/2) \to D(s,r)$ 
has the property that $\mathcal N$ in the new variables  has constant coefficients, namely:
\begin{equation}
\label{Sno}\mathcal N\circ\Psi= (\ome(\xi),y') +\sum_{k\in S^c}\tilde\Ome_k|z'_k|^2 +\tilde {\mathcal Q}(w')\,,
\end{equation}
 where $\omega(\xi)$ is defined in \eqref{Pno} and furthermore:
\smallskip

\noindent i) {\bf Asymptotic of the normal frequencies:} We have $\tilde\Ome_k= |k|^2 +\sum_i |\mathtt j_i|^2 L^{(i)}(k)$.

\noindent ii) {\bf Reducibility}: The matrix $\tilde{ Q}(\xi)$ which represents the quadratic form	$\tilde{\mathcal Q}(\xi,w')$ (see formula \eqref{represQ})   depends  only on the variables $\xi$  
and all its entries are homogeneous of degree one in these variables.  It is   block--diagonal  with blocks of dimension  $\leq d+1$ and satisfies  the following properties: 

\quad   All of the blocks except a finite number  are self adjoint.
 
\quad All the (infinitely many) blocks are  chosen from a finite list of matrices $\mathcal M(\xi)$.

\noindent iii)   {\bf Smallness:} \ If $\e^3<r<c_1\e$,  the perturbation $\tilde P:= P\circ \Psi $ is  small, more precisely 
we have   the bounds:
\begin{equation}\label{pertu}
\Vert X_{\tilde P}\Vert^\lambda_{s,r}\leq C (\e r + \e^{5} r^{-1}) \,, 
\end{equation}  where $C$ is independent of $r$ and depends on $\e,\lambda$ only through $\lambda/\e^2$. 
\end{theorem}
The smallness condition implies that, if $r$ is of the order of $\e^2$ and $\lambda/\e^2$ is of order one,  then $\Vert X_{\tilde P}\Vert^\lambda_{s,r}$ is of order $\e^3$. As we shall see this is exactly a type of smallness required in order to insure the success of the KAM algorithm (cf. Theorem \ref{gacom}).\smallskip

{\bf Warning}\quad In $\Z^n$ we always use as norm $|l|$ the $L^1$ norm $\sum_{i=1}^n|l^{(i)}|$. On the other hand  in $\Z^d$, and hence in $S^c$, we use the euclidean $L^2$ norm.

\subsection{The geometric graph $\Gamma_S$\label{gra}} It is important to recall that the term  $\tilde{\mathcal Q}(\xi,w')$  comes from the sum of two contributions,  the term  $ {\mathcal Q}(\xi,x,w )$, in the new variables and the contribution $-2\sum_{k\in S^c}\xi\cdot L(k)|z'_k|^2$ (coming from the $y$ variables).  Hence $\tilde {\mathcal Q}(w')=$
\begin{equation}\label{tildeQ}
-2\sum_{k\in S^c}\xi\cdot L(k)|z'_k|^2+ 4\sum^*_{  1\leq i\neq j\leq m    \atop h,  k \in S^c}\sqrt{\xi_{i}\xi_{j}} z'_{h}\bar z'_{k } + 2\sum^{**}_{ 1\leq i< j\leq m   \atop h,  k \in S^c }\sqrt{\xi_{i}\xi_{j}} z'_{h} z'_{k } +
  2\sum^{**}_{ 1 \leq i<j\leq m    \atop h,  k \in S^c }\sqrt{\xi_{i}\xi_{j}} \bar z'_{h}\bar  z'_{k }.
\end{equation}   In its matrix description the two terms will give the off diagonal and the diagonal terms respectively.

The off diagonal terms are described through a simple geometric construction (which gives a complicated combinatorics). Given two distinct elements $\mathtt j_i,\mathtt j_j\in S$  construct the sphere $S_{i,j}$  having  the two vectors as opposite points of a diameter  and the two Hyperplanes, $H_{i,j},\ H_{j,i}$,  passing through $\mathtt j_i$  and $\mathtt j_j$ respectively, and perpendicular to the line though the two vectors  $\mathtt j_i,\mathtt j_j.$ 
\smallskip 

From this configuration of spheres and pairs of parallel hyperplanes  we deduce a {\em geometric colored graph}, denoted by $\Gamma_S$, with vertices  the points in $S^c$ and two types of edges, which we call {\em black} and {\em red}.

\begin{itemize}\item A black edge connects  two points $p\in H_{i,j},\ q\in H_{j,i}$, such that the line $p,q$ is orthogonal to the two hyperplanes, or in other words $q=p+\mathtt j_j-\mathtt j_i$.

\item A red edge connects  two points $p,q\in S_{i,j} $ which are opposite points of a diameter ($p+q=\mathtt j_i+\mathtt j_j$).

\end{itemize}
 \begin{figure}[!ht]
\centering
\begin{minipage}[b]{11cm}
\centering
{
\psfrag{a}{$\mathtt j_j$}
\psfrag{b}{$\mathtt j_i$}
\psfrag{c}{$ a_2$}
\psfrag{d}{$ b_2$}
\psfrag{e}{$ a_1$}
\psfrag{f}{$ b_1$}
\psfrag{H}{$H_{i,j}$}
\psfrag{S}{$S_{i,j}$}
\psfrag{m}{$ \mathtt j_j-\mathtt j_i$}
\psfrag{l}{$ \mathtt j_j+\mathtt j_i$}
\includegraphics[width=11cm]{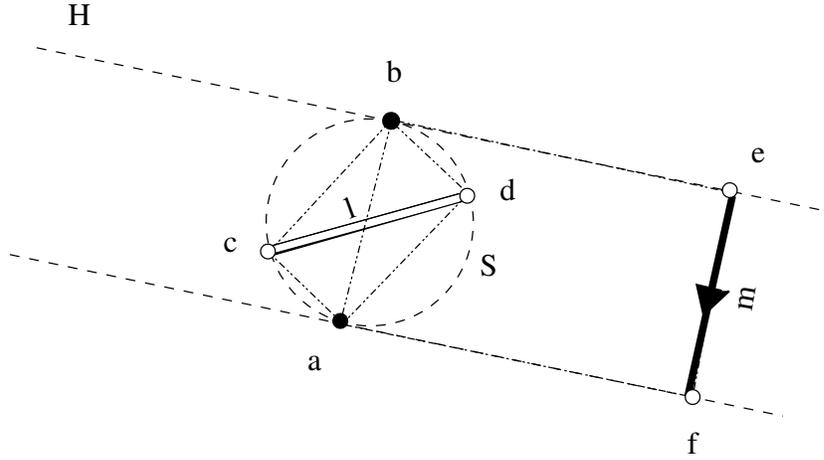}
}
\caption{\footnotesize{the plane $H_{i,j}$ and the sphere $S_{i,j}$. The points $a_1,b_1,\mathtt j_j,\mathtt j_i$ form  the vertices of a rectangle. Same for  the points $a_2 ,\mathtt j_j,b_2,\mathtt j_i$}}\label{fig1}
\end{minipage}
\end{figure}

The condition  for two points $p,q$ to be the vertices of an edge  is   given by algebraic equations.  Visibly $p\in H_{i,j}$  means that $(p-v_i,\mathtt j_i-\mathtt j_j)=0$, the corresponding $q=p+\mathtt j_j-\mathtt j_i$,  while $p\in S_{i,j}$  is given by $(p-\mathtt j_i,p-\mathtt j_j)=0$  and the corresponding opposite point $q$  is  given by $p+q=\mathtt j_i+\mathtt j_j$.

We thus have two types of constraints  describing when two points are joined by an edge, a linear  $q-p=\mathtt j_j-\mathtt j_i$ or $p+q=\mathtt j_i+\mathtt j_j$ and a quadratic constraint $(p-\mathtt j_i,\mathtt j_i-\mathtt j_j)=0$ or $(p-\mathtt j_i,p-\mathtt j_j)=0$. 
Given a connected component $A$ of the graph we can choose one vertex $x\in S^c$ and use the linear constraints in order to write all the equations which define $A$ by linear or quadratic equations on $x$.  We keep track of the linear constraints by  {\em marking }  the edges by $\mathtt j_j-\mathtt j_i$ for black edges  and $\mathtt j_j+\mathtt j_i$ for red ones.

  Now  each connected  component $A$  has  a  purely combinatorial description  which encodes  the information on the  edges which connect the vertices of $A$. We obtain an {\em abstract} graph with two types of edges (black, red) marked with pairs $i,j\in[1,\ldots,n]$. 
  
    A connected component of the geometric graph is a solution of a system of equations (associated to the graph) having the vertices as unknowns. It is easily seen that these equations may be all expressed on a single vertex (which we call the root $\er$), more precisely, as seen in \cite{PP}  we obtain one  equation (with unknown $\er$)  for each vertex $v\neq \er$. A combinatorial graph of this type is {\em admissible} if its equations admit a solution $\er\in \R^d$ for generic values of the tangential sites.\footnote{we are interested only in solutions in $S^c$ but it is more convenient to extend  the possible solutions to $\R^d$.}
  
   In \cite{PP} we have seen that such graphs have at most $2d$ vertices hence we have a finite list of combinatorial graphs  (which we have described explicitly in terms of a Cayley graph, since we do not need it here we do not recall it).  In \cite{PP1} we have strengthened this estimate, shown that for a generic choice of $S$ the vertices of the  geometric graph, corresponding to an admissible combinatorial graph,   are {\em affinely independent} and hence at most $d+1$. This stronger estimate is necessary for the proof of the second Melnikov condition.
  
  \smallskip
  
  We denote by $\GA$ the combinatorial graph associated to $A$, note that $\GA$ encodes the information on the equations which the vertices of $A$ must solve so naturally there may be many $A$ wich have the same $\GA$.

\begin{example} \label{ungra}  $$  
    \xymatrix{ &\er-\mathtt j_1+\mathtt j_3\ar@{<- }[d] _{2,1}& &\\  &\er-\mathtt j_2+\mathtt j_3& &\\  \er\ar@{ ->}[ru] _{3,2}\ar@{<-}[ruu] ^{3,1}\ar@{=}[rr]_{1,2} && -\er+\mathtt j_1+\mathtt j_2 \ar@{=}[luu]_{2,3} \ar@{=}[lu] ^{1,3} &   }    \xymatrix{ &\er-\mathtt j_1+\mathtt j_2+\mathtt j_4+\mathtt j_3 & &\\  &\er-\mathtt j_2+\mathtt j_3& &\\  \er\ar@{ ->}[ru] _{3,2}\ar@{=}[rr]_{1,2} &&-\er+\mathtt j_1+\mathtt j_2 \ar@{=}[luu]_{4,3} \ar@{=}[lu] ^{1,3} &   }    
 $$  
 the equations that $\er$ has to satisfy  are:
$$\begin{matrix}
(\er,\mathtt j_2-\mathtt j_3)=|\mathtt j_2|^2-(\mathtt j_2,\mathtt j_3)&&&(\er,\mathtt j_2-\mathtt j_3)=|\mathtt j_2|^2-(\mathtt j_2,\mathtt j_3)\\
|\er|^2-(\er,\mathtt j_1+\mathtt j_2)=-(\mathtt j_1,\mathtt j_2)&&&|\er|^2+(\er,\mathtt j_1+\mathtt j_2)=-(\mathtt j_1,\mathtt j_2)\\
(\er,\mathtt j_1-\mathtt j_3)=|\mathtt j_1|^2-(\mathtt j_2,\mathtt j_3)&&&(\er,\mathtt j_1-\mathtt j_2-\mathtt j_3-\mathtt j_4)=
-|\mathtt j_1|^2 +( \mathtt j_1,\mathtt j_2) +( \mathtt j_1,\mathtt j_3)  \\&&&-
 ( \mathtt j_2,\mathtt j_3) +( \mathtt j_1,\mathtt j_4) - 
(  \mathtt j_2,\mathtt j_4) - (\mathtt j_3,\mathtt j_4)
\end{matrix} $$\end{example}
In case the graph has no red edges  the equations for the vertex $x$ are all linear.
This implies that the  connected components of  $\Gamma_S$  which correspond to a given combinatorial graph with a chosen vertex      are all obtained from a single one by translations by vectors which are orthogonal to the edges of the graph.

 By convention  we also have chosen   a preferred  vertex, called the {\em root}, in each connected component,  in such a way that  the  roots of the translates are the translates of this root.  
 
 Formalizing, \begin{itemize}\item we have a map $\er:S^c\to S^c$ with image the chosen set  $S^{c,\er}$ of roots. \item  The fibers of this map are the connected components  of the graph $\Gamma_S$.  
\item When we walk from the root $\er(k)$ to $k$ (inside the corresponding connected component) we  count the parity $\pm 1$ of the number of red edges on the path, this is independent of the path and we denote by $\sigma(k)$ (the {\em color} of $k$).  
\item There are only finitely many elements $k$  with $\s(k)=-1$, the finitely many corresponding roots are exactly the roots of the components with red edges. \item  In any case the color of the root is always 1 ({\em black}).

\end{itemize}\smallskip

At this point we can explain how to construct the  elements $L(k)$ which tell us how to go from the root, of the component $A$ of the  graph  $\Gamma_S$ to which $k$ belongs, to $k$.
The  equations defining  the component $A$ imply  that 

\begin{equation}\label{defL}k+\sum_iL_i(k)\mathtt j_i=\sigma(k)\er(k)\,,\quad 
|k|^2+\sum_iL_i(k)|\mathtt j_i|^2=\sigma(k)|\er(k)|^2\,, \quad \sigma(k)=1+\sum_iL_i(k).  
\end{equation}
Note that the first of the equations \eqref{defL} defines the $L(k)$, which {\em depend only on the combinatorial graph}. The fact that this definition is well posed even if $A$ is not a tree is a consequence of our genericity conditions.

The main fact is that 
\begin{proposition}\label{laC}
The Hamiltonian $ {\mathcal Q}(\xi, x,w')$ in  the new coordinates $z'$ is the sum $\sum_\ell {\mathcal Q}_\ell(\xi,w')$ over all edges $\ell$ of the geometric graph of the  following elements \begin{itemize}
\item $ {\mathcal Q}_\ell(\xi,w'):=4\sqrt{\xi_i\xi_j}(z'_h\bar z'_k+z'_k\bar z'_h)$ if $h,k$ are joined by a black edge $\ell$ marked $i,j$ 
\item $ {\mathcal Q}_\ell(\xi,w'):= 4\sqrt{\xi_i\xi_j}(z'_hz'_k +\bar z'_h\bar z'_k)$  if $h,k$ are joined by a red edge $\ell$ marked $i,j$. 
\end{itemize}
\end{proposition} Form the previous remarks there are only finitely many elements of the second type.\subsubsection{The matrix blocks of $\tilde{\mathcal Q}$ and $ad(\mathcal N)$}  According to Proposition \ref{laC}, 
the graph has been constructed in such a way that
we can group  $\tilde{\mathcal Q}=\sum_A \tilde{\mathcal Q}_A$ (cf. \eqref{tildeQ}) where the sum runs over all blocks $A\in \Gama$ and, if $E(A)$ denotes the set of edges in $A$: $$\tilde{\mathcal Q}_A:=\sum_{k\in A}-2\xi\cdot L(k)|z'_k|^2+ \sum_{\ell\in E(A)} {\mathcal Q}_\ell(\xi,w')$$ is a quadratic Hamiltonian in the variables $w'_A=z'_k,\bar z'_k$ with $k$  running over the vertices of $A$. The matrix of $\tilde Q_A$ has a natural block diagonal structure in two conjugated blocks, corresponding to two Lagrangian subspaces in the symplectic space generated by the variables $z'_k,\bar z'_k, \ k\in A$ appearing in  it. We can thus divide $w'_A$  into two conjugate components $w'_A=(u',\bar u')$ where $u'_k= (z')^{\sigma(k)}_k$ then $-\frac\ii 2\tilde Q_A $ has as matrix denoted by  $ C_{A}\oplus -C_{A}$. By convention in the first block the root $\er$ corresponds to $z'_\er$.
\smallskip

Given two vertices $u'_h,u'_k$ $h\neq k\in A$  
 we have that the matrix element $c_{u'_h,u'_k}$ of $C_A$ is non zero if and only if $h,k$ are joined by an edge  (marked say $(i,j)$) and then 
 \begin{equation}
\label{Lema}c_{u'_h,u'_k}= 2\sigma(k)\sqrt{\xi_i\xi_j},\qquad c_{u'_k,u'_k}= - \sigma(k)(\xi,L(k)). 
\end{equation}  

By definition $L(k)$ depends only on the combinatorial graph $\GA$ of which $A$ is a realization, therefore the matrix $C_A=C_\GA$ depends only on the combinatorial block $\GA$.
\begin{remark}\label{stiL}
One may choose the root of each combinatorial graph so that any other vertex is connected by a path with at most  $[(d+1)/2]$  edges. One deduces the estimates $\sum_i|L_i(k)|\leq d+1,\, $ for all $k$.
\end{remark}\subsubsection{The space $F^{0,1}$}

In the   KAM algorithm we shall need to study in particular the  action by Poisson bracket of $ \mathcal N $ on a special space of functions  called $F^{0,1}$, so we recall some of this formalism.
\begin{definition}
We set $F^{0,1}$ to be the space of functions spanned by the basis elements
  $$e^{\ii \s \s (k)\nu\cdot x}{z'_k}\,^\sigma=e^{\ii \s ([\s (k)\nu+L(k)] \cdot x)}{z_k}^\sigma$$   which preserve mass and momentum. \footnote{we deviate from the notations of \cite{PP} and in  $F^{0,1}$ we also impose zero mass}\end{definition} 
One easily sees that $F^{0,1}$ is a symplectic space under Poisson bracket. The formulas for mass and momentum in the new variables are
  \begin{equation}\label{posta2}
  \{ \mathbb L, e^{\ii \s \s (k)\nu\cdot x}{z'_k}\,^\sigma\}= \ii \sigma\s (k)(\sum_i \nu_i+1)\,e^{\ii \sigma \s (k)\nu\cdot x}{z'_k}\,^\sigma,\quad 
  \end{equation}
  $$
   \{ \mathbb M, e^{\ii \sigma \sigma(k)\nu\cdot x}{z'_k}\,^\sigma\}= \ii\s \s (k) (\sum_i \nu_i \mathtt j_i+ \er(k))\,e^{\ii\s \sigma(k) \nu\cdot x}{z'_k}\,^\sigma,
   $$
hence the conservation laws  tell us that for an element $e^{\ii \s \s (k)\nu\cdot x}{z'_k}\,^\sigma\in F^{0,1}$     the vector $\nu\in\Z^d$   is constrained  by the fact that $-\sum_i \nu_i \mathtt j_i$ must be in the set of roots in $S^c$ and moreover the mass constraint $\sum_i \nu_i=-1$.

For each connected components $A$ of the graph $\Gamma_S$ with some root $\er$  any solution $\nu$ of $\sum_i \nu_i \mathtt j_i+ \er=0$ determines in the space $F^{0,1}$ a block  denoted $A,\nu$ with basis the elements $e^{\ii\s \sigma(k) \nu\cdot x}{z'_k}\,^\sigma$ with ${z'_k}\,^\sigma$  the corresponding basis of the two Lagrangian blocks corresponding to $A$.
 
From the previous formulas we have thus  that this space decomposes again into blocks indexed by   pairs $A,\nu$ with $A$ a connected component   of the graph $\Gamma_S$ and $\nu$ any solution of $\sum_i \nu_i \mathtt j_i+ \er=0$ where the mass of $\nu$ is $-1$, 
each such block is a symplectic space decomposed  into a  pair  of $-\frac{\ii}{2} {\mathcal N} $ stable Lagrangian subspaces.\medskip

 Given thus such a pair of a component $A$ and a  frequency $\nu$,  notice that in fact $A$ is determined by its root which is determined by $\nu$  by the conservation law.  We have to understand the action  of $-\frac\ii 2 ad(\mathcal N)$, on the block of $F^{0,1}$  with basis the elements  $e^{\ii \sum_j\nu_jx_j}{z'_{m}}^{\s(m)}$ with $ \er(m)=-\sum_i\nu_i\mathtt j_i$, (on its  conjugate $\bar A$ it is the minus transpose). The action of $\tilde {\mathcal Q}$ does not depend on $\nu$ and as before it is only through $\tilde {\mathcal Q}_A$ and gives the matrix $C_A$,  we need then to understand  the elements $(\ome(\xi),y') +\sum_{k\in S^c}\tilde\Ome_k|z'_k|^2$.   By Formula \eqref{defL}, the term $\sum_{k\in S^c}\tilde\Ome_k|z'_k|^2$ contributes on the first block the scalar $ |\er(m)|^2 $.  As of $(\ome(\xi),y')$  it also contributes by a scalar, this time $ \sum_i\nu_i |\mathtt j_i|^2  -2\sum_i\nu_i\xi_i.$ Summarizing
 \begin{proposition}\label{madiN}
The matrix of $-\frac\ii 2 ad(\mathcal N)$ on the block $A,\nu$ is the sum of   the matrix $C_\GA$ plus the scalar matrix $[\frac12(|\er(m)|^2+\sum_i\nu_i |\mathtt j_i|^2 )-\sum_i\nu_i\xi_i]\, I_A.$
\end{proposition}
   \subsubsection{ The  standard form\label{stanfor}}
 By the rules of Poisson bracket we have on the real space spanned by $z,\bar z$ that $\{a,\bar b\} =-\{\bar  a,b\}= \{b,\bar  a \}$ is imaginary so 
 $\{\bar a,\bar b\} =-\{  a,b\}=\overline{\{  a,b\}}$
and \begin{definition}
 $(a,b):=\ii \{a,\bar b\}$ is a real symmetric form, called the {\em standard form}.
\end{definition}
 For the variables  we have $(z_h,z_h)=1,\ (\bar  z_h,\bar  z_h)=-1$, so the form is positive definite on the space spanned by the $z$, negative on the space spanned by the $\bar z$  and of course indefinite if we mix the two types of variables.  Thus we may say that an element $a$ in the real space spanned by $z,\bar z$ is {\em of type $z$ (resp. $\bar z$)} if  $(a,a)=1$ resp.   $(a,a)=-1$. Now choose any quadratic real Hamiltonian $\mathcal H=\bar{\mathcal H}$. We have  $\{\mathcal H,\{a,\bar b\}\}=0$  by the Jacobi identity, moreover the map $x\mapsto \ii\{\mathcal H,x\}$ preserves the real subspace spanned by $z,\bar z$ hence we have \begin{equation}
\label{resi}(a, \ii\{\mathcal H,  b\})=\ii\{a,\overline{\ii\{\mathcal H,  b\} }\}= \{a,\{\mathcal H,\bar b\} \}=- \{\{\mathcal H,a\},\bar b \}=\ii \{\ii\{\mathcal H,a\},\bar b \}=  (\ii\{\mathcal H,a\},  b ).
\end{equation}Formula \eqref{resi} tells us that  the operator $\ii\{\mathcal H,-\}$ is symmetric with respect to this  form.

 \subsubsection{ The case of $-\frac\ii 2\{ \tilde{\mathcal Q},-\}$ } We apply the previous analysis to $  \mathcal H =-\frac 1 2  \tilde{\mathcal Q}$ and its block decomposition.
 When we have red edges each of the two Lagrangian blocks contains both variables  $z$ and $\bar z$  and by convention we  take as first block the one in which the variable corresponding to the root is of type $z$.  The standard form  $(a,b):=\ii\{a,\bar b\}$ is indefinite.

By assumption the operator $-\frac\ii 2\{\tilde{\mathcal Q},-\}$ can be put in normal form by a change of basis  preserving the form $(a,b)$. Thus the new basis  is formed still by elements  which we have called of types $z$  and $\bar z$.
 
From all these considerations one has:
\begin{lemma}\label{autoag}
For all combinatorial blocks $\GA$ which do not contain red edges, the matrix $C_{\GA}$ is self--adjoint for all $\xi\in \R_+^n$. If $\GA$ contains red edges then each vertex  $k$ has a sign and corresponds to an element $u'_k= (z')^{\sigma(k)}_k$. 

The diagonal matrix of signs $\sigma_\GA=$ diag$(\sigma(k))$ is the matrix of the standard form in the basis   $u'_k= (z')^{\sigma(k)}_k$ and  $C_{\GA}$ is self--adjoint with respect to the indefinite form defined by $\sigma_\GA$.
\end{lemma}
  The orthogonal group of the standard form acts on the entire symplectic block preserving the two Lagrangian subspaces and thus it has a mixed  invariant which we still call the standard form\begin{equation}
\label{stanfor1}\sum_{k \,|\, \mathtt r(k)=r} \s(k)  |z_k|^2=\sum_{k \,|\, \mathtt r(k)=r} \s(k)  |z'_k|^2.
\end{equation}
  \begin{lemma}[conservation laws]
  In the new variables the conserved quantities are:
  $$ \mathbb L= \sum_i y'_i + \sum_k \s(k)  |z'_k|^2 \,,\quad  \mathbb M = \sum_i \mathtt j_i y'_i + \sum_k \s(k) \er(k) |z'_k|^2\,,$$ $$\quad \mathbb K = \sum_i |\mathtt j_i|^2 y'_i + \sum_k \s(k) |\er(k)|^2 |z'_k|^2 $$  note that all of these three quadratic Hamiltonians are represented by a {\em scalar} matrix on each component of $w'_A$.
\end{lemma}
\begin{proof}
We substitute the new variables and use the identities \eqref{defL}.
\end{proof}
\section{Normal form reduction}
We now want to simplify   $-\frac{\ii}{2} {\mathcal N} $ on each pair of stable Lagrangian subspaces described in Proposition \ref{madiN}, using the standard Theory of canonical form of symplectic matrices.

The main ingredient we need is :
\begin{theorem}\label{te12p}[Proposition 1  of \cite{PP1}] i)\quad For all combinatorial blocks $\GA$, $C_\GA$ has distinct eigenvalues, namely it is regular semisimple  for values of the parameters $\xi_i$ outside a real hypersurface (the discriminant).

\smallskip

ii)\quad For any pair of distinct blocks $(A_1,\nu_1),(A_2,\nu_2)$  the resultant of the characteristic polynomials of the two matrices of the action of $-\frac{\ii}{2}ad(\mathcal N)$  is non--zero, hence outside this hypersurface the eigenvalues of these two blocks are distinct.
 \end{theorem}
The algebraic  hypersurface union of all discriminant varieties for all the combinatorial matrices $C_\GA$ will  be denoted by    $\mathfrak A$ and called {\em discriminant}. It is given by a homogeneous polynomial equation, thus  $(\R_+)^n\setminus \mathfrak A$ is a union of finitely many connected open cones $(\R_+)^n_1,\dots, (\R_+)^n_M $  where  the number of  real, resp. complex eigenvalues of any given combinatorial matrix $C_\GA$ is constant.  On each of these regions $(\R_+)^n_\alpha$  we can thus describe a normal form.

Since our normal form, thought of as operator has possibly also complex eigenvalues let us recall  the basic normal form of the simplest Hamiltonians.  

Consider  $H_\vartheta:=a (|z_1|^2-|z_2|^2) + b(z_1 z_2 +\bar z_1\bar z_2)$, setting $\vartheta=a+\ii b$. On the space with basis $z_1,\bar z_2,\bar z_1,   z_2$ (symplectic form $J$)  the operator $-\ii ad(H)$ has matrix.
$$ M_\vartheta=\begin{vmatrix}a& -b&0&0          \\
b& a&0&0\\
 0&0&-a&b\\
0&0&-b&-a
\end{vmatrix},\qquad J=\begin{vmatrix}0&0 &-\ii&0         \\
 0&0&0& \ii\\
\ii&0&0&0 \\
0&-\ii&0&0 
\end{vmatrix}$$ with eigenvalues $\pm\vartheta,\pm\bar\vartheta,\quad\vartheta=a+\ii b$. One easily sees that a $4\times 4$ real symplectic matrix commuting with  this matrix, when the 4 eigenvalues  $\pm\vartheta,\pm\bar\vartheta $ are distinct,  has the same block form $M_\beta$ for some complex number $\beta=c+\ii d$ and so it is represented by the Hamiltonian $H_\beta =c (|z_1|^2-|z_2|^2) + d(z_1 z_2 +\bar z_1\bar z_2)$.\smallskip

Now we need to decompose the various combinatorial blocks that we previously described. We have already defined the discriminant  hypersurface $\mathfrak A$. It is now convenient to choose the compact domain $\mathfrak K$ of Formula \eqref{Aro} to be a union  of compact domains \begin{equation}\label{ikappaa}
\mathfrak K=\cup_\alpha\mathfrak K_\alpha
\end{equation} each contained in the corresponding open connected component $(\R_+)^n_\alpha$  of $(\R_+)^n\setminus \mathfrak A$.
For each combinatorial matrix $C_\GA$  the standard form  $\sigma_\GA$ on $(\R_+)^n_\alpha$ has constant signature   and we have:

\begin{proposition}[cf. Williamson \cite{Wil}]
On each region $(\R_+)^n_\alpha$ the eigenvalues of  the $C_\GA$ are analytic functions 
of $\xi$, say $\val_1,\dots,\val_{dim(\mathcal A)}$.  

 For all $\xi \in (\R_+)^n_\alpha$ there exists a linear   symplectic change of coordinates  $ u' \to U_\GA(\xi) u'= u''$ such that:

1. $U_\GA(\xi)$ is orthogonal with respect to $\sigma_\GA$

2. $U_\GA(\xi)$ is analytic in $\xi$. 

3. $U_\GA(\xi)$  conjugates  $C_\GA$  into the following normal form:

For  each real eigenvalue $\val$ , $C_\GA$ acts as $\val I$  on the (one dimensional) eigenspace of $\val$ in $u_A$.  

For each  pair of conjugate  complex eigenvalues $\val_\pm= a\pm \ii b$, we  have a real two dimensional space such that the two complex eigenvectors lie in its complexification. Then we have a  basis of this subspace such that $C_\GA$ restricted to this subspace is a $2\times 2$   matrix $$ \begin{pmatrix}  a & -b \\ b& a\end{pmatrix}. $$
The matrix $\sigma_\GA$ of the standard form  on this basis is diag$(1,-1)$ so one of the variable is a $z$ and the other a $\bar z$.

\end{proposition}

  \begin{remark}\label{normU}
We note that the matrices $C_\GA$ have entries wich are homogeneous of degree one in $\xi$. Therefore given any compact domain $O$  which does not intersect $\mathfrak A$ the entries of the matrices $U_\GA(\xi), U_\GA(\xi)^{-1}, \partial_\xi U_\GA(\xi)$ can be uniformly bounded in $\rho \, O$ independently of $\rho$, in particular this applies to each $\mathfrak K_\alpha$.

Since the $U_{\GA}$ is determined by $\GA$ we denote its matrix elements by $[U_{\GA}]_{a,b}$ where $a,b$ run over the vertices of $\GA$.  Namely  on a given geometric block $A$ isomorphic to $\GA$, $[U_{\GA}]_{a,b}$ is the entry relative to the elements $z^\s_k$ associated to $a,b$ respectively. 
  \end{remark}   
 
 
Given a geometric block $A$ let $\GA$ be the corresponding combinatorial block.  In each connected component, the non-unique choice of the matrix $U$, putting in canonical form the matrix $C_{\mathcal A}$ determines a symplectic change of variables for all blocks $A$ with combinatorial block $\mathcal A$, we do this for all the finitely many $\mathcal A$. We may then index the new variables still by $S^c$ and decompose  $S^c$ in two sets: an infinite set $S^c_r$, which  indexes the real eigenvalues,  namely $u_k$ is an eigenvector of $-\frac\ii 2\mathcal Q$  of real eigenvalue $\val_k(\xi)$. Then a finite set $S^c_i$  which indexes the complex eigenvalues (note that by the special block structure of the Hamiltonian there are no purely imaginary eigenvalues). 
 By the reality condition each two by two block corresponding to a pair of conjugate eigenvalues is indexed by a pair $(h,k)$ of elements of $S^c_i$. We write the conjugate eigenvalues as $\val_{h,k},\bar \val_{h,k}$ with $\val_{h,k}= a_{h,k}+\ii b_{h,k}$.    Note that  for $(h,k)\in S^c_i$ we have $\sigma(h)=1$ and $\sigma(k)=-1$.
 By abuse of notation we still call $x,y,z_k,\bar z_k$ the new variables.  
We have finally the final {\em diagonal form of the Hamiltonian}
\begin{theorem}\label{xixi}
 i)\quad For each connected component  of $(\R_+)^n_\alpha$  on  $\ro \mathfrak K_\alpha$  we have a symplectic change of variables $\Xi$ (depending analytically on $\xi$) which puts the Hamiltonian $\mathcal N$ in the canonical {\em diagonal} form $\mathcal N= \mathbb K+2\mathbb K^1$ where
\begin{equation}
  {\mathbb K^1}= -\sum_{i=1}^n\xi_iy_i+\sum_{k\in S^c_r} \sigma(k) \val_k |z_k|^2 + \sum_{(h,k)\in S^c_i} a_{h,k} (|z_h|^2-|z_k|^2) + b_{h,k}(z_h z_k +\bar z_h\bar z_k)
\end{equation} The elements $\val_k , a_{h,k}, b_{h,k}$ are analytic functions of  $\xi$ and homogeneous of degree one..

ii) [Elliptic open set]\quad  There exists a connected component (hence a non empty open cone) of $ \R_+ ^n$ such that in this region   all the eigenvalues are real i.e.  $S^c_i$ is empty (\cite{PP1}, \ proposition 1.13).
\end{theorem}\begin{remark}
In order to simplify the notations we shall write the final variables as $z_k$ since no confusion should arise with the initial variables.
\end{remark}
We claim that for $\mathbb L, \mathbb M,\mathbb K$    we have still
   \begin{equation}\label{posta}
 \mathbb L= \sum_i y_i + \sum_k \s(k)  |z_k|^2 \,,\quad  \mathbb M = \sum_i \mathtt j_i y_i + \sum_k \s(k) \er(k) |z_k|^2\,,
   \end{equation}  $$\quad \mathbb K = \sum_i |\mathtt j_i|^2 y_i + \sum_k \s(k) |\er(k)|^2 |z_k|^2 .$$  
   In fact it is enough to compare the contribution of each block of given root $\er$, to these three quantities, it is the sum $\sum_{k \,|\, \mathtt r(k)=r} \s(k)  |z_k|^2$ times $1, \er(k),|\er(k)|^2$ respectively. So it is enough to  see that the quadratic expression $\sum_{k \,|\, \mathtt r(k)=r} \s(k)  |z_k|^2$ remains invariant.  This expression is in fact the standard form   \eqref{stanfor1}  and in our  case we can diagonalize the matrix block by an orthogonal transformation with respect to this form, the claim follows.
\begin{corollary}\label{partoKAM} i)  Let $(\R_+)^n_e$ be the elliptic open region, where all eigenvalues of all combinatorial matrices are real and $\mathfrak K_e$  the corresponding domain. 

ii)  For all $\xi\in \ro\mathfrak K_e$ the NLS normal form is
$$(\omega,y) +\sum_k \Omega_k |z_k|^2 \,, $$
where $$\omega_i= |\mathtt j_i|^2 -2 \xi_i\,,\quad \Omega_k = \sigma(k)( |\er(k)|^2 +2\val_k) .  $$ 
The functions  $\val_k$ are real valued and analytic. \end{corollary}
\subsubsection{Complex coordinates\label{coco}}Although it is not strictly necessary it is convenient to diagonalize also the blocks with complex eigenvalues, although this implies the introduction of non real symplectic transformations.

We write everything in possibly complex coordinates as $\mathcal N=(\omega,y)+ \sum_{k\in S^c } \Omega_k |\zeta_k|^2 $, where $\zeta_k=z_k$ if $k\in S^c_r$. In order to do this we have to define $\zeta_k$ for $k\in S^c_i$. Consider one of the terms   , say ${\mathbb K_1}^{(h,k)}:=a_{h,k} (|z_h|^2-|z_k|^2) + b_{h,k}(z_h z_k +\bar z_h\bar z_k)$ and set $\vartheta_{k}:= a_{h,k}  + \ii  b_{h,k}$, $\vartheta_h= a_{h,k}  - \ii  b_{h,k}$ . Think of  $z\mapsto \bar z$ as a {\em $\C $ linear map on polynomials!} so that
$\overline{z_h+\ii \bar z_k}= \bar z_h+\ii  z_k$ and thus setting $\zeta_h:= \frac{z_h+\ii \bar z_k}{\sqrt{2}},\ \zeta_k:= \frac{ \bar z_h-\ii   z_k}{\sqrt{2}}$ we have that $\zeta_h  ,\bar  \zeta_h  $ (resp. $\zeta_k,\ \bar \zeta_k$)  are eigenvectors with opposite eigenvalues for $ad(\mathbb K_1)$:
$$\{\mathbb K_1^{(h,k)},\zeta_h^\s\}= \s\ii\vartheta_h\zeta_h^\s,\quad   \qquad      \{\mathbb K_1^{(h,k)},\zeta_k^\s\}=-\s\ii\vartheta_k \zeta_k^\s$$
$$\{ \zeta_h^{\sigma_1},\zeta_k^{\sigma_2}\}=0,\quad  \{\bar\zeta_h,\zeta_h\}=\{\bar\zeta_k,\zeta_k\}= \ii,\  \bar\alpha |\zeta_h|^2+ \alpha |\zeta_k|^2=\mathbb K_1^{(h,k)}. $$
 Moreover in these coordinates the three quantities   $\mathbb L, \mathbb M,\mathbb K$    are still in the form of Formula \eqref{posta}.
 In fact the same argument that we gave before applies  since the complex transformation that we have used is in the orthogonal group of the form.
 
We can finally claim that we can use the notation $z_k$ also for complex coordinates and write $\mathcal N=(\omega,y)+ \sum_{k\in S^c } \Omega_k |z_k|^2$
we assume that the $\Omega_k$ are all distinct and the complex ones come together with their conjugates according to the previous rules. 
\smallskip

Summarizing:

 A monomial $\mathfrak m= e^{\ii (k,x)}y^lz^\alpha\bar z^\beta$ has momentum $\ii\pi_\er(\mathfrak m)$ with\begin{equation}
\label{mome}\pi_\er(\mathfrak m):=\pi_\er(k,\alpha,\beta)=\pi(k) +\sum_{j\in S^c}(\alpha_j-\beta_j)\s(j)\er(j)=\pi (k,\alpha,\beta) +\sum_{j\in S^c}(\alpha_j-\beta_j)(\s(j)\er(j)-j) 
\end{equation}and it satisfies momentum conservation if  
$
\pi_\er(k,\alpha,\beta)=0
$.
Note that given functions $f,g$ which are eingenvectors of momentum we have $\pi_\er(\{f,g\})=\pi_\er(f)+\pi_\er(g)$.
  Given $k\in S^c $ (corresponding to the eigenvalue $\val_k$ of $C_\GA$)   the
   monomial $e^{\ii \sigma\sigma(k)\nu\cdot x}z_k^{\sigma}$ is an eigenvector for all our operators with eigenvalues:
 $$ \{ \mathbb L, e^{\ii \sigma\sigma(k)\nu\cdot x}z_k^{\sigma}\}= \ii \sigma\sigma(k)(\sum_i \nu_i +1)\,e^{\ii \sigma\sigma(k)\nu\cdot x}z_k^{\sigma},$$$$  \{ \mathbb M, e^{\ii \sigma\sigma(k)\nu\cdot x}z_k^{\sigma}\}= \ii \sigma\sigma(k)(\sum_i \nu_i \mathtt j_i+ \er(k) )e^{\ii \sigma\sigma(k)\nu\cdot x}z_k^{\sigma}$$
  $$\{ \mathbb K, e^{\ii \sigma\sigma(k)\nu\cdot x}z_k^{\sigma}\}= \ii \sigma\sigma(k)( \sum_i \nu_i |\mathtt j_i|^2 + |\er(k)|^2)\,e^{\ii \sigma\sigma(k)\nu\cdot x}z_k^{\sigma},$$ $$\{ \mathbb K^1, e^{\ii \sigma\sigma(k)\nu\cdot x}z_k^{\sigma}\}= \ii \sigma\sigma(k)(- \sum_i \nu_i \xi_i +\val_k)\,e^{\ii \sigma\sigma(k)\nu\cdot x}z_k^{\sigma},$$

\section{The kernel of $ad(\mathcal N)$} 

\subsubsection{ Non-degenerate quadratic Hamiltonians}
Consider a quadratic Hamiltonian
\begin{equation}
\mathcal Q=(\omega,y)+ \sum_{k\in S^c_r} a_k |z_k|^2 + \sum_{(h,k)\in S^c_i} a_{h,k} (|z_h|^2-|z_k|^2) + b_{h,k}(z_h z_k +\bar z_h\bar z_k)
\end{equation} we want to study  the kernel of $ad(\mathcal Q)$ on the space of Hamiltonians of degree $\leq 2$. We set $\vartheta_{h,k}:=a_{h,k}+\ii b_{h,k}$  so that $\pm\vartheta_{h,k},\pm\bar \vartheta_{h,k}$ are the four eigenvalues of $ad(a_{h,k} (|z_h|^2-|z_k|^2) + b_{h,k}(z_h z_k +\bar z_h\bar z_k))$ acting on the space spanned by $z_h,z_k,\bar z_h,\bar z_k$. It is convenient to write
 $z^1=z,\ z^{-1}=\bar z$ so if we do not want to specify if a variable is $z$ or $\bar z$ we write is as $z^\sigma$ where $\s $ can be $\pm 1$.
 
 Next the operator $\ii \mathcal Q$ acts on the real space spanned by the elements $e^{\ii \s \nu\cdot x}z_k^{\sigma}$ (we need some convergence conditions given by its norm).  If as in our  case  $\mathcal Q$  commutes with the mass and momentum then it acts  also  on the subspace $F^{0,1}$  where \begin{definition}
We denote by  $F^{0,1}$  the space of functions spanned   by the   elements $e^{\ii \s \nu\cdot x}z_k^{\sigma}$ with zero mass and momentum.  I.e. $\sum_i\nu_i=-1,\sum_i\nu_i\mathtt j_i+k=0$.
\end{definition}\begin{definition}
We say that $\mathcal Q$ is {\em non-degenerate}\footnote{in the usual language  we should say {\em regular semisimple}.} if the coordinates $\omega_i $ are  linearly independent over $\Q$ and  its eigenvalues for the action on $F^{0,1}$ are all non--zero and distinct. 
\end{definition}
It is then not difficult to analyze the kernel of  $\ii\, ad(\mathcal Q)$, i.e. the elements which Poisson commute with   $\ii\, ad(\mathcal Q)$, on the space of Hamiltonians of degree $\leq 2$ commuting with mass and momentum, we have:
\begin{proposition}\label{Pocom} If  $\mathcal Q$ is  non-degenerate then a (real)  Hamiltonian of degree $\leq 2$, which commutes with mass,   Poisson commutes with  $\mathcal Q$ if and only if it is of the form:
\begin{equation}\label{comfor}
\mathcal Q'=(\omega',y)+ \sum_{k\in S^c_r} a'_k |z_k|^2 + \sum_{(h,k)\in S^c_i} a'_{h,k} (|z_h|^2-|z_k|^2) + b'_{h,k}(z_h z_k +\bar z_h\bar z_k)
\end{equation} with $a'_k, a'_{h,k} ,b'_{h,k} \in\R$.
\end{proposition}
\begin{proof}  It is immediate that a Hamiltonian of the form of Formula \eqref{comfor}  commutes with $\mathcal Q$  we need to show the converse.\smallskip

 {\bf Degree zero in $w$:} Monomials of degree $\leq 2$ and of degree 0 in $w$ are of the form $y^\ell e^{\ii \nu\cdot x},\ \ell=0,1$ and for the  eigenvalues  we have:
\begin{equation}\label{piffero}
 \{ \mathcal Q, y^\ell e^{\ii \nu\cdot x}\}= - \ii (\omega,\nu)\,y^\ell e^{\ii \nu\cdot x}
\end{equation}
By hypothesis of linear independence over  $ \Q $ these eigenvalues are 0 if and only if $\nu=0$,   hence the Kernel of $ad(\mathcal N )$ is $x$ independent and hence of the form $ c  +(\omega',   y)$.
 
 \smallskip
 
  {\bf Degree one in $w$:} By definition the eigenvalues of the adjoint action of $\mathcal Q$ on $F^{0,1}$ are non-zero. 

 \smallskip
 
  {\bf Degree two in $w$:}  We write everything in possibly complex coordinates as $\mathcal Q=(\omega,y)+ \sum_{k\in S^c } a_k |\zeta_k|^2 $ as in the previous paragraph.  We assume that the $a_k$ are all distinct and the complex ones come together with their conjugates according to the previous rules. Then we can write any monomial of degree 2
$$M=e^{\ii \nu\cdot x} \zeta_h^{\sigma_1} \zeta_k^{\sigma_2},\quad \{\mathcal Q,\, M \}=\ii [(\omega,\nu)+\sigma_1a_{h}+  \sigma_2a_{k } ]  M$$
conservation of mass $\eta(\nu)+\sigma_1+\sigma_2=0$ implies   that we can write  $\nu=\nu_1+\nu_2$ with $\eta(\nu_1)+\sigma_1= \eta(\nu_2)+ \sigma_2=0$.  Then $(\omega,\nu_1)+\sigma_1a_{h} $ and $(\omega,\nu_2)+   \sigma_2a_{k }$ are two eigenvalues of elements in $F^{0,1}$  since the eigenvalues are all distinct we must have that  the corresponding eigenvectors are one the conjugate of the other and then $M$ is of the form $|\zeta_k|^2$ for some $k$.  Finally if we assume that the Hamiltonian is real  the two terms associated to a pair $h,k$ giving complex eigenvalues must have conjugate coefficients so that in real coordinates they give a term of type $a'_{h,k} (|z_h|^2-|z_k|^2) + b'_{h,k}(z_h z_k +\bar z_h\bar z_k)
$.\end{proof}
  {\bf Warning}\quad From now on even if we shall use complex coordinates we shall denote them by $z_k$ and not $\zeta_k$.

\subsection{Eigenvalues and eigenvectors\label{eiei}}The fact that  the Hamiltonian  $\mathcal N$  is non-degenerate  for generic values of $\xi$ is essentially a consequence of Theorem \ref{te12p}, we state it as:
\begin{proposition}\label{nodeno}
 The normal form $\mathcal N$ is non-degenerate for all $\xi$ outside countably many algebraic hypersurfaces. 
\end{proposition} \begin{proof}
We have $\ome_i= |\mathtt j_i|^2 -2\xi_i$ which are linearly independent over the rationals outside  countably many algebraic hypersurfaces (the dependency relations).
 
 The eigenvalues of the action on $F^{(0,1)}$ are the roots of the characteristic polynomials of the blocks into which this space decomposes. We have seen that these polynomials are all irreducible and distinct. For a block of size $>1$ irreducibility implies that the constant term of the characteristic polynomial is a non--zero polynomial in $\xi$ so the eigenvalues are non--zero outside the hypersurfaces given by these determinants. Moreover it is immediate that for the  blocks of  size 1 the  eigenvalues are non--zero linear polynomials.
 
 In order to have that all the eigenvalues be distinct we have to remove  the discriminant $\mathfrak A$, and countably many {\em resultants} which are all non--zero polynomials in $\xi$ by the irreducibility and separation Theorem. \end{proof} 
%

\begin{remark}
\label{lemme} In the course of the KAM algorithm  of Part \ref{rottura} we shall see that the non-degeneracy of the Normal form plays a fundamental role. Imposing the non-degeneracy however requires removing a countable number of proper hypersurfaces hence working on complicated sets in parameter-space.
In order to avoid this problem, we impose that the compact domains $\mathfrak K_\al$  of Formula \eqref{ikappaa}, should   be disjoint from finitely many resultants, i.e. we fix an integer $S_0$ and impose $\forall \xi\in \cup_\al\mathfrak K_\al$:
$$ -(\xi,k) + \vartheta_i(\xi) \pm \vartheta_j(\xi) \neq 0\,.  $$
for all the couples of eigenvalues  of (different) combinatorial matrices and  forall $k\in \Z^n\,, |k|<S_0$ .
\end{remark}

We would like to choose coordinates, independent of $\xi$,   which are eigenvectors    for $\mathbb M$, $\mathbb L$, $\mathcal N$  so that
\begin{lemma}
A regular analytic function $F$ Poisson commutes with $\mathcal N$, $\mathbb M$, $\mathbb L$ for generic values of $\xi$ if and only if each monomial appearing in $F$ Poisson commutes with   $\mathbb M$, $\mathbb L$, $\mathbb K$ and  $\mathbb K^1$.
\end{lemma}\begin{proof}
An element commutes with $\mathcal N$ if and only if it commutes with both homogeneous parts of degree 0,1 that is $\mathbb K$ and  $\mathbb K^1$.
\end{proof}

 \part{Quasi  T\"oplitz functions}

\section{Optimal presentations cuts and Good Points} 
{\em Let us recall the definition of $N$-optimal presentations, cuts, good points as given in \cite{PX}, we omit most proofs.}

\subsection{$N$-optimal presentations}
An affine space $A$ of codimension $\ell$ in $\R^d$  can be defined
by a list of $\ell$ equations $A:=\{x\,|\,(v_i, x)=p_i\}$ where
the $v_i$ are independent row vectors in $\R^d$.  We will write shortly that
$A=[v_i;p_i]_{\ell}$. We will be interested in particular in the
case when $v_i,p_i$ have integer coordinates, i.e.  are {\em integer
vectors}\footnote{such a subspace is usually called {\em rational}.} and the vectors $v_i$  lie in a prescribed ball $B_N$ of radius  some constant $\kappa N$.
Recall that,   \eqref{essec1}, we have set ${\kappa} := \max_i |\mathtt j_i|$. We denote by $$\langle v_i\rangle_{\ell}={\rm
Span}(v_1,\dots,v_\ell;\R)\cap \Z^d \,,\quad B_N:=\{x\in
\Z^d \setminus \{0\}\,\vert\; |x| < {\kappa}   N\}  ,$$  here $N$ is
any large number.  In particular we  implicitly assume that $B_N$  contains a basis of $\R^d$. \smallskip

For given $s\in \N$, in the set of vectors $\Z^s$  we can define  the {\em sign lexicographical order $\prec$ as in \cite{PX}} as follows.\begin{definition}
 Given $a=(a_1,\ldots,a_s)$ set $(|a|):=(  |a_1|,\ldots,|a_s|)$ then we set $a\prec b$  if either $(|a|)< (|b|)$ in the lexicographical order (over $\N$) or if  $(|a|)=(|b|)$ and $a>b$ in the lexicographical order in $\Z$.
\end{definition}
  With this definition every non empty set of elements in $\Z^s$ has a unique minimum.

Notice that, by convention, among the finite number of vectors with a given prescribed value of  $(|a|)$ we have chosen as minimum the one with non negative coordinates.\smallskip

  In particular consider a fixed but  large enough $N$.
  \begin{definition}
We set  $\mathcal H_{N }$ to be the set of all   affine spaces $A$ which can be presented as $A=[v_i;p_i]_\ell$  for some $0<\ell\leq d$ so that that $v_i \in B_N,\, p_i\in\mathbb N$. 

We denote the  subset of $\mathcal H_{N }$ formed by the subspaces of codimension $\ell$ by $\mathcal H_{N }^\ell$.\end{definition}\noindent We display as  $(p_1,\ldots,p_\ell;v_1,\ldots,v_\ell)$  a given presentation, so that it is a vector in $\Z^{\ell(d+1)}$. Then we can say that $[v_i;p_i]_\ell \prec [w_i;q_i]_\ell$ if  $(p_1,\ldots,p_\ell;v_1,\ldots,v_\ell)\prec (q_1,\ldots,q_\ell;v_1,\ldots,v_\ell)$.
  \begin{definition}
 The $N$--optimal presentation $[l_i;q_i]_\ell$ of $A\in \mathcal H_{N }^\ell$ is the minimum, in the sign lexicographical order, of the presentations of $A$ which satisfy the previous bounds.
 
 Given an affine subspace $A:=\{x\,\vert (v_i, x) =p_i\,,\  i=1,\dots,\ell\}$ by the notation $A\frec{N} [v_i;p_i]_\ell$ we mean that the subspace has codimension $\ell$ and the given presentation is $N$--optimal.  
 \end{definition}

\begin{remark}\label{lordi} i) Note that each point $m=(m_1,\ldots,m_d)\in \Z^d $ has a $N$--optimal presentation (this presentation is usually not the naive one $[e_i,m_i]_d$ where the $e_i$ form the standard  basis of $\Z^d$).

ii) Thus we may use the ordering given by  $N$--optimal presentations of points  in order to define a new lexicographic order on $\Z^d$ which we shall denote by $a\prec_N b$ or $a\prec b$ when $N$ is understood.
\end{remark}
\begin{remark}
At this point we extend the definition of $\prec$ to the elements of  $\mathcal H_{N }$ by using their $N$--optimal presentation.
\end{remark}
\begin{lemma}\label{basso}
i) If the presentation $A=[v_i;p_i]_\ell$ is $N$--optimal,  we have
\begin{equation}\label{va2}
0\leq  p_1 \leq p_2 \leq\ldots\leq  p_\ell .  \end{equation}

 ii) For
all $j<\ell$ and for all $v\in \Bn\cap(\langle
v_1,\dots,v_\ell\rangle \setminus\langle
v_1,\dots,v_j\rangle)$, one has:
\begin{equation}\label{va3}
|(v,r)|\geq  p_{j+1}\,,\quad  \forall r\in A.\end{equation}

iii) Given $j<\ell$ set $A_j:=\{x\,|\,(v_i, x)=p_i,\ i\leq j\},$ then the presentation $A_j=[v_i,p_i]_j$ is $N$--optimal.

iv) Finally   $-A$  has a $N$--optimal presentation $-A =[v'_i,p_i] $ with the same constants $p_i$ and   $(|v'_i|)=(|v_i|)$.
\end{lemma}

\begin{remark}\label{numero}
For fixed $N$, $\ell$, $p$  the number of  affine spaces in $\mathcal H_{N }^\ell$ of codimension $\ell$ and such that $ p_\ell \leq p$ is bounded by
$(2{\kappa}   N+1)^{\ell d} (p+1)^{\ell}$.
\end{remark}

\subsection{ A decomposition of $\Z^d$}
   We shall need several auxiliary parameters in the course of our proof.
 We start by fixing some numbers 
  \begin{equation}
\label{itau}\tau_0> \max( d^2+ n,12),\quad \td:=(4d)^{d+1}(\tau_0+1)\,,
\end{equation}
$$
\mathtt c\leq \frac12\,,\; \mathtt C\geq 4 \,,\; N_0 =(d+1)! {\kappa}   ^{d+1} \mathtt C \mathtt c^{-1}.
$$
In what follows $N$ will always denote some large number, in particular $N>N_0$. \smallskip

Using the fixed parameters $\mathtt c,\mathtt C $ and the notion of optimal presentation, for each  $N>N_0$ we want to construct a decomposition \begin{equation}\label{pustola}
\Z^d= \cup_{i=0}^{d} A_i(N) 
\end{equation}of $\Z^d $ which will be crucial for the estimates of small denominators (cf. \S \ref{Kamstep}) and
 given by the following
 \begin{definition} \label{papilla} i)\quad  A subspace $A\frec{N}[v_i;p_i]_\ell\in \mathcal H_N^\ell$ with $1\leq \ell<d$ is  called {\em $N $--good} if    $ p_\ell \leq \mathtt c N^{{\tau_1}\over 4d}$. The set of $N $--good subspaces  of codimension $\ell<d$ is denoted by $\mathcal H_N^{\ell ,g}$.
\smallskip

ii)\quad  Given $A\in \mathcal H_N^{\ell,g}$ the set: 
\begin{equation}\label{pippo}
A^{g} :=  \left\{x\in
A\cap \Z^d\,\vert \quad  |\er(x)|> \mathtt C N^{\tau_1}\,, \; |(v,x)|\geq \mathtt  C \max(
N^{4d\tau_0},\mathtt c^{-4d} p_\ell ^{4d}) , \forall v\in B_N\setminus
\langle v_i\rangle_\ell\right\} \end{equation}  will be called the $N-${\em good}
portion of the subspace $A$.
  \end{definition}

\begin{remark}
 Notice that every $v\in B_N\setminus
\langle v_i\rangle_\ell $ gives a non constant linear function $(v,x)$ on $A$. Thus the good points of $A$ form a non empty open set complement of a finite union of strips  around subspaces of codimension 1 in $A$.
Note moreover that  we are interested only in integral points  and the integral points in $A$ which are not good are formed, by the finitely many points with $|\er(x)|\leq \mathtt C N^{\tau_1}$, plus  a finite union of affine subspaces of codimension one in $A$.\end{remark}
We construct a decomposition of $\Z^d $ using the following Proposition   which is a variation of  \cite{PX} Proposition 1).
\begin{proposition}\label{key2}
Each point $m\frec{N}[v_i,p_i]$ with $|\er(m)|>\mathtt C N^{{\tau_1}}$  and $ p_1 < \mathtt C N^{4d\tau_0}$ belongs to the set $[v_i;p_i]_\ell^g$  for some choice $0<\ell<d$. 
\end{proposition}
\begin{proof}
Consider $A_1:=[v_1;p_1]$, since $\mathtt C N^{4d\tau_0}\leq \mathtt c N^{ \frac{\tau_1}{4d}}$   we see that it is $N$--good, so if $m\in A_1^g$ we are done, otherwise we have that  $p_2<    \mathtt  C \max(
N^{4d\tau_0},\mathtt c^{-4d} p_1 ^{4d})\leq  \mathtt c N^{{\tau_1}\over 4d}.$ Consider $A_2:=[v_1,v_2;p_1,p_2]$, we   see again that it is $N$--good, so if $m\in A_2^g$ we are done otherwise we have that  $p_3<    \mathtt  C \max(
N^{4d\tau_0},\mathtt c^{-4d} p_2 ^{4d})\leq  \mathtt c N^{{\tau_1}\over 4d} $ we continue in this way and either we show that $m\in A_i^g, \ i<d$ or we have a sequence of inequalities
$p_j<    \mathtt  C \max(
N^{4d\tau_0},\mathtt c^{-4d} p_{j-1 }  ^{4d}). $  Let $k\geq 1$ be the maximum index  such  $p_k<  \mathtt  C 
N^{4d\tau_0}$, if $k<d$  we have for all $p_j, j=k+1,\dots, d$ that $p_j<(\mathtt c^{-4d}\mathtt C)^{j-k} N^{(4d)^{j-k+1}\tau_0}.$ We then compute the coordinates of $m$ by Cramer's rule and by just estimating the numerator we have a sum  of  $d!$ terms each of which can be  bounded by $N^{d-1}N^{(4d)^{d}\tau_0}  (\mathtt c^{-4d}\mathtt C)^{d-1} $. This sum is then bounded by $N^{5d(d-1)+(4d)^{d}\tau_0}$, it follows that $|m|\leq \sqrt{d}N^{5d(d-1)+(4d)^{d}\tau_0}$  and $\er(m)\leq  \sqrt{d}N^{5d(d-1)+(4d)^{d}\tau_0}+d\kappa$. Recall $\td:=(4d)^{d+1}(\tau_0+1)$ thus we easily see that $|\er(m)|<\mathtt c N^\td$.
\end{proof}From this proposition to a point  $m\frec{N}[v_i,p_i]$ with $|\er(m)|>\mathtt C N^{{\tau_1}}$  and $ p_1 < \mathtt C N^{4d\tau_0}$ we have associated an affine space $A\in \mathcal H_N^{\ell,g}$  such that   $m\in A^g$ and we set $A_\ell(N)$ to be the set of points of previous type for which $A$ has codimension $\ell$.

The remaining points may be distributed in two sets:
$$ A_d= A_d(N)\subset  \{m\in \Z^d\,:  |\er(m)|\leq \mathtt C N^{\tau_1}\}$$
and 
\quad \begin{equation}\label{nodiv}
 A_0:=   \{ m\in \Z^d  \;:\; m\frec{N}[v_i;p_i]\quad{\rm with} \quad |\er(m)|>\mathtt C N^{{\tau_1}},\quad p_1\geq  \mathtt C K^{4d \tau_0} \}.
\end{equation}
In this way we construct a decomposition of $\Z^d=\cup_{\ell=0}^dA_\ell(N)$.
 \begin{figure}[!ht]
\centering
\begin{minipage}[b]{11cm}
\centering
{\psfrag{A}{$A_0$}
\psfrag{a}{$\mathtt j_j$}
\psfrag{b}{$\mathtt j_i$}
\psfrag{c}{$ a_2$}
\psfrag{d}{$ b_2$}
\psfrag{e}{$ a_1$}
\psfrag{f}{$ b_1$}
\psfrag{H}{$H_{i,j}$}
\psfrag{S}{$S_{i,j}$}
\psfrag{m}{$ \mathtt j_j-\mathtt j_i$}
\psfrag{l}{$ \mathtt j_j+\mathtt j_i$}
\includegraphics[width=11cm]{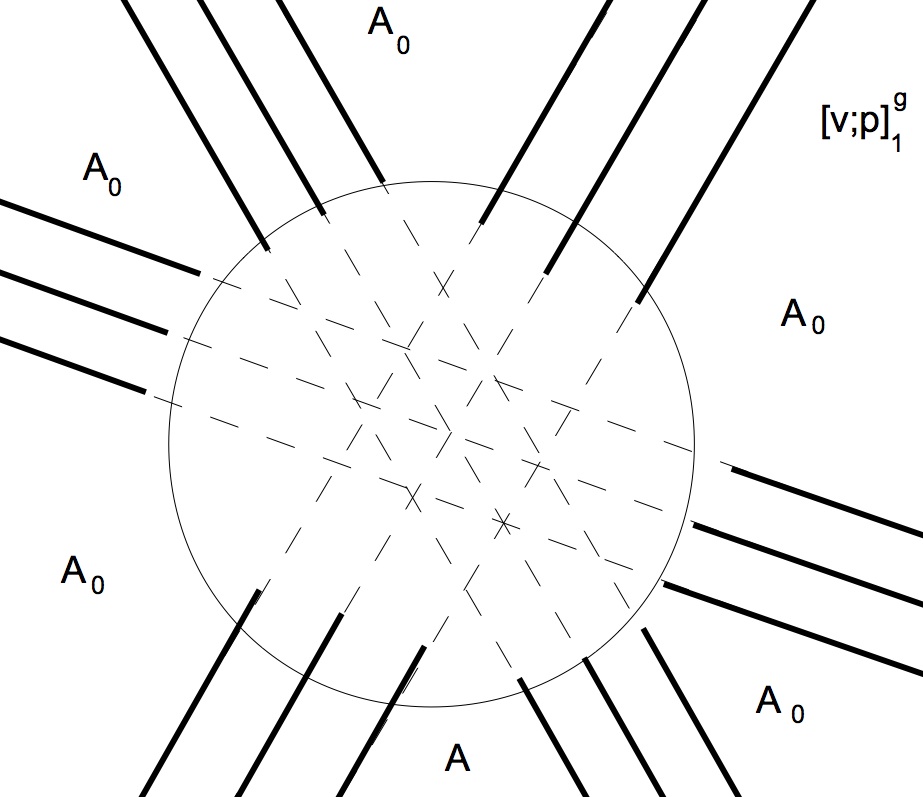}}
\caption{\footnotesize{A drawing of the standard decomposition in $\Z^2_1$. $A_0$ is $\Z^2_1$ minus the dashed lines (each dashed line is described by an equation $[v;p]_1$). On each dashed line the set $[v;p]_1^g$ is signed in solid boldface. Note that $[v;p]_1^g$ is $[v;p]_1\cap \Z^2_1$ minus a finite number of subspaces of codimension two, i.e. points.} \label{fig2}}
\end{minipage}
\end{figure}

\subsection{Cuts} In the previous paragraph  a point $m$ which is a good point has an associated affine space $A$ of some codimension $\ell$ which, as we have seen in the proof of Proposition \ref{key2} corresponds to a jump, or as we shall say a {\em cut}, in the optimal presentation.  In fact for technical reasons  having to do with the structure of Poisson bracket we need to refine this notion, introducing auxiliary parameters $\mu,\theta$ which give a better control on the cut and allow some flexibility in the constructions.  This is the topic of this paragraph.\smallskip

We assume that $N$ has been fixed.
  Given a point $m$ we write $m \frec{N}[v_i;p_i]$ for its optimal presentation  dropping
the index $\ell$ which for a point equals $ d$, we implicitly also mean that $m\in \Z^d$.
 Set by
convention $p_{0}=0$ and $p_{d+1}=\infty$.

We  then make a definition involving three more parameters: 
\begin{definition}
\label{allpa} The parameters $N, \theta,\mu,\tau$ are called {\em allowable} if 
$$\tau_0\leq \tau\leq\tau_1/(4d),\quad  \mathtt c< \theta,\mu<\mathtt C,\quad  N> N_0.$$
\end{definition} 

We need to
analyze certain {\em cuts}, for the values $p_i$ associated to an optimal presentation of a point.  This will be an index $\ell$  where the values of the $p_i$ jump according to the following: 

\begin{definition}\label{cut}
The point $m \frec{N}[v_i;p_i]$ has a {\em cut} at  $\ell\in\{0,1,\dots,d\}$  with the  parameters  $\underline p=( N,\theta,\mu,\tau)$,   if $\ell$ is such that $  p_\ell  < \mu
N^{\tau}$, $ p_{\ell+1 }  > \theta N^{4 d \tau}$. \smallskip

The space $A:=\{x\,|\, (v_i, x) =p_i,\ i=1,\ldots,\ell\}$  has $[v_i;p_i]_\ell$  as optimal presentation and it is called the {\em affine space associated to the cut} of $m$.
\end{definition}
 \begin{remark}\label{oboe}   Note that, by Lemma \ref{basso}, if $m\frec{N}[v_i;p_i]$ has a $(N,\theta,\mu,\tau)  $ cut at $\ell$  then  $|(m,v)|> \theta N^{4 d \tau} $ for all $v\in B_N\setminus \langle v_i\rangle_l $. 
 \end{remark}

Consider a subspace $A\in\mathcal H_N$ of codimension $\ell$   such that in its optimal presentation $p_\ell<\mu N^\tau$.  The set of points $m\in A$  which have $\ell$ as a cut with the  parameters $N, \theta,\mu,\tau$ have $A$ as associated affine space, if furthermore   $|\er(m)|> \theta N^{\tau_1}$  we call them $(N, \theta,\mu,\tau)$--{\em good points of} $A$, we write $ m\in A^g_{(N, \theta,\mu,\tau)}$.

By definition the other affine spaces have no good points with respect to these parameters.\smallskip
\begin{definition}\label{isotb} Given allowable parameters $\underline p=(N, \theta,\mu,\tau)$ 
we denote by $\mathcal H_{\underline p}^\ell$ the set of affine subspaces $A\in\mathcal H_N$ of codimension $\ell$   such that in their optimal presentation $p_\ell<\mu N^\tau$.  The union of all these sets for $0<\ell<d$ is denoted by
$\mathcal H_{\underline p}$.
\end{definition}
Notice that  $\theta N^{4 d \tau}>\mu
N^{\tau}$ (since by \eqref{itau} we have $\mathtt c N^{4 d \tau}>\mathtt C N^{  \tau}$), so  for any given $m\in S^c $   there is at  most {\bf one} choice of $\ell$ such that $m$ has a $\ell$ cut with parameters $\theta,\mu,\tau$.
\begin{remark}
1)\quad The purpose of defining a cut $\ell$ is to separate the
numbers $p_i$ into {\em small} and {\em large}. The parameters $N, \theta,\mu,\tau$
give a quantitative meaning to this statement.

2)\quad The set of good points  $A^g_{(N, \theta,\mu,\tau)}$, if non--empty,  is the complement in $A$ of a finite number of codimension one subspaces plus finitely many points. 

3)\quad Given any rational affine subspace $A$ (i.e. defined by equations over $\Z$) there is an $\bar N$ so that $\forall N\geq \bar N$ we have $A\in \mathcal H_N$, its optimal presentation is independent of $N$, the set $A^g_{(N, \theta,\mu,\tau)}$  is non-empty.

\end{remark}
We need an auxiliary parameter depending on an optimal presentation
\begin{definition}\label{taup}
Given $p\leq  \mathtt c N^{\tau_1/(4d)}$, set  $\tau(p)$ so that $N^{\tau(p)}=\max(N^{\tau_0}, {\mathtt c}^{-1} p )$. 
\end{definition}Notice that $\tau_0\leq \tau(p)\leq \tau_1/(4d)$.
The connection between the notion  of good points $A^g$, defined in \eqref{pippo}, of a given subspace $A$ and the notion just introduced is explained by the following Lemma.
\begin{lemma}\label{minko} For all $\mathtt c< \theta,\mu<\mathtt C$
and for all affine subspaces $[v_i;p_i]_\ell\in {\mathcal H}_N $ such   $  p_\ell \leq  \mathtt c N^{\tau_1/(4d)}$, we have that  every point $m\in [v_i;p_i]_\ell^g$  is a $(N,\theta,\mu,\tau(p_\ell))$--good point for $[v_i;p_i]_\ell$. I.e. $[v_i;p_i]_\ell^g\subset ([v_i;p_i]_\ell)^g_{(N,\theta,\mu,\tau(p_\ell))}$ for all $\mathtt c<\theta,\mu<\mathtt C$. 
\end{lemma}

 \begin{remark}\label{varlm0}
1)\quad If    $m \frec{N}[v_i;p_i]$ has a cut at $\ell$      for the
  parameters $ \theta',\mu' ,\tau$ then it has also a cut at $\ell$   for parameters $\theta,\mu,\tau$ with $ \theta\leq
\theta',\ \mu'\leq \mu $ provided  $ \theta,\mu,\tau$ are   allowable.

If $ \theta\leq
\theta',\ \mu'\leq \mu $ we shall say that the allowable parameters $\theta,\mu$ are {\em less restrictive } than $\theta',\mu'$.\smallskip

2)\quad If for a given $\ell,\tau$ we have  $p_\ell\leq \mathtt cN^\tau,\, p_{\ell+1}\geq \mathtt CN^{4d\tau}$, then $\ell$ is a cut with parameters $ \theta,\mu,\tau$ for every choice of allowable $ \theta,\mu$.

\end{remark}
\begin{lemma}\label{mah}[Neighborhood property]
Consider $m,r \in \Z^d $ with $m\frec{N}[v_i;p_i]$, $r\frec{N}[w_i;q_i]$.
 Suppose that $m$ has a  cut at $\ell$    for     the parameters $N, \theta',\mu',\tau$,  and suppose there exist  allowable parameters   $ \theta<\theta',\mu'<\mu $:
 \begin{equation}
\label{vicin}|r-m|<\min({\kappa}   ^{-1}(\mu-\mu')N^{\tau-1},\ {\kappa}   ^{-1}(\theta'-\theta)N^{4d\tau-1} ).
\end{equation} then:

 (1) The point $r$ has a cut at $\ell$  for all allowable parameters  $\theta, \mu,\tau$  for which \eqref{vicin} holds.
\smallskip

(2) $\langle
w_1,\dots,w_\ell\rangle=\langle v_1,\dots,v_\ell\rangle$.

(3)  $[w_i;q_i]_\ell$ is the $N$--optimal presentation of $[v_i;p_i]_\ell+r-m$.
\end{lemma}
\begin{corollary}\label{dueta}
Consider $m\frec{N}[v_i;p_i]$, $r\frec{N}[w_i;q_i]$  such that
 \begin{equation}\label{edueta}
 |r-m| < {\kappa}   ^{-1}c( N^{4d \tau-1}- {\mathtt c \mathtt C^{-1}} N^{\tau -1})
\end{equation}
and both $m,r$ have a cut with parameters $\underline p=(N,\theta,\mu,\tau)$  then we can deduce:

\noindent i)\quad the vectors $m,r$  have the cut at the same $\ell$;

\noindent ii)\quad   the space $B$ associated to the cut of $r$  is the one parallel to $A=[v_i;p_i]_\ell$ and passing through $r$ namely
\begin{equation}   B = A+ r-m.
 \label{traslo} 
\end{equation}
\end{corollary}
The previous results explain why we wanted to introduce the parameters $\mu,\theta$  to define cuts, in  fact  
\begin{remark}
With the above lemma we are stating that if $m$  has a $\ell$ cut with parameters $\theta',\mu',\tau$  then, for all choices of $\theta<\theta',\mu'<\mu$, for which $ \theta,\mu$ are allowable parameters, we have described  a spherical neighborhood $B$ of $m$ such that all points $r\in B$  have a $\ell$ cut with parameters $\theta,\mu,\tau$. The radius of $B$ is determined by Formula \eqref{vicin}. Note moreover that if $r$ has a cut at $\ell$ for some parameters then so has $-r$ and with the same parameters. Then lemma \ref{mah} holds verbatim if in formula \eqref{vicin} we substitute $|m-r|$ with $|m+r|$.
\end{remark}
We finally combine  \ref{minko} and \ref{dueta}
\begin{lemma}\label{lintor}
    For all affine subspaces $[v_i;p_i]_\ell$ with  $p_\ell \leq  \mathtt c N^{\tau_1/(4d)}$  the following holds. For all  $m\in \Z^d $  with $m\in [v_i;p_i]_\ell^g$,  for all $r\in \Z^d $ and
 for all parameters $\mathtt c< \theta,\mu<\mathtt C$ such that
\begin{equation}\label{ladisa}
|r-m|<{\kappa}   ^{-1}(\mu-\mathtt c) N^{\tau_0-1},{\kappa}   ^{-1}(\mathtt C-\theta)N^{4d\tau_0-1} ,
\end{equation}
$r,m$ have the same cut $ \ell$ with parameters $ \theta,\mu,\tau(p_\ell) $ with parallel corresponding affine spaces.\end{lemma}

The definitions which we have given are sufficient to define and analyze the quasi--T\"oplitz functions, which are introduced in section \ref{toppa}.
In the next subsection we collect some definitions which are useful for the measure estimates and which are independent of the auxiliary parameters $\theta,\mu$.

\subsection{Graphs and cuts\label{gac}} Recall that the choice of the vectors $S:=\{\mathtt j_i\}$  determines  a colored marked graph $\Gamma_S$ with vertices in the set $S^c$. 

 This graph has   finitely many  components containing red edges.\footnote{A rough estimate of a bound on the norm of these points is $(2d+3)\kappa$.} The remaining set will be denoted by ${\bar S}^c$ and it is a union of connected components each combinatorially isomorphic  to a combinatorial graph out of a finite list
 $\mathcal G:=\{\GA_1,\ldots,\GA_N\}$ formed only of black edges. It will be enough to concentrate our analysis only on these {\em black graphs}. 
 \begin{definition}
Given $\GA\in \mathcal G$ we set $\Sigma_\GA$ to be the union of all connected components of $\Gamma_S$ isomorphic to $\GA$.
\end{definition}
The set $\mathcal G$ is partially ordered  by setting $\GA_i\leq \GA_j$ if $\GA_i$  is isomorphic as marked graph to a subgraph of $\GA_j$.

From the theory developed it follows that, if $\GA\in\mathcal G$ has $d_\GA+1$ vertices we have that $\Sigma_\GA$ is a union of translates of any of its components (of the graph $\Gamma_S$).  Moreover $\Sigma_\GA$  is the portion of $\bar S^c$ in  a union of $d_\GA+1$ parallel affine subspaces of codimension $d_\GA$  minus   a union of finitely many affine subspaces of higher codimension  whose points lie in $\bigcup_{\GA_i> \GA}\Sigma_{\GA_i}$.

 Let us recall how we arrived at this statement.   We   choose a {\em root}  in each $\GA_i$. Using the root    a geometric realization of  some $\GA=\GA_i$ is an isomorphic graph, with vertices in $\bar S^c$,  in which the image $\er$ of the root     solves a certain set of $d_{\GA}$ independent linear equations (Formula (61) of \cite{PP}) and the other vertices are determined by the labels on the graph  $\GA$.
 The components $A$ of $\Gamma_S$ which are  isomorphic  to $\GA$  are exactly those     geometric realizations of   $\GA$   which  are not properly contained in a larger component.  Recall that we have imposed generic conditions so that the  vertices in a component are affinely independent. Therefore, a component  $A$ associated to $\GA$, spans an affine subspace $\langle A\rangle$ of dimension exactly $d_\GA$.  All other geometric realizations of $\GA$ are obtained from a given $A$  by translating the graph $A$ with the integral vectors orthogonal to $\langle A\rangle$.  This set is thus a union of $d_\GA+1$ parallel affine subspaces of codimension $d_\GA$ passing each through an element $m$ of $A$, image of a point $a\in\GA$.  It may well happen that a translate of $\er$ may solve also the linear equations defining a larger graph, the  points in the corresponding component lie thus in a stratum $\Sigma_{\GA_i},\ \GA_i>\GA$.

\begin{example}       
  $$   \xymatrix{ &x-\mathtt j_1+\mathtt j_3\ar@{<- }[d] _{2,1} ^{\qquad }& &equations&&(x,\mathtt j_1-\mathtt j_3)=|\mathtt j_1|^2-(\mathtt j_2,\mathtt j_3)\\ x\ar@{ ->}[r] _{3,2}\ar@{<-}[ru] ^{3,1} &x-\mathtt j_2+\mathtt j_3& &  &&(x,\mathtt j_2-\mathtt j_3)=|\mathtt j_2|^2-(\mathtt j_2,\mathtt j_3) }   \quad  
 $$  
 \end{example} Notice that this is a subgraph of the graph of example \ref{ungra}.

    If $a\in \GA $  we let $\Sigma_{\GA,a}$ be the subset of  $\Sigma_\GA$ formed by the corresponding elements so that $\Sigma_\GA$ is the disjoint union of the  strata  $\Sigma_{\GA,a}$ as $a$ runs over the vertices of $\GA$.    

The set $\Sigma_{\GA,a}$ spans an affine space  $\langle \Sigma_{\GA,a}\rangle$ of codimension $d_\GA$ and, in fact, it  is the complement in this affine space of the points which belong to graphs  which contain strictly $\GA $. Thus $\Sigma_{\GA,a}$ is obtained from the affine space $\langle \Sigma_{\GA,a}\rangle$ removing  a finite union of proper affine subspaces.  \smallskip

Thus for any $\GA\in\mathcal G$ having chosen a root $\er$  we have the stratum $\Sigma_{\GA_,\er}$ which is the complement in an affine space of a finite union of codimension 1 subspaces.  Any point $m\in {\bar S}^c$    lies in a unique stratum $\Sigma_{\GA,a}$ which is parallel to $\Sigma_{\GA,\er}$ we thus have for $m$ a corresponding root $\mathtt r(m)\in \Sigma_{\GA,\er}$ which is the intersection of $\Sigma_{\GA,\er}$ with the connected component of the graph $\GA_S$ in which $m$ lies. Thus $m-\mathtt r(m)$ depends only upon the stratum $\Sigma_{\GA,a}$.
\begin{definition}
We denote the vector $m-\mathtt r(m)$ as the {\em type of $m$}.
\end{definition}

\begin{remark}
\label{types} The possible types run on a finite set $\mathcal Z$ of vectors. The type of a vector $m$,  as seen in Formula \eqref{defL} is a linear combination of the elements $\mathtt j_i$ with coefficients the coordinates of $L(m)$. Hence each  $u\in \mathcal Z$ has $|u|\leq d\kappa$. Note that when $m\in \Sigma_{\GA,a}$ the element $L(m)$ is fixed (by $a$), the corresponding type will  be denoted by $u_a$.  Notice that $\Sigma_{\GA,a}=\Sigma_{\GA,\er}+u_a$. \end{remark}
\begin{remark}\label{opstrat}
Among the combinatorial graphs we have the graph $\{0\}$ formed by a single vertex.  The corresponding {\em open} stratum $\Sigma_{\{0\},0}$  obviously spans $\Z^d$, it is formed of all the points in $S^c$ which do not belong to any of the proper strata $\Sigma_\GA$.
\end{remark} 
\medskip

We have thus finitely many affine subspaces  $\langle \Sigma_{\GA,a}\rangle$ associated to the pairs $(\GA,a)$, these subspaces can be presented using the linear equations associated to the geometric realization by formulas (61) of \cite{PP}. 
We have a finite number of possible systems of equations with coefficients depending linearly or quadratically from the set $S$. We verify  that all the constant coefficients of these equations are $<\mathtt c N_0$ (in fact  the coefficients can be bound by $(2d\kappa)^2$). \label{ilpb}Thus by the bounds chosen each of these subspace lies in $\mathcal H_N,\ \forall N>N_0$ and thus has an $N$--optimal presentation $[w_i;q_i]_{d_\GA}$, furthermore each $q_i<\mathtt c N_0$.
\begin{lemma}\label{dimelle}
Assume that $m\in \Sigma_{\GA,a}$ has a cut at $\ell$ with parameters $\theta,\mu,\tau$ and associated space $[v_i,p_i]_\ell$.    Then  $[v_i,p_i]_\ell$ is contained in the affine space $\langle \Sigma_{\GA,a}\rangle=[w_i;q_i]_{d_\GA}$. In particular $\ell\geq d_{\GA}$.
\end{lemma}
\begin{proof}
Since $m$ is in both spaces  it is enough to prove that $\langle w_i\rangle_{d_\GA}\subset \langle v_i\rangle_\ell$.  By contradiction if some  $w_j\notin \langle v_i\rangle_\ell$ then we have that $|(w_j, m)|>c N^{4d\tau_0}$. This is incompatible with the estimates on the $ q_i$.
\end{proof}
\begin{theorem}\label{Lostra}
If $m,n$ have the same cut $\ell$ and the same associated affine space $[v_i,p_i]_\ell$ then they belong to the same stratum $\Sigma_{\GA,a}$ and hence have the same type, i.e. $ m-\er (m)= n-\er(n)$.  \end{theorem}
\begin{proof}
If both $m,n$ form an isolated component then they belong to the open stratum. Assume that $m\in \Sigma_{\GA,a}$  with $d_\GA>0$. Thus $m$ satisfies the equations $[w_i;q_i]_{d_\GA}$. By the previous Lemma  since $n\in [v_i,p_i]_\ell$ we know that $n\in [w_i;q_i]_{d_\GA}$. This implies  $n\in \Sigma_{\GA' }$ where $\GA'$ contains   $\GA$ so that $d_{\GA'}\geq d_\GA$.

Exchanging the roles of $m,n$ we see that $m\in \Sigma_{\GA'' }$ where $\GA''$ contains   $\GA'$.   This implies that $\GA=\GA'$.  Since the equations $[w_i;q_i]_{d_\GA}$ define the affine space spanned by $\Sigma_{\GA,a}$ the claim follows. 
\end{proof}

   \section{Functions \label{toppa}}

 \subsection{T\"oplitz approximation} \subsubsection{Piecewise T\"oplitz functionsn} Given a parameter  $ N> N_0 $ we
will call {\em low momentum variables} relative to   $N$,   denoted by $w^L$,   the $z_j^\sigma$ such that $|\er(j)|< \mathtt C N^3$. Similarly  we call {\em high momentum variables},  denoted by
$w^H$,  the $z_j^\sigma$ such
that $|\er(j)|> \mathtt c N^{\tau_1} $. Notice that  by \eqref{itau} the low and high variables are
separated. Furthermore all variables $z_j$ belonging to blocks with red edges are low by choice of the parameters.
The remaining variables will be denoted by $w^R$. Given any set $X$ of conjugate variables  by $\{\cdot,\cdot\}^X$ we mean the Poisson bracket performed only respect to the variables $X$ (keeping the other variables as parameters).
\smallskip

Our definitions will depend on several parameters which in turn depend upon the particular problem treated. We shall denote them by a compact symbol $\underline p$.

Recall first that a monomial $
e^{\ii (k, x)} y^i z^\alpha {\bar z}^{\b}
$ has  momentum $$\ii(\sum_{i=1}^nk_i\mathtt j_i+\sum_{j\in S^c}\sigma(j)\er(j)(\alpha_j-\beta_i)).$$ Thus we make a:
\begin{definition}\label{LM} {\bf (Low-momentum)}
A monomial
$
e^{\ii (k, x)} y^i z^\alpha {\bar z}^{\b}
$
is $(N,\mu)$-low momentum if\footnote{For $k\in\Z^n$, by $|k|$ we always mean the $L^1$ norm $\sum_i|k_i|$. In $\Z^d$ instead we use the $L^2$ norm.}
\begin{equation}\label{zerobis}
\sum_{j \in S^c} |\er(j)| (\alpha_j + {\b}_j ) < \mu N^3 \, ,  \quad |k| < N \, .
\end{equation}
We denote by  
$$
{\mathcal  L}_{s,r}(N,\mu) \subset {\mathcal  H}_{s,r}
$$
 the subspace of functions
\begin{equation}\label{g}
g = \sum g_{k,i,\a,\b}e^{\ii (k, x)} y^i z^\a {\bar z}^{\b} \in  {\mathcal  H}_{s,r}
\end{equation}
whose monomials are $(N,\mu)$-low momentum.
The corresponding projection
\be\label{proiLNmu}
\Pi^L_{N,\mu} : {\mathcal  H}_{s,r} \to {\mathcal  L}_{s,r}(N,\mu)
\end{equation}
is defined as $ \Pi^L_{N,\mu} := \Pi_I $ 
 where $ I $ is the subset of indexes $(k,\alpha,\beta)$
satisfying \eqref{zerobis} (notice that the exponent $i$ of $y$ plays no role).
Finally, given $h\in\Z^d$, we denote by  
$$
{\mathcal  L}_{s,r}(N,\mu,h) \subset {\mathcal  L}_{s,r}(N,\mu)
$$
the subspace  of functions of momentum $-\ii\,h$, i.e. whose monomials satisfy (cf. \eqref{mome}):
\be\label{vincula}
\pi_\er(k,\alpha,\beta) + h =0  \,.
\end{equation}
\end{definition}
By \eqref{zerobis},  
any function in ${\mathcal  L}_{s,r}(N,\mu)$, $ \cc < \mu < \CC $,  only depends on $ x, y, w^L $
and therefore
\begin{equation}\label{scamuffo}
g,g'\in {\mathcal  L}_{s,r}(N,\mu) \ \
\Longrightarrow\ \
g g' ,\  \{g,g'\},\  \{g,g'\}^{x,y}, \ \{g,g'\}^L \  \ {\rm do\ not\ depend \ on \ } w^H\,.
\end{equation}Moreover  if
\begin{equation}\label{narsete}
|h| \geq \mu N ^3+ \kappa  N \quad
\Longrightarrow \quad {\mathcal  L}_{s,r}(N,\mu,h)=0 \, .
\end{equation}
\begin{definition}\label{taubilinear}
Given $N$ and allowable parameters $\underline p=(N,\theta,\mu,\tau)$    we
say that a monomial
\begin{equation}
\label{monb}\mathfrak m=\mathfrak m_{k,l,\alpha,\beta,m,n,\s,\s'}:= e^{\ii (k,x)}y^lz^\alpha{\bar z}^\beta z_m^\sigma z_n^{\sigma'}
\end{equation}  is $\underline p=(N,\theta,\mu,\tau)$--bilinear 
 if:
 
 1)\quad  it satisfies momentum conservation, $\pi_\er(\mathfrak m)=0$, \eqref{mome}  and:
\begin{equation}\label{zerouno}
|k|<  N\,,\qquad |\er(n)|,|\er(m)| > \theta N^{\tau_1} \,,\qquad \sum_j |\er(j)|
(\alpha_j+\beta_j)< \mu N^3 \,.
\end{equation}

2)\quad  There is an $0<\ell<d$     so that both $ m,n$ have an $\ell$ cut with parameters $N,\theta,\mu,\tau$.
\smallskip

To the high variables $m,n$ of the monomial $\mathfrak m$ we associate the two affine subspaces $A,B $ associated to their $\ell$--cut.  By reordering the variables if necessary,  we may assume that $A\prec B$ and associate to the monomial, or equivalently to the pair $m,n$, only $A$.

A monomial which is  $\underline p$ bilinear with associated affine space $A$  is also called {\em $A,\underline p $ restricted}.\end{definition}

We shall often write $m\in \underline p-{\rm cut}$ to mean that there is an $0<\ell<d$     so that
  $ m$ has an $\ell$ cut with parameters $\underline p=(N,\theta,\mu,\tau)$.  If we want to stress the affine space $A$ associated to the cut we write $m\in (A,\underline p)-{\rm cut}$. \begin{remark}
 Note that by momentum conservation and Remark \ref{types}
 $$
 0=\! |\pi_\er(\mathfrak m)|\!= \!|\s \mathtt \er(m)+\s' \mathtt \er(n) + \sum_j \s(j)\mathtt \er(j) (\a_j-\b_j) +\pi(k)| \geq | \s m +\s' n| - (2 d\kappa+\mu N^3 + \kappa   N)
 $$
we deduce 
 \begin{equation}\label{limmn}
 | \s m+\s' n| \leq  \mu N^3 +3 d\kappa N.
\end{equation}
 \end{remark}
Thus the monomial $\mathfrak m=g z_m^\sigma z_n^{\sigma'}$ of Formula \eqref{monb} has $g\in {\mathcal  L}_{s,r}(N,\mu, \s \er(m) +\s'\er(n))$.

\begin{remark}\label{dueta2}Note that   under condition 1), in condition  2) it is sufficient to assume that $m,n$ have a cut with the same parameters $N,\theta,\mu,\tau$. The fact that the cut is at the same $\ell$ follows from Corollary \ref{dueta} since Formula \eqref{limmn} implies Formula \eqref{edueta}. 
 
   By Theorem \ref{Lostra}, for all $A,\underline p$ restricted monomials $\mathfrak m$ with given $\s m+\s'n ,\s,\s'$ we may deduce $\er(m)- m$ from $A$.  By Corollary \ref{dueta} and  \eqref{limmn},    we deduce  
\begin{equation}\label{ricci}
 -\s\s' A + \s'(\s m+\s' n)=B.
\end{equation} Note that, by hypothesis, $m\in A^g_{\underline p},\  n\in B^g_{\underline p}$.  $B$ in turn fixes $\er (n)- n$ and hence the {\em type} of the monomial $u(\mathfrak m)$, defined as
\begin{equation}\label{pru}
 u(\mathfrak m):= u(A, \s m+\s'n,\s,\s')= \s(\er(m)-m)+\s' (\er(n)-n),
\end{equation}
  depends only on the elements $A, \s m+\s'n,\s,\s'$.\end{remark}

  \begin{definition}
Set $\underline p= (s,r,N,\theta,\mu,\tau)$, in ${{\mathcal H}}_{s,r}$ we consider the  space ${{\mathcal B}}_{\underline p}$ of
$(N,\theta,\mu,\tau)$--bilinear functions, that is whose monomials are all $(N,\theta,\mu,\tau)$--bilinear.  We call
$$\Pi_{\underline p}:=\Pi_{(N,\theta,\mu,\tau)}:{{\mathcal H}}_{s,r}\to {{\mathcal B}}_{\underline p}$$ the projection onto this subspace. A function $f\in {{\mathcal B}}_{\underline p}$ is of the form:
\be\label{pri}
f(x,y,z,\bar z)= \sum_{\s,\s'=\pm} \quad \sum_{|\er(m)|,|\er(n)|>\theta N^{\tau_1}\,,\;  \atop m,n\in \underline p-{\rm cut}  }f_{m,n}^{\s,\s'}(x,y,w^L) z_m^\s z_n^{\s'}
\ee
with $f_{m,n}^{\s,\s'}(x,y,w^L)\in \mathcal L_{s,r}(N,\mu,\s \er (m)+\s' \er (n))$.

By convention we assume $f_{m,n}^{\s,\s'}(x,y,w^L)=f_{n,m}^{\s',\s}(x,y,w^L)$.

Note that, by    Definition \ref{taubilinear},  to each element $f_{m,n}^{\s,\s'}(x,y,w) z_m^\s z_n^{\s'}$  is associated an affine subspace $A$.
\end{definition}

\begin{remark}\label{varlm}
Of course, if    we take less restrictive parameters $\theta',\mu'$   with $\theta\leq
\theta',\ \mu'\leq \mu$ we have that the set of  $(N,\theta,\mu,\tau)$--bilinear monomials  contains  the set of  $(N,\theta',\mu',\tau)$--bilinear monomials.  In particular we have, for each $s,r$:
$$\Pi_{(N,\theta',\mu',\tau)}\Pi_{(N,\theta,\mu,\tau)}=\Pi_{(N,\theta,\mu,\tau)}\Pi_{(N,\theta',\mu',\tau)}=\Pi_{(N,\theta',\mu',\tau)}\,. $$ 
\end{remark} 
 
 \vskip20pt

\begin{definition}\label{topa} Given parameters $\underline p$ and an affine space $A\in\mathcal H_{\underline p}$ (cf. Definition \ref{isotb}), 
the space $\mathcal T^{\sigma,\sigma'}_{A,\underline p}$  of  $A,\underline p$-restricted {\em
T\"oplitz} bilinear functions  of signature $ {\sigma,\sigma'}$  is formed by the functions $g\in{\mathcal B}_{\underline p}$ where the coefficients of Formula \ref{pri} satisfy:
\begin{equation}\label{nzo}
g  =\sum_{m,n}^{(A,\underline p)} {\mathtt g}(\s m+\s' n) z_m^\s z_n^{\s'}\,.\quad \end{equation}  
The apex $(A,\underline p)$, with $\underline p={(N,\theta,\mu,\tau)}$, means that the sum is on the $A,\underline p$--restricted monomials (cf. \ref{taubilinear}), with bilinear part $z_m^\s z_n^{\s'}$. Notice that, by definition of  $A,\underline p$--restricted monomials, the space 
$\mathcal T^{\sigma,\sigma'}_{A,\underline p}$ is non--zero only if $A\frec{N}[v_i;p_i]_\ell $ is of codimension $\ell$ with $0<\ell<d$ and $p_\ell <\mu N^\tau$.  
 
For all $h= \s m+\s' n $ we have: 
\begin{equation}\label{ceci}
\mathtt g(h)\in \mathcal L_{s,r}(N,\mu, h + u(A,h,\s,\s') ) 
\end{equation}
 where $u(A,h,\s,\s')$ is the type, see Formula \eqref{pru}. 
 \end{definition}  
\begin{remark}
i) Notice that we have a translation invariance property (which justifies the name restricted T\"oplitz). Indeed given $A,\s,\s',h$ one can choose arbitrarily an element $\mathtt g^{\s,\s'}(A,h)$ satisfying \eqref{ceci}, and use formula \eqref{nzo} to define a  function  in ${{\mathcal B}}_{\underline p}$. One easily sees that indeed such an expression defines a function in $\mathcal H_{s,r}$.

\noindent ii) Note that condition \eqref{ceci} implies that $\mathtt g(s m+\s' n)$  in \eqref{nzo} has momentum $\ii(\s \er(m)+\s'\er(n))$ hence $g$ has zero momentum (as required).

 \end{remark}

  Finally we define
   \begin{definition}\label{topp}
 The space  $\mathcal T_{\underline p }$  of  {\em piecewise 
T\"oplitz} bilinear functions  \begin{equation}\label{stro} g= \sum_{A\in  \mathcal H_{N},\sigma,\sigma'=\pm 1} g^{\sigma,\sigma'}(A),\quad  \ g^{\sigma,\sigma'}(A)\in  \mathcal T^{\sigma,\sigma'}_{A,{\underline p}}. \end{equation}
\end{definition}
Of particular significance are the  piecewise 
T\"oplitz diagonal  functions \begin{equation}\label{diagtf}
 {\mathcal Q}(z)= \sum_{A \in \mathcal H_N }\sum^{(A,\underline p)}_{m } \mathcal Q(A) z_{m}\bar z_m=  \sum_{A \in \mathcal H_N }\sum_{m\in  (A,\underline p)-{\rm cut} } \mathcal Q(A) z_{m}\bar z_m\,,\end{equation} in this formula  the elements $m$ run over  all vectors which have an $\ell$-cut with $0<\ell<d$ and $A$ is the corresponding affine spaces. 
 \medskip

By definition  $\mathcal T_{\underline p }\subset \mathcal B_{\underline p }$ is a subspace of the  $(N,\theta,\mu,\tau)$ bilinear
functions. 
 
\begin{lemma}\label{poisb}
Consider   $ f, g \in {\mathcal T}_{\underline p}$ and  $q \in {\mathcal  L}_{s,r}(N,\mu_1,0) $,
$\cc<\mu,\mu_1<\CC$. 
Given any $\underline p'= (s',r',N,\teta',\mu',\tau)$ with
$s/2<s'<s\,, \, r/2<r'<r$,  $ \teta'  \geq \teta , \mu' \leq \mu $, one has
\be\label{pro11}
\Pi_{N,\teta',\mu',\tau} \{ f , q \}^L \, ,
\ \Pi_{N,\teta',\mu',\tau} \{ f , q \}^{x,y}
  \in {\mathcal T}_{\underline p'} \, .
\end{equation}
If moreover 
\begin{equation}\label{ugo}\kappa (\mu N^3+3d\kappa N) <( \teta' - \teta) N^{4d\tau-1}, (\mu-\mu')N^{\tau -1}
\end{equation}
then \be\label{pr2}
\Pi_{N,\teta',\mu',\tau} \{f,g \}=\Pi_{N,\teta',\mu',\tau} \{f,g \}^H \in {\mathcal T}_{\underline p'} \, .
\end{equation}

\end{lemma}

\begin{proof}
 Write $ f \in \mathcal T_{\underline p} \subset \mathcal B_{\underline p}$ as in \eqref{pri} where
\be
f_{m,n}^{\s,\s'}= f^{\s,\s'}(A,\sigma m+\sigma' n)
\in{\mathcal L}_{s,r}(N,\mu,\sigma m+\sigma' n+u)={\mathcal L}_{s,r}(N,\mu,\sigma \er(m)+\sigma' \er(n)) \, ,
\end{equation}
similarly for $ g $ (recall that $u= u(A,\sigma m+\sigma' n ,\s,\s')$), with $f_{m,n}^{\s,\s'}=f_{n,m}^{\s',\s}$.
\\[1mm] {\sc Proof of \eqref{pro11}.}
Since the variables $ z_m^\s $, $ z_n^{\s'} $, $|\er(m)|,|\er(n)|> \teta N^{\tau_1} $, are high momentum,
$$
\{f^{\s,\s'}(A,\sigma m+\sigma' n)z_m^\s z_n^{\s'}\,, \, q  \}^L
=
\{f^{\s,\s'}(A,\sigma m+\sigma' n), \, q  \}^L \, z_m^\s z_n^{\s'}.
$$
The function  $\{f^{\s,\s'}(A,\sigma m+\sigma' n)\,, \ q  \}^L$ in in $\mathcal H_{s',r'}$ by \eqref{commXHK} and does not depend on
 $ w^H $ by  \eqref{scamuffo}. 
   Hence the coefficient of
 $z_m^\sigma z_n^{\sigma'}$ in $\Pi_{N,\teta',\mu'}\{f,q\}^{L}$ is,
$$
 \Pi^L_{N,\mu'}\{f^{\s,\s'}(A,\s m +\s' n)\,, \, q  \}^L
\in
{\mathcal L}_{s',r'}(N,\mu',\s \er(m)+\s'\er(n))
$$
since $\pi_\er(q)=0$, $\pi_\er( f^{\s,\s'}(A,\s m +\s' n))= - \s \er(m)-\s'\er(n)$.

The proof that
$\Pi_{N,\teta',\mu'}\{f,q\}^{x,y}
 \in \mathcal{T}_{s',r',N, \teta', \mu'}$
is analogous.
\\[1mm]
{\sc Proof of \eqref{pr2}.}
A direct computation gives
$$
\{ f , g\}^H=\sum_{|\er(m)|,|\er(n)|> \teta N^\tau_1,\atop  m,n \in \underline p-{\rm cut}}\sum^{\s,\s'=\pm} p_{m,n}^{\s,\s'}z_m^\s z_n^{\s'}
$$
with
\begin{equation}\label{valentiniano}
p_{m,n}^{\s,\s'}=
-2\ii \sum_{l\,,\ \s_1=\pm}
\s_1 \Big(
f_{m,l}^{\s,\s_1} g_{l,n}^{-\s_1,\s'} +
f_{n,l}^{\s',\s_1} g_{l,m}^{-\s_1,\s}
\Big)\,,
\end{equation} 
of course $l$ gives a contribution only if suitably restricted by the bilinearity constraint.

By \eqref{scamuffo} the coefficient $p_{m,n}^{\s,\s'}$ does not depend on $w_H.$
Therefore
\begin{equation}\label{giustiniano}
\Pi_{N, \teta',\mu'} \{ f , g\}^H=
\sum_{|\er(m)|,|\er(n)|> \teta' N^{\tau_1},\, \s,\s'=\pm} q_{m,n}^{\s,\s'}z_m^\s z_n^{\s'}
\quad \mbox{with} \quad
q_{m,n}^{\s,\s'} := \Pi^L_{N,\mu'} p_{m,n}^{\s,\s'}\,.
\end{equation}
 It results
$ q_{m,n}^{\s,\s'}\in {\mathcal L}_{s',r'} (N,\mu',\s \er(m) +\s' \er(n)) $
by \eqref{giustiniano}, \eqref{valentiniano},  and  momentum conservation.

It remains to prove the $ (N,\teta',\mu') $-T\"oplitz property:     \begin{equation}\label{valente}
q^{\s,\s'}_{m,n} =  q^{\s,\s'}\big(A,\s m+\s' n\big)
\quad
{\rm for\ some } \quad  q^{\s,\s'}(A, h)\in{\mathcal L}_{s,r}(N,\mu',h+u))\,,
\end{equation} where $A$ is the affine space associated to the pair $m,n$ (cf. \ref{taubilinear}, 2)).\smallskip

  Let us consider  in \eqref{valentiniano}-\eqref{giustiniano} the term with  $m,n$ fixed  and $\s=+1,\s'= -1, \s_1= +1$ (the other cases are analogous)
\be\label{sommal}
\Pi^L_{N,\mu'}
\sum_{l}
f_{m,l}^{+,+} g_{l,n}^{-,-}\,,
\end{equation}
 Since by hypothesis $f,g\in {\mathcal T}_{\underline p}$ we have that  the  $l$ that give a contribution have   a $\underline p$-cut at $\ell$ and $|\er(l)|>\theta N^{\tau_1}$. By Remark \ref{dueta2}  the  affine space associated to the cut of $l$ is $A'=-A +l+m$; (while the affine space  associated to the cut of $n$ is $B=A +n-m$).  We may assume without loss of generality that $A\prec B$.

We then divide
\eqref{sommal} in 3 parts according to the relative position of $A'$ in the $\prec $ order with respect to  $A\prec B$. Note that all these constraints depend only upon $A, m-n,j:=l+m$.  We   treat the case  $A\prec A'\prec B$, the other 2 cases are similar.

Since $f,g\in {\mathcal T}_{\underline p}$ we have  
\be\label{oggi}
f_{m,l}^{+,+}=f^{+,+}\big( A, m+ l \big) \in {\mathcal L}_{s,r}(N, \mu,  \er(m)  + \er( l))
\end{equation}
\begin{equation}\label{lalla}
g_{l,n}^{-,-}
=g^{-,-}\big(A',- l- n \big) \in {\mathcal L}_{s,r}(N, \mu, -\er( l) -\er(n)) \,,
\end{equation}
for all $m,n,l$ which satisfy the bilinearity constraints with parameters $\underline p$. 

By construction $m,n$ satisfy the bilinearity constraints with the more restrictive  parameters $\underline p'$. Set  $j:= m+l$,  by formula \eqref{limmn} we have that the elements $j$, which come from the elements $l$ which contribute,  have $|j|< \mu N^3+3d\kappa N$.  On the other hand  
 by condition \eqref{ugo}  we see that these $j$  satisfy: \begin{equation}
\label{vicin2}|j|<{\kappa}   ^{-1}(\mu-\mu')N^{\tau-1},\ {\kappa}   ^{-1}(\theta'-\theta)N^{4d\tau-1} .
\end{equation}
  thus, by Lemma \ref{mah} the fact that $j$  satisfies \eqref{vicin2} implies that $m,n,l$ satisfy the bilinearity constraints with parameters $\underline p$. 

Then
\begin{eqnarray*}
\Pi^L_{N,\mu'}
\sum_{A\prec A'\prec B}
f_{m,l}^{+,+} g_{l,n}^{--}
&\stackrel{\eqref{ricci}}=&
\Pi^L_{N,\mu'}
\sum_{j\in\Z^d: |j|< \mu N^3+3d\kappa N \atop
A\prec A+j\prec A+n-m }
f^{++}\big(A, j \big)
g^{--}\big(- A + j,m-n-j \big)
\end{eqnarray*}
depends only on $A$ and $m - n$, i.e. \eqref{valente}.
\end{proof}

\subsection{Quasi--T\"oplitz functions}
 
Given $f\in {\mathcal H}_{s,r}$  and $\mathcal F\in\mathcal T_{\underline p }$, we
define     \begin{equation}\label{limi} \bar f=\bar f(\mathcal F):=N^{4
d\tau}\big( \Pi_{(N,\theta,\mu,\tau)}f- \mathcal F\big),
\end{equation}
and  for $\Pappa > N_0$  set   \begin{equation}\label{unobisbis} \|X_f \|_{s,r}^{  {\Pappa},\theta,\mu} 
:= \sup_{N,\tau : N\geq {\Pappa}\atop
\underline p=(s,r, N,\theta,\mu,\tau) \ }[\inf_{ {\mathcal F}\in\mathcal  T_{\underline p}}( \max(\| X_{\mathcal F}\| _{s,r},\|
X_{\bar f}\|  _{s,r} )) ] 
 \end{equation}\begin{remark}\label{varlm1}
If    we take new parameters ${\Pappa}', \theta',\mu' $   with ${\Pappa}\leq {\Pappa}', \theta\leq
\theta',\ \mu'\leq \mu$ we have by Remark \ref{varlm} that  $$ \|X_{f }  \|_{s,r}^{{\Pappa}',\theta',\mu'} \leq \|X_{f }  \|_{s,r}^{{\Pappa},\theta,\mu}. $$
\end{remark}

 \begin{definition} \label{topbis}
We say that $f\in {\mathcal H}_{s,r}$  is {\em quasi- T\"oplitz} of  parameters
$({\Pappa},\theta,\mu)$ if  $ \|X_f\|_{s,r}^{{\Pappa},\theta,\mu }<\infty$ and we call this number  the {\em quasi-T\"oplitz norm} of $f$.

 \end{definition}
 \begin{remark}
 Given $f \in {\mathcal H}_{s,r}$  with finite quasi-T\"oplitz norm and  parameters $\underline p$ we say that  a function $\mathcal F\in\mathcal T_{\underline p }$,  approximates $f$ at order $\e$ if
 \begin{equation}
 \|X_{\mathcal F}\|_{s,r}, N^{4
d\tau}\|X_{\Pi_{(N,\theta,\mu,\tau)}f- \mathcal F}\|_{s,r} < (1+\e) \|f\|^{K,\theta,\mu}_{s,r}.
\end{equation}
Note that by our definitions such approximating functions exist for all allowable parameters $\underline p$ and for all positive $\e$.  
\end{remark}
 
Since in our algorithm we deal with functions which depend in a Lipschitz way on some parameters $\xi \in \mathcal O$ (a compact set) we take finally a norm which includes also the weighted  Lipschitz norm (cf. \eqref{weno})  with $\lambda$ a positive number: 
\begin{definition}
 We set $\vec p= (s,r,\PaPi,\theta,\mu,\lambda,\mathcal O)$.
 We define
 \begin{equation}\label{finale}
 \|X_{f }  \|^T_{\vec p}:= \max(\|X_{f }  \|_{s,r}^{  {\Pappa},\theta,\mu}, \| X_f\|_{s,r}^\lambda)
\end{equation}

We denote by $\mathcal Q^T_{\vec p} \subset {\mathcal H}_{s,r}$ the set of functions with finite norm $ \|X_{f }  \|^T_{\vec p}$.
\end{definition}

\begin{remark}\label{dupro}

 Notice that our definition includes the T\"oplitz and anti-T\"oplitz
functions, setting $\mathcal F = \Pi_{(N,\theta,\mu,\tau)} f$ and
hence $\bar f=0$.
  In the case of  T\"oplitz functions one trivially has $\|X_f\|_{\vec p }^{T }= \| X_f\|^\lambda_{s,r}$.\smallskip
  
  The definition includes also functions with   fast decay in the coefficients so that, taking always as  T\"oplitz approximation $\mathcal F=0$,  we still have   $\sup_NN^{4
d\tau}\| X_{\Pi_{(N,\theta,\mu,\tau)}f}\|^\lambda_{s,r}<\infty$.
\end{remark}
 
 \subsubsection{Some basic properties} 
The following Lemmas are proved in \cite{BBP1}.

\begin{lemma}{\bf (Projections 1)}\label{proJ1}
Set $\vec p= (s,r,K, \teta, \mu,\lambda,\mathcal O)$ and $\underline p= (s,r,N,  \teta, \mu,\tau)$ with $N\geq K$. Consider a subset of monomials $ I $  
such that the projection (see \eqref{caligola}) maps
\be\label{tototo}
\Pi_I :  \Ta_{\underline p}  \to  \Ta_{\underline p}
  \, , \quad \forall N \geq  \PaPi \, .
\ee
Then $ \Pi_I :  {\mathcal Q}^T_{\vec p}  \to  {\mathcal Q}^T_{\vec p}  $ and
\be\label{lessore}
\| X_{\Pi_I F} \|_{\vec p} ^T \leq \| X_F \|_{\vec p}^T  \, , \quad \forall   F \in {\mathcal Q}^T_{\vec p}  \, .
\ee
Moreover, if $  F\in {\mathcal Q}^T_{\vec p}  $ satisfies
$\Pi_I F =F $, then, $ \forall N \geq  \PaPi $, $ \forall \e > 0 $,
there exists a  decomposition
$\Pi_{N,\theta,\mu,\tau} F=
\tilde F+ N^{-4d\tau}\hat F$
with a
 T\"oplitz approximation $\tilde F\in \Ta_{\underline p}  $
satisfying  $\Pi_I \tilde F = \tilde F $,
 $\Pi_I \hat F = \hat F $
 and $ \|  X_{\tilde F}\|_{s,r },  \|  X_{\hat F}\|_{s,r } <   \|X_F\|^T_{\vec p} +\e $.
 \end{lemma}
\begin{lemma} {\bf (Projections 2)}\label{proJ}
For all $ l \in \N $,  $ K \in \N $, $ N \geq  \PaPi $, the projections
\begin{equation}\label{margherita}
 \Pi^{(l)},  \Pi_{|k| < K}, \Pi_{\rm diag} : \Ta_{\underline p} \to \Ta_{\underline p} \, .
\end{equation}
here $ \Pi^{(l)}$ maps to the space of homogeneous functions of degree $l$,  $\Pi_{\rm diag}:= \Pi^{(2)} \Pi_{ k=0} $.

If $ F\in {\mathcal Q}^T_{\vec p}  $
 then,
\be\label{fhT}
\| X_{\Pi^{(l)} F} \|^{T}_{\vec p}  \, , \ \|  X_{\Pi_{|k| < K}F} \|^T_{\vec p}  \, , \ \| X_{\Pi_{\rm diag} F} \|^T_{\vec p}  \leq \|  {X_F} \|^T_{\vec p}  \, ,
\ee
\be\label{fhle2T}
\| {X_{F^{(\leq 2)}}} \|^{T}_{\vec p}  \, , \ \| {X_F} - {X_{F^{(\leq 2)}_{|k|<K}}} \|^{T}_{\vec p}  \leq  \| {X_F} \|^{T}_{\vec p}  \, .
\ee
Moreover, $ \forall\, 0 < s' < s $, set $\vec p\,{}':=(s',r ,  \PaPi, \teta,\mu,\lambda,\mathcal O)$:
\begin{equation}\label{smoothT}
\| X_{\Pi_{|k|\geq K}F} \|_{\vec p\,{}'}^T
\leq
e^{-K(s-s')}\frac{s}{s'} \|  {X_F} \|_{\vec p}^T
\end{equation}
Finally, $ \forall\, 0 < r' < r $, set $\vec p\,{}':=(s,r' ,  \PaPi, \teta,\mu,\lambda,\mathcal O)$:
\begin{equation}\label{smoothT2}
\| X_{\Pi^{(l \geq D)}F} \|_{\vec p\,{}'}^T
\leq
(\frac{r'}{r})^{D-2} \|  {X_F} \|_{\vec p}^T
\end{equation} 
\end{lemma}

\begin{lemma}\label{diago} 
Let $ Q(z)=\sum_mQ_mz_{m}\bar z_m$ be a  quasi-T\"oplitz
 {\em diagonal quadratic function} in the variables $z,\bar z$ with constant coefficients.
For all allowable choices of $\underline p$ and $\e>0$,  there exists a  {\em diagonal quadratic function} ${\mathcal Q}(z)\in \mathcal T_{\underline p}$ which approximates $Q$ at order $\e$:
\begin{equation}\label{pargol}
 {\mathcal Q}(z)= \sum_{A \in \mathcal H_{\underline p} }\sum^{(A,\underline p)}_{m } \mathcal Q(A) z_{m}\bar z_m\,,\end{equation} 
so that setting   $\bar Q(z)$: $$  N^{-4d\tau} \bar Q(z) = \Pi_{(N,\theta ,\mu,\tau)}Q(z)- \mathcal Q(z) \,,$$
for all $m$ which have a cut at $\ell$ with parameters $(N,\theta ,\mu,\tau)$   associated to $A$ one has\begin{equation}\label{appr1}
 Q_m = \mathcal Q(A) + N^{-4d \tau} \bar Q_m ,
\end{equation}
 and
\begin{equation}\label{appr2}
 |Q_m|,|\mathcal Q(A)|,|\bar Q_m|\leq (1+\e) |X_Q|^T_{\vec p}\,.
\end{equation}
For all $m,m'$ with a cut at $\ell$ with parameters $(N,\theta ,\mu,\tau)$ associated to $A$ we have
 \begin{equation}\label{lastima}
|Q_m  -  Q_{m'}| \leq N^{ -4d\tau} 2\|X_Q\|^{T  }_{\vec p }.
\end{equation}
\end{lemma}
\begin{proof}
Since  $Q$ is quasi-T\"oplitz  we may approximate it by a function $\mathcal F\in \mathcal T_{\underline p} $; moreover since $Q$ is quadratic and diagonal by the previous discussion we may choose $\mathcal F$ of the same form.

Hence we can we can find a  quadratic and diagonal function $\mathcal Q\in \mathcal T_{\underline p} $   so that, with  $\bar Q= N^{4d\tau} (\Pi_{N,\theta ,\mu,\tau} Q- \mathcal Q)$,  we have $\|X_{\mathcal Q}\|_r, \|X_{\bar Q}\|_r\leq 2 \|X_Q\|^T_r$.
To conclude, by Formula \eqref{diagtf}, we have that a quadratic, diagonal and piecewise T\"oplitz  $\mathcal Q$ is of the form \eqref{pargol}.  

Our last statement is proved by  noting that (see \eqref{normadueA} for the norm of a vector field)
$$
 \|X_{Q}\|_{r}^2= 2 \sup_{\|z \|_{a,p}\leq r}\sum_{h\in S^c}
 |Q_h|^2 \frac{|z_h|^2}{r^2} e^{2a|h|} |h|^{2p} \geq |Q_j|^2
$$
by evaluating at  $z^{(j)}_h:=   \delta_{jh} e^{-a |j|}|j|^{-p}r/\sqrt{2}$.
The same holds for $\mathcal Q$ and $\bar Q$.
\end{proof}
\begin{corollary}\label{pirlino}
 Let $Q$ be as above. Given parameters $\underline p$ for every $A\in \mathcal H_{\underline p}$ choose a point $m_A\in A^g_{\underline p}$.  The the function ${\mathcal Q}(z)\in \mathcal T_{\underline p}$ defined as in formula \eqref{pargol} with $\mathcal Q(A)= Q_{m_A}$ is a T\"oplitz approximant of order $\e=1$ for $Q$.
\end{corollary}

 \section{Poisson bracket}
\subsection{Poisson bracket estimate}
Analytic quasi-T\"oplitz functions are closed under Poisson bracket and respect the Cauchy estimates.

More precisely fix allowable $\vec p=(s,r,{\PaPi},\theta,\mu,\lambda,\mathcal O)$:
\begin{proposition}\label{main}

(i) Given  $f^{(1)},f^{(2)}\in\mathcal Q^T_{\vec p }$, quasi-T\"oplitz with parameters $\vec p$    we have that  $\{f^{(1)},f^{(2)}\}\in Q^T_{\vec p\,{}'}$,   for all  allowable parameters $\vec p\,{}':=(s',r',{\PaPi}',\theta',\mu',\lambda,\mathcal O)$, with $\vec p\,{}'$  satisfying :
$$  r/2< r'<r,\quad s/2<s'< s  ,\quad { 2\kappa} < (\mu-\mu'){{\PaPi}'}^{2} \,,\;    \kappa\mathtt C  <  (\theta'-\theta) { {\PaPi}'}^{ 4d\tau_0-4},$$ \begin{equation}\label{pois1}
  e^{-(s-s'){ {\PaPi}'}} {{\PaPi}'}^{{\tau_1}}<1.
\end{equation}We have the bound

 \begin{equation}\label{poisbound}
 \| X_{\{f^{(1)},f^{(2)}\}}\|^T_{\vec {p'}} \leq 12\, 2^{2n+3}\delta^{-1} \|X_{f^{(1)}}\|^T_{\vec { p}}\|X_{f^{(2)}}\|^T_{\vec { p}}
\end{equation}  where $\delta := \min(1-\frac{s'}{s},1-\frac{r'}{r})<1$.\vskip5pt
 (ii) Given $f^{(1)},f^{(2)}$  as in item (i),   assume that  
 \begin{equation}\label{funm}
3\, 2^{2n+8}  \, e   \delta^{-1}    \|X_{f^{(1)}}\|^T_{\vec { p}}< 1/2,
\end{equation} then the function  $f^{(2)}\circ \phi_{f^{(1)}}^t:= e^{t \,ad(f^{(1)} ) }(f^{(2)})$, for $|t|\leq 1$, is  quasi-T\"oplitz for the parameters $\vec p\,{}'$. More precisely  $f^{(2)}\circ \phi_{f^{(1)}}^t\in\mathcal Q^T_{\vec p\,{}' }$  for all parameters $({ {\PaPi}'},\theta',\mu')$ for which
\begin{equation}\label{nubo}  e^{-(s-s')\frac{ {\PaPi}'}{(\ln {\PaPi}')^2}} {{\PaPi}'}^{{\tau_1}}<1\,,\quad { 2\kappa} < (\mu-\mu'){{\PaPi}'}^{2} \ln(\PaPi')^{-2}\,,\;    \kappa\mathtt C  <  (\theta'-\theta) { {\PaPi}'}^{ 4d\tau_0-4}\ln(\PaPi')^{-2},
\end{equation}
Finally we have
\begin{equation}\label{flussoT}
 \|X_{f^{(2)}\circ \phi_{f^{(1)}}^t}\|^T_{\vec  p\,{}'}\leq2\|X_{f^{(2)}}\|_{\vec p}^T 
\end{equation}

\begin{equation}
\label{flussoT1}
 \|X_{f^{(2)}\circ \phi_{f^{(1)}}^t}-X_{f^{(2)}}\|^T_{\vec  p\,{}'}\leq 
2 \|X_{f^{(1)}}\|^T_{\vec { p }}\|X_{f^{(2)}}\|_{\vec p}^T 
\end{equation}

\end{proposition}
\begin{proof}

  In order to prove this proposition we need some preliminaries.\end{proof}

\subsection{A technical Lemma}\quad
We take allowable parameters $\underline p:=(N,\theta,\mu,\tau)$  and $\underline p':=(N,\theta',\mu',\tau)$ such that 
  \begin{equation}\label{pois2}   
 {\kappa} \mathtt C<   (\theta'-\theta)N^{4d\tau-4} \,,\quad   2\kappa< (\mu-\mu')N^2 .
\end{equation}

  Let us  set up some notation.   In Definition \ref{LM} we have introduced the notion of 
   a monomial  $\mathfrak m=e^{i(k,x)}y^l z^\alpha\bar z^\beta$ of low momentum with  respect to the parameters $\underline p$ and denoted by $\Pi^L_{N,\mu}$ the projection on this subspace. Recall that one of the conditions is also that $|k|<N$, that is it has also a {\em low frequency}.

 We shall say instead that  $\mathfrak m$ is of $N$--{\em high frequency} if $|k|\geq N$ and denote $\Pi^U_N$ the projection on this subspace.

We denote its degree  in the  high   variables to be 
 $ d_H(\mathfrak m) $. We further  set  $$m (M):=\sum |\er(j)|(\alpha_j+\beta_j),\quad m_L(M):=\sum_{j\ \text{low}} |\er(j)|(\alpha_j+\beta_j).$$
 
 The projection symbol $\Pi_{N,\theta,\mu,\tau}$  is given in definition \ref{taubilinear}.

We use a mixed decomposition   $f= \Pi_{N,\theta,\mu,\tau} f+ \Pi_{N,3\mu'}^L f + \Pi_N^U f + \Pi_{R} f$ where $ \Pi_{R} f$ is by definition the projection on those monomials which are neither $(N,\theta,\mu,\tau)$ bilinear nor of $(N,3\mu')$-low momentum nor of $N$-high frequency. 

We hence divide the Poisson bracket in four terms:
 $\{\cdot,\cdot\}= \{\cdot,\cdot\}^{y,x}+\{\cdot,\cdot\}^{L}+\{\cdot,\cdot\}^{H}+\{\cdot,\cdot\}^{R}$ where the apices identify the variables in which we are performing the derivatives.

We need the following   technical lemmas and definitions.

\begin{lemma}
\label{incom} Consider a monomial $M=e^{\ii \langle k,x\rangle}A z_m^\sigma=e^{\ii \langle k,x\rangle}y^jz^\alpha\bar z^\beta z_m^\sigma$ which, with respect to some allowable parameters $\underline p$,  has   $|k|<N,$   $m(A)<\mu N^3+\kappa N$. Assume that   $z_m$ is a {\em high variable} i.e. $|\er(m)|>\mathtt c N^{\tau_1}$. Then $M$ cannot satisfy conservation of momentum.
\end{lemma} \begin{proof}
If we have conservation of momentum $\pi(k)+\sum_j\sigma(j)\er(j)(\alpha_j-\beta_j)+\sigma \er(m)=0$. By the hypothesis and the  triangle inequality we have  \begin{equation}
\label{nohl}\mathtt c N^{\tau_1 }  < |\er(m)|\leq m(A )+ \kappa N<\mu N^3+ 2\kappa N  <(\mathtt C+\kappa)N^3
\end{equation}
 a contradiction to Formula \eqref{itau}.
\end{proof}
In order to simplify the notations let us set  $\Pi':=\Pi_{N,\theta',\mu',\tau},\ \Pi:= \Pi_{N,\theta,\mu,\tau} $. Assume $f^{(1)},f^{(2)} $ are two functions satisfying conservation of momentum.
\begin{lemma}
  The following splitting formula holds:
\begin{equation}\label{split}
\Pi'\{f^{(1)},f^{(2)}   \}=
\Pi' (  \{\Pi
f^{(1)},\Pi f^{(2)}\}^H \end{equation}
$$+ \{\Pi  f^{(1)},\Pi^L_{N,3\mu'}
f^{(2)}\}^{y,x}+\{\Pi
 f^{(1)},\Pi^L_{N,3\mu'} f^{(2)}\}^L$$ $$  \{\Pi^L_{N,3\mu'}
f^{(1)},\Pi
f^{(2)}\}^{y,x}+\{\Pi^L_{N,3\mu'}
f^{(1)},\Pi f^{(2)}\}^L+ $$$$ + \{\Pi^U_N
f^{(1)},f^{(2)}\}+\{f^{(1)},\Pi^U_N
f^{(2)}\}  -\{\Pi^U_Nf^{(1)},\Pi^U_N
f^{(2)}\} ) $$
\end{lemma}
\begin{proof}
We 
perform a case analysis: we replace each $f^{(i)}$ with  a single
monomial to show which terms may contribute  non trivially to the
projection $\Pi'\{f^{(1)},f^{(2)}   \}$.

 Consider the expression
  \begin{equation}
\label{ilP} \Pi'\{e^{i (k^{(1)},x)}y^{l^{(1)}} z^{\alpha^{(1)}}\bar z^{\beta^{(1)}},e^{i (k^{(2)},x)}y^{l^{(2)}} z^{\alpha^{(2)}}\bar z^{\beta^{(2)}}\}.
\end{equation}
 If one or both of the $|k^{(i)}|>N$ then one or both monomials are of high frequency and we obtain a term in the last line of \eqref{split}.

 Suppose now that  $|k^{(1)}|,|k^{(2)}|\leq N$ we wish to understand under which conditions on the $\alpha^{(i)},\beta^{(i)}$ the expression \eqref{ilP} is not zero. 
  
 For a monomial $M:=e^{i(k,x)}y^az^{\alpha }\bar z^{\beta }$  if $\Pi'(M)\neq 0$  we must have  $d_H(M)=2 $ (plus further conditions).    For two monomials $M_1,M_2$  we see that each term of $\{ M_1,M_2\}$  has as degree   $d_H(\{ M_1,M_2\})$ equal to:
 \begin{enumerate}
\item $ d_H(M_1)+d_H(M_2)-2 $ if we have contracted conjugate high variables $z_j^\s, z_j^{-\s}$.
\item $ d_H(M_1)+d_H(M_2) $ otherwise.
\end{enumerate}

In case  i) in order to have  $\Pi'\{ M_1,M_2\}\neq 0$ we must have $  d_H(M_1)+d_H(M_2)=4$, this happens        either if  a) $d_H(M_1)= d_H(M_2)=2$  or b) $d_H(M_1)=1,d_H(M_2)=3$ (resp. $d_H(M_1)=3,d_H(M_2)=1$). Let us show that, by conservation of momentum,   case b) is not possible.  Let us denote by $A_i, i=1,2$ the part of the monomials $M_i$ in the variables $z_h$ which are not  high i.e. $M_1=e^{i(k_1,x)}y^{a_1}A_1 z_j^{\pm},\ M_2=e^{i(k_2,x)}y^{a_2}A_2 z_j^{\mp}z_m^\s z_n^{\s'}$. In $\{ M_1,M_2\}$  the part of the monomial  in the variables $z_h$ which are not  high is $A_1A_2$  so, if $\Pi'\{ M_1,M_2\}\neq 0$ we must have $m(A_1A_2)=m(A_1)+m(A_2)<\mu' N^3$, we are thus  in the hypotheses of Lemma \ref{incom} for $M_1$  a contradiction.
  
 In case a)  we claim that both monomials $M_1,M_2$ are $N,\mu,\theta,\tau$ bilinear, so that this contribution comes from the first line of  \eqref{split}.  For this we only need to verify that, if $z_j$ is the variable we contract, then $j$ has a cut at  $\ell$   for the parameters $N,\theta,\mu,\tau$.   Write $M_1=e^{i(k_1,x)}y^{a_1}A_1 z_j^{\pm}z_m^\s,\ M_2=e^{i(k_2,x)}y^{a_2}A_2 z_j^{\mp}z_n^{\s'}$. Since $\Pi'\{ M_1,M_2\}\neq 0$ we have $m(A_1)+m(A_2)\leq \mu' N^3$. By conservation of momentum $| \er(j)\pm \er(m)|\leq \mu' N^3+\kappa N$ hence $|   j \pm  m |\leq \mu'N^3+\kappa N+4d\kappa$,  using \eqref{pois2}  we have 
 $$|   j \pm  m |\leq \mu'N^3+\kappa N+4d\kappa<  \mu'N^3+2\kappa N<  \mu N^3.$$
   We have that  $m$ has a $\ell$ cut for the parameters $\mu',\theta',\tau$ and the hypotheses in Formula \eqref{vicin} of Lemma \ref{mah}  
 are satisfied,  hence $j$ has the cut and 
we obtain the first term in formula \eqref{split}.\medskip

In case ii)   we can have $d_H(M_1)+d_H(M_2)=2$ either if a) $d_H(M_1)=2,d_H(M_2)=0$ (resp.   $d_H(M_1)=0,d_H(M_2)=2$)   or b) $d_H(M_1)= d_H(M_2)=1$.

We claim that in case a) we obtain the contributions of lines 2,3 of Formula \eqref{split}.

In fact say that $d_H(M_1)=2,d_H(M_2)=0$ and $M_1= e^{i(k_1,x)}y^aA_1z_m^{\s } z_n^{\s' } ,\ M_2=e^{i(k_2,x)}y^bA_2$ where $A_i$ do not contain high  variables. We have that $|k_i|<N$ by hypothesis,  the high variables of $M_1$  have a $N,\mu',\theta',\tau$ cut  also by hypothesis and so also a  $N,\mu,\theta,\tau$ cut, finally if we contract variables $x,y$  we have $m(A_1), m(A_2) \leq m_L\{ M_1,M_2\} <\mu'N^3 $.  Assume we contract conjugate variables $z_h$ which are not high, let $A_i=B_i z_h^{\pm}$ so that $B_1B_2$ is the part of  $\{ M_1,M_2\}$   in the low variables and $m(B_1)+m(B_2)< \mu'N^3$.  By conservation of momentum for $M_2$ we have $|\er(h)|\leq m(B_2)+\kappa N$, hence $m(A_2)\leq 2m(B_2)+ \kappa N  <2\mu'N^3+ \kappa N$. In both cases we deduce that we  we obtain the contributions of lines 2,3  of Formula \eqref{split}.
 \smallskip

 Let us show finally that, by conservation of momentum  case b) is not possible.    We are now assuming that for instance  $M_1=e^{i(k_1,x)}y^{a_1}B_1 z_hz_m^{\pm}  $ with $z_m$  high  while $z_h$ not high, hence $|\er(h)|\leq \mathtt c N^{\tau_1}$ and  $m(B_1)<\mu' N^3$.  
 We have $m(B_1 z_h)\leq \mu' N^3+\mathtt c N^{\tau_1} $. So by conservation of momentum  we have.
 $$\theta'  N^{\tau_1}< |\er(m)|\leq m(B_1)+|\er(h)|+\kappa N<\mu' N^3 +\theta N^{\tau_1}+\kappa N$$ which implies $(\theta'  -\theta)N^{\tau_1} <\mu' N^3  +\kappa N$ contradicting \eqref{pois2}.
\end{proof}

\subsection{The proof of Proposition \ref{main}} We use all the notations and hypotheses of  \ref{main}.
We can first use the standard estimates \eqref{commXHK} and obtain
 \begin{equation} 
\|X_{\{f^{(1)},f^{(2)}\}}\|^\lambda_{s',r'} <  2^{2n+3}\delta^{-1}\|X_{f^{(1)}}\|_{s,r}^\lambda\|X_{f^{(2)}}\|_{s,r}^\lambda\leq 2^{2n+3}\delta^{-1}\|X_{f^{(1)}}\|_{\vec p}^T\|X_{f^{(2)}}\|_{\vec p}^T,
\end{equation}
here $\delta$ is defined in \eqref{diffusivumsui}.\smallskip

Since the   allowable parameters ${\Pappa}',\theta',\mu' $ satisfy  \eqref{pois1} we have that $ N ,\theta',\mu',\tau$ satisfy \eqref{pois2} for all $N>{\Pappa}',\tau\geq \tau_0$.  
In order to show that $\{f^{(1)},f^{(2)}\}$ is quasi--T\"oplitz (with respect to the chosen parameters), it is enough to provide, for
all $N>{\Pappa}'$ and the  allowable parameters   $\underline p':=(N,\theta',\mu',\tau)$   a decomposition
$$\Pi' \{f^{(1)},f^{(2)}\}=\Pi_{N,\theta',\mu',\tau} \{f^{(1)},f^{(2)}\}= \mathcal F^{(1,2)}+ N^{-4d \tau} \bar f^{(1,2)}$$  so that  $\mathcal F^{(1,2)}\in \mathcal  T_{\underline p'}$ and also
\begin{equation}\label{stizza}
 \|X_{\mathcal F^{(1,2)}}\|_{s',r'}, \|X_{\bar f^{(1,2)}}\|_{s',r'} <
12\, 2^{2n+3}\delta^{-1}  \|X_{f^{(1)}}\|_{s,r}^T\|X_{f^{(1)}}\|_{s,r}^T.
\end{equation}
\smallskip

For $\epsilon>0$ take $\mathcal F^{(i)}\in \mathcal
T_{\underline p}$, ${\underline p}=(s,r,N,\theta,\mu,\tau)$ so that  setting
     \begin{equation} \bar f^{(i)} :=N^{4
d\tau}\big( \Pi_{(N,\theta,\mu,\tau)}f^{(i)}- \mathcal F^{(i)}\big),
\end{equation}
we have   \begin{equation} \label{epsilon}  \max( \| X_{\mathcal F^{(i)}}\|_{s,r},\|
X_{\bar f^{(i)} }\| _{s,r} ))< \|X_{f ^{(i)}}  \|_{\vec p}^{T}+\epsilon  .
 \end{equation}
 We thus define the function $$\mathcal F^{(1,2)}:=
\Pi_{N,\theta',\mu',\tau}\!\left( \! \{ \mathcal F^{(1)},\mathcal
F^{(2)}\}^H+\! \{\mathcal F^{(1)},\Pi^L_{N,3\mu'}
f^{(2)}\}^{(y,x)+L}\!+\!\{\Pi^L_{N,3\mu'} f^{(1)},\mathcal
F^{(2)}\}^{(y,x)+L} \right)$$ where we have denoted
$\{\cdot,\cdot\}^{(y,x) +L}=
\{\cdot,\cdot\}^{(y,x)}+\{\cdot,\cdot\}^{L}$.  We need to show that it satisfies the required conditions.

\begin{lemma}
(i) One has
$\mathcal F^{(1,2)}\in \mathcal  T_{\underline p'}$. 

(ii) Setting $\bar
f^{(1,2)}= N^{4d\tau}(\Pi'
\{f^{(1)},f^{(2)}\}-\mathcal F^{(1,2)})$ one has that the bounds
\eqref{stizza} hold.
\end{lemma}
 \begin{proof} 
(i) The constraints \eqref{pro11} are satisfied and we just apply Lemma \ref{poisb}.

(ii) The estimate \eqref{stizza} for $\mathcal F^{(1,2)}$ follows by  Cauchy  estimates since
 $$ \|X_{\mathcal F^{(1,2)}}\|_{s',r'}\leq  \|X_{\{\mathcal F^{(1)},\mathcal F^{(2)}\}}\|_{s',r'}+ \|X_{\{\mathcal F^{(1)},f^{(2)}\}}\|_{s',r'}+ \|X_{\{\mathcal F^{(2)},f^{(1)}\}}\|_{s',r'} $$
 $$\leq 2^{2n+3}\delta^{-1}\big[\|X_{ \mathcal F^{(1)}}\|_{r ,s }\| X_{\mathcal F^{(2)} }\|_{r ,s } +\|X_{ \mathcal F^{(1)}}\|_{r ,s }\| X_{ f^{(2)} }\|_{r ,s } +\|X_{ f^{(1)}}\|_{r ,s }\| X_{\mathcal F^{(2)} }\|_{r ,s }\big ]$$

 $$\leq  3\, 2^{2n+3}\delta^{-1}   \|X_{ f^{(1)}}\|^T_{\vec p} \| X_{ f^{(2)} }\|^T_{\vec p } .$$

We now estimate $  \|X_{\bar f^{(1,2)}}\|_{s',r'} $. We have $\bar
f^{(1,2)}=$
$$ N^{4d\tau} \Pi'(
\{f^{(1)},f^{(2)}\}- \{ \mathcal F^{(1)},\mathcal
F^{(2)}\}^H- \{\mathcal F^{(1)},\Pi^L_{N,3\mu'}
f^{(2)}\}^{(y,x)+L}-\{\Pi^L_{N,3\mu'} f^{(1)},\mathcal
F^{(2)}\}^{(y,x)+L})$$ We substitute in  formula \eqref{split} $
\Pi_{\underline p}f^{(i)}= \mathcal F^{(i)}+
N^{-4d\tau}\bar f^{(i)}$.
\medskip
 
Thus $\bar
f^{(1,2)}=\Pi'(\Xi)$ with
\begin{equation}\label{split2}
 \Xi=\big[ \{ \mathcal F^{(1)}+ N^{-4d\tau}\bar f^{(1)},  \bar f^{(2)}\}^H+ \{\bar f^{(1)}),  \mathcal F^{(2)}+N^{-4d\tau}\bar f^{(2)}\}^H+ N^{-4d\tau}\{\bar f^{(1)},  \bar f^{(2)}\}^H\end{equation}
$$+ \{  \bar f^{(1)},\Pi^L_{N,3\mu'}
f^{(2)}\}^{y,x+L}+   \{\Pi^L_{N,3\mu'}
f^{(1)},  \bar f^{(2)})\}^{y,x+L} $$$$ +N^{4d\tau}\big[ \{\Pi^U_N
f^{(1)},f^{(2)}\}+\{f^{(1)},\Pi^U_N
f^{(2)}\}  -\{\Pi^U_Nf^{(1)},\Pi^U_N
f^{(2)}\} \big]$$ 

 In  order to estimate the norm $  \|X_{\bar f^{(1,2)}}\|_{s',r'} $ we estimate the norm   $  \|X_{\Xi}\|_{s',r'} $.

  If we have chosen $\epsilon$ sufficiently small  in Formula \eqref{epsilon}  the first  two lines of $\Xi$ can be  estimated by  $9\, 2^{2n+3}\delta^{-1}   \|X_{ f^{(1)}}\|^T_{\vec p } \| X_{ f^{(2)} }\|^T_{\vec p }$ and the last by the smoothing estimates \eqref{smoothl}.
 $$ \|X_{\Pi^U_N f}\|_{s',r'}\leq  2e^{-N(s-s')}\|X_f\|_{s,r} ,$$
 $$\|  X_{ \{\Pi^U_N f^{(1)},f^{(2)}\}+ \{f^{(1)},\Pi^U_N
f^{(2)}\}}\|_{s',r'} \leq  8\, 2^{2n+3} e^{-N(s-s')} \delta^{-1} \|X_{ f^{(1)}}\|_{s,r}\|X_{f^{(2)}}\|_{s,r}. $$  Since by \eqref{pois1} $N^{4d\tau } e^{-N(s-s')}<1$, the estimate \eqref{stizza} follows. 
 \end{proof}

 \begin{proof}{\it Conclusion of the proof of (\textbf{proposition \ref{main}})}
 Proposition \ref{main}(i) follows from the previous Lemma.
The proof of (ii)   is identical to that of Proposition  5 (ii) of \cite{PX}.
 \end{proof}

   \part{The KAM algorithm}\label{rottura}
\section{An abstract KAM theorem}
 The starting point for our KAM Theorem is a class of Hamiltonians $H$, variation of the Hamiltonians  considered in \cite{PX}:
 \begin{equation}\label{hamH0}  H:=\mathcal N +P\,, \quad P=
P(x,y,z,\bar z, \xi).
\end{equation}$$\mathcal N:= (\omega (\xi),y)+\sum_{k\in S^c_r} \Omega_k(\xi) |z_k|^2+ \sum_{(h,k)\in S^c_i} a_{h,k} (|z_h|^2-|z_k|^2) + b_{h,k}(z_h z_k +\bar z_h\bar z_k)
, $$
defined in $ D(s,r) \times  \mathcal O$, where we take $\mathcal O\subseteq \ro \mathfrak K_\alpha
$   a compact domain of diameter of order $\varepsilon^2$ contained in one of the components of Theorem \ref{xixi} and  subject to the restriction given in Remark \ref{lemme}. The functions $\omega(\xi),\Omega_n(\xi),a_{h,k} ,b_{h,k}$ are well defined for
$\xi\in \mathcal O$.  In our examples the set $S^c_i$ of {\em complex eigenvalues} or  {\em hyperbolic  terms}  is finite.\smallskip

It is well known that, for each $\xi\in \mathcal O$, the Hamiltonian equations of motion
for the unperturbed $\mathcal N$
admit the special solutions  $(x, 0, 0, 0)\to (x+\omega(\xi) t,
0,0,0)$ that  correspond  to
  invariant tori in the phase space.

\smallskip  Our aim is to prove that,  under suitable hypotheses, there is a set   $\mathcal O_\infty\subset \mathcal O$   of positive Lebesgue measure, so that,   for all
 $\xi \in \mathcal O_\infty$  the Hamiltonians $H$ still admit
 invariant tori (close to the ones of the unperturbed system) with some frequency $\omega^\infty(\xi)$ (close to  $\omega (\xi)$). \smallskip
 
 Given a value $\xi$ of the parameters   we  have the torus  given by the equations $y=z=0$.  If the Hamiltonian vector field $X_{H}$ of a Hamiltonian $H$ is tangent to this torus, and if on this torus it coincides with $\sum_{i=1}^n\omega^\infty_i(\xi)\pd{}{x_i}$  then the Hamiltonian evolution is quasi--periodic on this torus, which is called a {\em KAM torus}  for $H$. 
 
 This condition depends only on the terms $H^{\leq 2}$ of   $H$ of degree $\leq 2$. Denote by $H^{(i,j)}$   the part of degree $2i$ in $y$ and $j$ in $z$, recall that we give degree $0$ to the angles $x$, $2$ to $y$ and $1$ to $w$.:
 \begin{equation}\label{ordine}
H^{\leq 2}= H^{0}(x)+ H^{0,1}(x; w) +H^{2}(x;y,w),\quad H^{2}(x;y,w)=H^{1,0}(x;y)+H^{0,2}(x; w)
\end{equation}
 For a value $\xi$ giving a KAM torus for $H$ we have that the term $H^{0}$ must be constant (and we usually drop it), the term $H^{(0,1)}=0$ and finally $H^{(1,0)}=\sum_{i=1}^n\omega^\infty_i(\xi)y_i $ (there is no condition on $H^{0,2}(x; w)$).  

Therefore our goal is to find a change of variables (possibly in a smaller domain)  so that we have  a large set $\mathcal O_\infty$ of parameters  defining KAM tori for $H$.  The precise statement is contained in Theorem \ref{KAM}.

We start by describing the class of Hamiltonians to which the method applies.
 \medskip
 
\subsection{Compatible Hamiltonians\label{CaH}}\quad We consider a class of Hamiltonians stable under the KAM algorithm. 

In the construction there will appear   parameters $$\underline p=(r,s,{\Pappa},\theta,\mu, \O, a,{\esse}, M, L),\quad (\e,\g,\c) $$  playing different roles, where $\O\subset  \ro \mathfrak K_\alpha$ is a compact set of positive measure (of order $\e^{2n}$)  while all the others are positive numbers such that ${\Pappa}>N_0$  will play the role of a {\em frequency cut}  and will grow to $\infty$ in the recursive algorithm, $\g<1$ is an auxiliary parameter which we fix at the end of the algorithm and should be thought of as of order smaller than $\e^2$ but larger of the order of the perturbation,   and \footnote{note that the condition $LM<4$ is added only in order to simplify notations, any $\e^2,r,\Pappa$ independent constat would be acceptable.}: 
$$  (\mu-\mathtt c){\Pappa}^{\tau_0}, (\mathtt C-\theta){\Pappa}^{4d\tau_0}>\kappa {\Pappa}^4,  $$\begin{equation}\label{palleK} \gamma \leq 2 \varepsilon^2 M<1/6\,,\;  (8 M\e^2)^{-1}>{\esse}> 4\sqrt{n}ML \,,\quad a \leq M \,,\quad LM<4.
\end{equation} 
Recall that $\kappa= \max(|\mathtt j_i|)$ and see \eqref{itau} for the definition of $N_0,\mathtt c,\mathtt C, \tau_0$.  \smallskip

 We consider Hamiltonians, defined in $ D(s,r) \times  \mathcal O$,  of the form: 
 \begin{equation}\label{hamH}  H:=\mathcal N +P,\, \mathcal N:= (\omega(\xi),y)+\sum_{k\in S^c_r} \Omega_k(\xi) |z_k|^2+\mathcal C, \quad P=
P(x,y,z,\bar z, \xi),\end{equation}$$\mathcal C= \sum_{(h,k)\in S^c_i} a_{h,k} (|z_h|^2-|z_k|^2) + b_{h,k}(z_h z_k +\bar z_h\bar z_k),\quad  (\omega(\xi),y)=\sum_{i=1}^n\omega_i(\xi)y_i.
$$ 
We also may use for the complex eigenvalues    the complex coordinates, as explained in \S \ref{coco}. In that case unless there is a risk of confusion we may write the Hamiltonian in full diagonal form   $\mathcal N:= (\omega(\xi),y)+\sum_{k\in S^c} \Omega_k(\xi) |z_k|^2$ with the understanding that finitely many $z_k$ are complex coordinates  which come in pairs and that the corresponding $\Omega_k(\xi)$ are   complex and in conjugate pairs.

  \begin{definition}\label{buone} We say that a Hamiltonian \eqref{hamH} is {\em compatible with the parameters $\underline p$} if the following conditions $(A1)$--$(A5)$ are satisfied:\smallskip 

\noindent $(A1)$\quad {\it Non--degeneracy:} The map $\xi\to
\omega(\xi)$ is a lipeomorphism\footnote{in our applications all maps will actually be analytic} from  $\mathcal O$ to
its image with $ |\ome^{-1}|^{lip}_\infty \leq L$. Setting $\vgot_i:= |\mathtt j_i|^2$, for $i= 1,\dots,n$ we have $ |\ome(\xi)-\vgot|_\infty\leq M \varepsilon^2$.
\bigskip

\noindent $(A2)$\quad  {\it Asymptotics of normal frequency:}
 For all $n\in S^c$ we have a decomposition:
\begin{equation}\label{asymp1} \Omega _n(\xi)=\s(n)(|\er(n)|^2+2\val_n(\xi))+ \tilde
\Omega _n(\xi).
\end{equation} 
 We assume that the $ \vartheta _n(\xi)$ are chosen in a finite list of analytic functions which are homogeneous of degree one in $\xi$, moreover  the $\tilde\Ome:=\{ \tilde\Ome_n\}_{n\in S^c}$ are Lipschitz functions from $\mathcal O \to l_\infty$ with\footnote{recall that on $\R^n$ we use the $l^\infty$ norm.}   \begin{equation}\label{omelip}
|\ome|^{lip}_\infty+|\Ome|^{lip}_\infty \leq M\quad{\rm and }\quad 2|\val|^{lip}_\infty \leq M\,,\quad  2|\val|_\infty \leq M \e^2\,. \end{equation}  
\bigskip 

\noindent$(A3)$\quad  {\it Regularity and Quasi--T\"oplitz property:} the functions $P$, $ \val(z):=\sum_j\vartheta_j |z_j|^2$  and $ \tilde\Omega (z):=\sum_j\tilde\Omega _j|z_j|^2$ are $M$--regular, preserve momentum as in \eqref{mome},  are Lipschitz in the parameters $\xi$  and  quasi-T\"oplitz
with parameters $({\Pappa},\theta,\mu)$  (cf. Definition \ref{topbis}). Moreover for all $N \geq K$, $\tau_0\leq \tau\leq \tau_1/4d$ we have
$\Pi_{(N,\theta,\mu,\tau)} \sum_j\val_j |z_j|^2 \in \mathcal T_{(N,\theta,\mu,\tau)}$.
 \smallskip
 
We need to control the norms of the above functions, we use the free parameter $\gamma$, whose purpose is to estimate the measure of  the various Cantor sets which will appear, and  set  $\vec p= (s,r,\PaPi,\theta,\mu,\lambda=\gamma M^{-1},\mathcal O)$.

We define:
\begin{equation}
\label{glieps}\gamma^{-1}\|X_{P^{(i)}}\|^{T}_{\vec p}:= \ix^{(i)},\;i=0,1,2\,,\quad \vec{\ix}=(\ix^{(0)},\ix^{(1)},\ix^{(2)}),\quad \gamma^{-1}\|X_{ P}\|^{T}_{\vec p}:=\Theta\,.
\end{equation}     
\medskip
\quad We require that 

\noindent $(A4)$\quad{ \it Smallness condition}:\footnote{by  $|\vec\ix|:=|\ix^{(0)}|+|\ix^{(1)}|+|\ix^{(2)}|$ we mean its $L^1$ norm  note that $|\vec{\ix}|<3 \Theta$.} 
\begin{equation}\label{topo}
 \Theta<{1} \,, \; \gamma^{-1}\| X_{ \tilde\Omega   }\|^{T}_{\vec p}<  1\,,\; 2^{2n+15}  |\vec{\ix}| {\Pappa}^{ {4d\tau_1} } < 1.
\end{equation}
Note that the definition of T\"oplitz norm for diagonal matrices and \eqref{topo} imply 
	  \begin{equation}
\label{Omeinf}|\tilde\Ome|_{\infty}\leq \| X_{ \tilde\Omega   }\|^{T}_{\vec p} < \gamma\leq 2M \varepsilon^2 .
\end{equation}

We  note moreover that from condition $(A3)$ and  by  Remarks \ref{dupro},\ref{diaH}, we have that 
$$\| X_{ \val (z)}\|^{T}_{\vec p}= \|X_{ \val (z)}\|^{\lambda}_{s,r}= |\val|_\infty + \lambda |\val|_\infty^{lip}=|\val|_\infty + \gamma M^{-1} |\val|_\infty^{lip}. $$  Then \eqref{palleK}  and \eqref{omelip} imply $\| X_{\val (z)}\|^{T}_{\vec p}<   2M\e^2$.
 
 \bigskip 

\noindent$(A5)$\quad {\it Non--degeneracy ({\em Melnikov conditions}):}  We denote by $\Delta_{\xi,\varrho}f=\frac{|f(\xi)-f(\varrho \xi)|}{(1-\varrho)|\xi|}$  the variation of $f$ in the radial direction:  \begin{equation}\label{pota}
\inf_{\varrho \in \R^+\!\!\!\!\!,\, \,\,\xi\in \O\,\atop
 \varrho\neq 1, \, \varrho\xi\in \O } |\Delta_{\xi,\varrho}(\langle \ome, k\rangle  +  \Ome \cdot l) |>a    \, ,
 \quad \forall k \in {\mathbb Z}^n, \, l \in {\mathbb Z}^{S^c} \, , |l| \leq 2 \,,\ |k|\leq {\esse}\,,\ \ (k,l) \neq 0\,,
\end{equation}  for all $(k,l)$ compatible with momentum conservation and such that, setting $\Vgot:=(\Vgot_n)_{n\in S^c},\,$ with $ \Vgot_n:=\sigma(n)|\er(n)|^2$, one has (see $(A1)$):  
\begin{equation}\label{argh}
\langle \vgot, k\rangle  +  \Vgot \cdot l=0.
\end{equation}
Observe that $\langle \vgot, k\rangle  +  \Vgot \cdot l\in\Z$.
\bigskip
\end{definition}\begin{remark}\label{lapalisse}
If $H$ is compatible with the parameters $\underline p$ it is also compatible with all choices of parameters $\underline p'$ where
$$s'\leq s\,,\;r'\leq r\,,\;{\Pappa}'\geq {\Pappa}\,,\;\O'\subseteq \O\,, $$ provided that \eqref{topo} still holds.
\end{remark}
  In working with the  cubic NLS  we will also have (at the first step of the algorithm): 
\smallskip

\noindent $(A2*)$\quad{ \it Homogeneity}:\quad  The functions  $\ome(\xi)-\vgot, \Ome_m - \s(m)|\er(m)|^2$ are analytic and homogeneous of degree $1$.\smallskip
\begin{remark}\label{eind}
By the homogeneity of $\ome(\xi)$ and $\val$ we may easily see that $M,L,a$ can be taken $\e$ independent. These 3 paramters will remain bounded away from $0,\infty$ in the course of the algorithm.   If the condition $(A2*)$ were that the functions are homogeneous of degree $q>1$ then we would only have $ML, M a^{-1}$ as $\e$ independent constants. 
\end{remark}
 
\bigskip
\subsubsection{Infinite dimensional KAM theorem}
 \noindent Now we state our infinite dimensional KAM theorem.  We use the symbol $\lessdot$ to mean that we have  $\leq cost  $ where $cost$ depends only on $n,d,\kappa$. We use the same symbols $*$ as in the previous paragraph but with a $*_0$. 
 \begin{theorem}\label{KAM}

Assume  that a Hamiltonian $\mathcal N_0+P_0$ in (\ref{hamH}) is compatible  with the parameters $p_0= (r_0,s_0,{\Pappa}_0,\theta_0,\mu_0,\O_0=\ro\mathfrak K_\alpha,a_0, {\esse}, M_0, L_0) $,  i.e. satisfies  $(A1-A5), \  (A2*)$,  and    that:
\begin{equation}
\label{normal}\mathtt c\leq  \frac{\mu_0}{2}\,,\quad \mathtt C\geq 2\theta_0 \,,\quad M_0L_0=2\,,\
\quad  \esse= 8 \sqrt{n}M_0L_0=16 \sqrt n
\end{equation} 
and that $\Pappa_0$ is sufficiently large (depending only on the remaining  parameters $p_0$).

  There exists $0<\c<1$ and a positive  constant $B:=B(\Pappa_0)$ such that if furthermore $\Theta_0<\c$ (cf. \eqref{glieps}) and for all $ \g$ such that $B  \varepsilon^{ 2(n-1)} \g <|\O_0|a_0$, 
we may construct: 
\vskip15pt $ \bullet $ {\bf (Frequencies)}
Lipschitz functions $ \o_\infty: \O_0 \to {\mathbb R}^n $,
$ \Ome^{(\infty)} : \O_0 \to  \ell_\infty $, satisfying
\begin{equation}\label{muflone}
| \o_\infty - \o_0 |^\lambda \, , \ | \Ome^{(\infty)} -  \Ome^{(0)} |_\infty^\lambda \leq  \gamma |\vec{\ix_0}|{\c}^{-1}
\end{equation}
where $\vec{\ix_0}$ is defined by Formula \eqref{glieps} and $ | \o_\infty|^{\rm lip} $, $ | \Ome^{(\infty)} |^{\rm lip}_{\infty}  \leq 2 M_0 $.
\vskip15pt $ \bullet $  {\bf (Cantor set)}
A Cantor set ${\mathcal O}_\infty$ 
\begin{equation}
\label{finca}| \mathcal O_0 \setminus {\mathcal O}_\infty|\leq B \varepsilon^{ 2(n-1)} \frac{\gamma}{a_0} ,
\end{equation} the smallness condition on $\g$   ensures that $\O_\infty$ has positive Lebesgue measure.
\vskip15pt $ \bullet $  {\bf (KAM normal form)}
A Lipschitz family of analytic symplectic maps
\be\label{eq:Psi}
\Phi : D(s/4, r/4)\times {\mathcal O}_\infty   \mapsto  D(s, r)
\end{equation}
of the form $ \Phi = {\rm I} +  \Psi $ with $\| \Psi \| _{r/4,s/4} \lessdot |\vec{\ix}_0|$,
where  $ {\mathcal O}_\infty $ is defined in the previous item,
such that,
\be\label{Hnew}
H^\infty (\cdot ; \xi ) :=  H \circ \Phi (\cdot ; \xi) =
\o_\infty(\xi) y_\infty + \Ome^{(\infty)}(\xi)z_\infty \bar z_\infty + P^\infty
\quad {has} \quad P^\infty_{\leq 2} = 0. 
\end{equation}
\end{theorem}
Formula \eqref{Hnew} tells us that the final Hamiltonian has invariant KAM tori parametrized by $\xi\in \mathcal O_\infty$.
\begin{remark}\label{lapa2}
We will use the freedom given by Remark \ref{lapalisse} to choose ${\Pappa}_0$ large enough so that some necessary bounds, which appear in the proof, hold. 

It is important to notice that all the conditions  which we shall impose on $K$  are independent of $r_0$ and $\mathcal O_0$.
\end{remark}
Theorem \ref{KAM} is proved by an iterative procedure, which occupies the rest of this long section. We produce a
sequence of Hamiltonians $H_\nu=\mathcal N_\nu+P_\nu$, each of these Hamiltonians will  satisfy the properties (A*) of Section \ref{CaH} for suitable parameters $p_\nu$. 

 In particular  for the compact sets $\mathcal O_\nu$  and  for the domains $D(s_\nu,r_\nu)$   we have the {\em telescoping} $\mathcal O_\nu\subset \mathcal O_{\nu-1},\ D(s_\nu,r_\nu)\subset D(s_{\nu-1},r_{\nu-1}).$ 

 We set $\mathcal O_\infty=\cap_\nu\mathcal O_\nu,\ D(s_\infty,r_\infty)=\cap_\nu D(s_\nu,r_\nu)$ and the construction will be such that $\mathcal O_\infty$ has positive measure while $s_\infty=s/4,r_\infty=r/4$.

 We have a sequence of symplectic
transformations $\Phi_\nu:H_{\nu-1}:= H_{\nu}$, where $\Phi_\nu$  is  the value at 1  of the flow  generated by the Hamiltonian vector field $X_{F_{\nu-1}}$ associated to the generating function $F_{\nu-1}$  determined as the unique solution of the {\em Homological equation}, and depending on  $P_{\nu-1}^{\leq 2}$. The transformation is well defined on the domain $D(s_\nu,r_\nu)\times \mathcal O_\nu$, with $D(s_\nu,r_\nu)\subset  D(s_{\nu-1},r_{\nu-1})$. At each step, the  perturbation  remains bounded while the part $P_{\nu}^{\leq 2}$ becomes
 smaller.  We also denote $D(s_\nu,r_\nu)=D_\nu$.  The T\"opliz norm  of a function $G$  defined on $D_\nu\times \mathcal O_\nu$ and relative to the parameters  extracted from the $p_\nu$  will be denoted by $\Vert G\Vert_\nu$.
 
 The goal is to  pass to a limit Hamiltonian $H_\infty=\mathcal N_\infty+P_\infty$ with the property that $P_\infty^{\leq 2}=0 $ so that the Hamiltonian vector field on the family of tori parametrized by the parameters $\xi\in \mathcal O_\infty$,  where the normal coordinates are 0,  coincides with the Hamiltonian vector field of $\mathcal N_\infty$.\smallskip
 
 The relevant estimates to be performed are the following.
 \begin{itemize}
\item We have to estimate the norm  of  each $F_{\nu-1}$ so to make sure  that  the value $\Phi_{\nu-1}^1$ at 1  of the flow  generated by the Hamiltonian vector field $X_{F_{\nu-1}}$ is well defined on $D _\nu \times \mathcal O_\nu$.

Here the problem  is that of {\em small denominators}  since we have to divide by eigenvalues. Here the quasi-T\"opliz properties play a major role and the key is Proposition \ref{submain}.
\item While we establish the previous item we have to estimate the measure of the set $\mathcal O_\nu$,  Lemma    \ref{measure}.
\item We have to perform all the estimates on the new parameters $p_\nu$, this is done in \S \ref{s:Kstep}.
\item We have to estimate the norm   of the  part $P_{\nu}^{\leq 2}$, for this we have to control simultaneously the three parts in which it naturally decomposes.
\item We need to prove that the set $\mathcal O_\infty$ has positive measure, Corollary \ref{cor:finale}.
\item We need to prove that on the set $\mathcal O_\infty$ we have a limit change of coordinates $\Phi_{\infty}^1$ giving rise to the limit Hamiltonian \ref{cor:finale}. 
\end{itemize}
 {\bf Warning}\quad In order for this Theorem to give a non--empty statement we need to have conditions  which ensure that the constraints on $\gamma$ can be satisfied. These constraints amount to a smallness condition on the perturbation, (cf. Theorem \ref{gacom}). In the applications to the NLS this condition is satisfied by suitably restricting the domain of definition of $\mathcal H$.   \bigskip

\subsection{KAM step}

\subsubsection{Formal KAM step}
The input of a KAM step is  a Hamiltonian $H$ of the previous form with parameters $\underline p$. Of particular relevance is the parameter   ${\Pappa}\geq \Pappa_0$ which gives a  {\em frequency cut}.   The output must produce a new Hamiltonian $H_+$ of the previous form with parameters $\underline p_+$. Thus we need to start from a  subset $\mathcal O_+\subset\mathcal O$  of the parameters $\xi$, two  new  values for the radii of the domain $s_+\leq s,\ r_+\leq r$  and a  symplectic  transformation  $\Psi:D(s_+,r_+)\times \mathcal O_+\to D(s,r)$   of type  $\Psi=e^{ad(F)}$, so that finally we have a new Hamiltonian  $H_+=H\circ \Psi$ which we expect to be a {\em simplified version of $H$} by evaluating the new parameters $\underline p_+$.  After iterating infinitely many times the KAM step, we hope to arrive at  the desired final Hamiltonian which shows the existence of quasi--periodic orbits as in Theorem  \ref{KAM}.

The function $F$ is obtained by solving the {\em homological equations}. In order to explain this 
it will be convenient to write explicitly  the terms of $ P^{(2)}(x,y,w)$:
  $$P^{(1)}(x;w) = \sum_{m\in S^c,\; \sigma=\pm,
  \; k}P^{(1)}_{k,m,\sigma}  e^{\ii (k, x)} z_m^\sigma$$
$$ P^{(0)}(x)= \sum_{k} P^{(0,0)}_k e^{\ii (k, x)} \,,\quad P^{(1,0)}(x;y)= \sum_{k} P^{(1,0)}_k\cdot y e^{\ii (k, x)},$$ 
$$  P^{(0,2)}(x;w)= \sum_{n,m\in S^c\,,\; \sigma,\sigma'=\pm\,,\; k } P^{(0,2)}_{k,m,\sigma,n,\sigma'} e^{\ii (k, x)}z_m^\sigma z_n^{\sigma'}$$
 Only those terms which satisfy conservation of mass and momentum may appear.
We set
\begin{equation}\label{nucleo}
[P^{\leq 2}]=   (P^{(1,0)}_0  , y )+ \sum_{m\in S^c_r} P^{(0,2)}_{0,m,+,m,-} |z_m|^2+\mathcal P.
\end{equation} Where $\mathcal P$  is diagonal only in the complex notation and arises from the term $\mathcal C$ of the normal form $\mathcal N$ (see \eqref{hamH}).

On the space of quadratic Hamiltonians $ad(\mathcal N)$ has a basis of eigenvectors  described in \S \ref{eiei}. On the space relative to the  non-zero eigenvalues $ad (\mathcal N)$ is formally invertible, hence for those $\xi$ for which  the Melnikov resonances do not occur,    $[P^{\leq 2}]$ is the projection of $P^{\leq 2}_{\leq {\Pappa}}$ on the kernel of $ad(\mathcal N)$. 
We define
\begin{equation}
\label{laF}F:= ad(\mathcal N)^{-1}( P^{\leq 2}_{\leq {\Pappa}}- [P^{\leq 2}])\quad\implies  ad(F)\mathcal N=  [P^{\leq 2}]-P^{\leq 2}_{\leq {\Pappa}}
\end{equation}  and since $ad(\mathcal N)^{-1}$ is diagonal (at least in complex coordinates) this definition can be given degree by degree, thus defining $F^0, F^{(1)},F^{(2)}$. Notice that even if we use complex coordinates $F$ is always real.

\subsubsection{Estimates}
Formula \eqref{laF} defines $F$ as a formal expression. We now
impose a lower bound on the eigenvalues of  $ad(\mathcal N)$ on the space of functions of degree $\leq 2$ which implies that $F$ is analytic.  
Let us restrict our attention,  for instance,  
 to the set ${\mathcal O}'$ of  $\xi\in\mathcal O$ such that:  for all $k\in \Z^n$, $|k|\leq {\Pappa}$ and
$l\in \Z^{S^c}$, $|l|\leq 2$ which satisfy momentum conservation, we have
\begin{equation}\label{zero1}
|\langle \ome,k\rangle +(l,\Omega)|\geq \gamma {\Pappa}^{-{{2d{\tau_1}}}}  .\end{equation}
   \begin{lemma}
 We have:
\begin{equation}\label{effelambda}
\|X_{F^{i}}\|^\lambda_{s,r, {\mathcal O}'}\leq {\Pappa}^{{2d{\tau_1}}}\gamma^{-1}\|X_{P^{i}}\|^\lambda _{s,r,{\mathcal O}} \,,\quad i=0,1,2\,.
\end{equation} 
\end{lemma}
 \begin{proof}
We first notice that \eqref{zero1} implies that $P^{\leq 2}_{\leq {\Pappa}}- [P^{\leq 2}]$ is a  sum of eigenvectors of $ad(\mathcal N)$ with eigenvalues bounded from below (in absolute value) by $\gamma {\Pappa}^{-{2d{\tau_1}}}$, therefore since we are using the Majorant norm we have 
$$ \|X_{F^{i}}\|^\lambda _{s,r}\leq {\Pappa}^{{2d{\tau_1}}}\gamma^{-1}\|X_{P^{(i)}- [P^{(i)}]}\|^\lambda_{s,r}\leq {\Pappa}^{{2d{\tau_1}}}\gamma^{-1}\|X_{P^{(i)}}\|^\lambda_{s,r} \,,\quad i=0,1,2.$$
\end{proof}  
Then by Proposition \ref{mpomn} $F$ defines a symplectic transformation $e^{(ad(F))}$ on a domain $D(s',r')\times \mathcal O'$, since by \eqref{topo} the condition $ 2^{2n+3}e \delta^{-1}    \|X_{F}\|_{s',r'} < 1$ holds  for a suitable choice of $s',r'$ and possibly restricting to a subset $\mathcal O'$ of the parameters.  More precisely in the next paragraph we will define a set    $\mathcal O_+\subset\mathcal O$  (see Definition \ref{Opiu}) and show in Lemma \ref{measure} that, provided $\gamma$ is sufficiently small,  this set has positive measure. 
On this set    we shall prove in Proposition \ref{key}   that the inequalities \eqref{zero1} hold. 
In order to iterate this procedure we need to be sure that $F$ is quasi-T\"oplitz and estimate its norm, this will be proved in Proposition \ref{submain}.
So for the procedure to succeed we need that the perturbation $P^{(i)}, \ i\leq 2$ be rather small.

\subsubsection{KAM step\label{Kamstep}} For simplicity, below we always use the same symbol  $cost$ to denote 
constants independent on the iteration and on the parameters of the Hamiltonian.

We now  start from a Hamiltonian in the class of Definition \ref{buone}  and describe a procedure which produces a  change of variables under which the Hamiltonian is still in the same class with new parameters which we estimate explicitly.

\begin{step}
(1) We Define in \ref{Opiu},  a compact  set $\mathcal O_+\subset \mathcal O$  such that 
\begin{equation}\label{stimis}
| \O\setminus \O_+ |\leq  \Gamma\frac{\gamma}{a} \varepsilon^{2(n-1)} {\Pappa}^{-\tau_0+n+d/2}, 
\end{equation}  
where $\Gamma$ is a constant depending only on $n,d,\kappa$.

(2) We construct,  by Formula \eqref{laF},  the   function $F$. We prove in   Proposition \ref{submain} that $F$  is quasi-T\"oplitz with parameters ${\Pappa},\theta,\mu$ for all $\xi\in \O_+$.   

For all 
positive numbers $r_+<r$ and $s_+<s$ for which:   
\begin{equation}\label{rpiu}
\,2^{2n+14} (\min( 1- \frac{r_+}{r}, 1-\frac{s_+}{s}))^{-1} |\vec{\ix}| {\Pappa}^{4d\tau_1}< \frac12\,,
\end{equation} we show that $F$  generates a  1--parameter group of  analytic symplectic
transformations $\Phi_F^t: D(s_+,r_+)\to D(s,r)$,   well defined for all $t,\, |t|\leq 1$ and for all $\xi\in\mathcal
O_+ $. 

(3) Applying Proposition \ref{main} {\it ii}) with $\underline p'\rightsquigarrow  \underline p^+= (N_+, \mu_+,\theta_+,s_+,r_+)$ we show that 
$ \Phi_F^1 H:=H_+=\mathcal N_++P_+$ is quasi-T\"oplitz for all choices of parameters ${\Pappa}_+,\theta_+,\mu_+$ 
satisfying \eqref{nubo}. 
 
\end{step}
 \begin{remark}
In order to make sure that  $ \O_+$ is non--empty we need the corresponding constant $B=\Gamma  {\Pappa}^{-\tau_0+n+d/2}$  should satisfy  $B  \varepsilon^{ 2(n-1)} \g <|\O |a $ which we shall satisfy by imposing a smallness condition on $\g$ since $\Pappa\geq \Pappa_0$.
\end{remark}
The core of the construction is to compute the parameters  $M_+,L_+,a_+$, see \eqref{emmepiu}, and  $\Theta_+,\vec{\ix}_+$, see \eqref{sonasega} relative to the new Hamiltonian. The iterative KAM algorithm is based on the fact that if $\c$ in Theorem \ref{KAM} is small enough then $H_+$ is  compatible with the parameters $\vec p_+$ and respects the smallness condition $(A3^*)$, so one may iterate the {\em step}.
 
\subsection{The set $ O_+$} 
\begin{definition}\label{Opiu}  $\mathcal O_+=\mathcal O_{+,\gamma}$ is defined to be the   subset of $\xi\in\mathcal O$ where the following {\em Melnikov non--resonance conditions} are satisfied:
\begin{itemize} 
 \item[i)] For  all $k\in \Z^n,\  |k|\leq {\Pappa} $    and $h\in \Z$, so that  $(h,k)\neq (0,0)$: 
\begin{equation}\label{zero}
|\langle \ome(\xi),k\rangle+ h|\geq 2\gamma \PaPi ^{-\tau_0}\,. \end{equation}
\item[ii)] For all  $ k\in \Z^n,\  |k|\leq {\Pappa} $, $m\in S^c$,  with $\pi(k)\pm \er(m)=0$:
\begin{equation}\label{first}
|\langle \ome(\xi),k\rangle \pm \Omega_m |\geq 2\gamma {\Pappa} ^{-\tau_0}.
\end{equation}  \item[iii)]For all $ |k|\leq {\Pappa}, m,n\in S^c $ such that  $\min(|\er(m)|,|\er(n)|)\leq  \mathtt C {\Pappa} ^{{\tau_1}}$ and $\pi(k)\pm   (\er(m) +\sigma \er(n)) =0 $ with $\sigma =\pm 1$:
 \begin{equation}\label{second}
\ |\langle \ome(\xi),k\rangle \pm ( \Omega_m +\sigma\Omega_n)  | \geq  2\gamma \PaPi ^{-2d{\tau_1}}.
\end{equation}
 For all $ |k|\leq {\Pappa}, m,n\in S^c $   $|\er(m)|,|\er(n)| >  \mathtt C {\Pappa} ^{{\tau_1}}$ and $\pi(k)\pm   (\er(m) + \er(n)) =0 $
 \begin{equation}\label{second2}|\langle \ome(\xi),k\rangle \pm (\Omega_m+\Omega_n) | \geq  2\gamma \PaPi ^{-2d{\tau_1}}
 \end{equation}
\item[iv)] For all affine spaces  $A=[v_i,p_i]_\ell$
 in $\mathcal H_{\Pappa }$ ($1\leq \ell<d$) with $ p_\ell <\mathtt  c \Pappa^{{\tau_1}/4d}$ we choose a point    $m_A\in [v_i;p_i]^g_\ell$ with $|\er(m_A)|> \mathtt C \Pappa^{\tau_1}$. For all
 such $m_A$ and for all  $k$ such that $|k|\leq {\Pappa}, $ we require:
\begin{equation}\label{third}
|  \langle \ome(\xi),k\rangle+\Omega_{m_A }-\Omega_{\bar n }|\geq  2\gamma  \min( \Pappa^{-2d\tau_0},\mathtt c^{2d} p_\ell ^{-2d}),\end{equation}
for all  $\bar n$ such that $\er(\bar n)= \er(m_A)+ \pi(k)$.
\end{itemize}
\end{definition}
In order to analyze $\mathcal O_+$  we need a first Lemma for the measure estimates. We define
$$\mathcal R ^{{\tau}}_{k,l}:= \left\{\xi\in \mathcal O \vert\; |\langle \ome,k\rangle + (l,\Omega)|< \gamma  {\Pappa}^{-{\tau}}\right\}, \quad l\in \Z^{S^c}.$$
\begin{lemma}\label{misu} For all $(k,l)\neq (0,0)$ $|k|\leq {\Pappa}$ and $|l|\leq 2$, which satisfy momentum conservation,   one has
\begin{equation}\label{protesi}
|\mathcal R ^{{\tau}}_{k,l}|\lessdot  \frac{\gamma}{a} \varepsilon^{2(n-1)}  {\Pappa}^{-{\tau}}.
\end{equation}
\end{lemma}
\begin{proof}
Let us first state a general fact.
Let $f$  be a Lipschitz function on  the domain $ \mathcal O\subset \ro \mathfrak K$ such that
$$ |f(x,\xi_2,\dots,\xi_n)- f(y,\xi_2,\dots,\xi_n)|>a |x-y|$$ 
for all $x\neq y$ such that $(x,\xi_2,\dots,\xi_n), (y,\xi_2,\dots,\xi_n)\in \mathcal O$.

We consider  the map $F : \xi \mapsto (f(\xi),\xi_2,\ldots,\xi_n)$  which maps $ \mathcal O$ bijectively to some set $B $. $F$ is a lipeomorphism and its   inverse has Lipschitz constant $< \max(1,a^{-1})$.  In $B $, $f$ is a coordinate and the level surfaces of $f$  are contained in a hypercube of volume $\e^{2(n-1)}$ therefore the volume of the set where $|f|<c$ can be estimated by $2\e^{2(n-1)}c$  hence on  $ \mathcal O$  it can be estimated by $2 a^{-1}\e^{2(n-1)}c$. A similar argument is valid if we work in polar coordinates, $x,y$  are radii and the other $\xi_i$  coordinates on the unit sphere. 

If $|k|\leq {\esse}$, we have assumed that,   for all functions $f_{k,l}(\xi)=\langle \ome,k\rangle +(l,\Omega(\xi))$  which satisfy \eqref{argh} we have $\inf_{\xi\neq \eta \in \O} |\Delta_{\xi,\varrho} f_{k,l} |>a $.  So we may apply the previous argument.

If, on the other hand,  $f_{k,l}$ does not satisfy \eqref{argh} (but $|k|\leq \esse$)  then we may write
$$f_{k,l} (\xi)= \mathfrak{n} + F_{k,l} (\xi), $$ where  $\mathfrak{n}$ is the non--zero integer computed in \eqref{argh} and
$$|F_{k,l}|\leq  | \langle \ome(\xi)-\vgot,k\rangle| + 4\sup_{n} |\val_n(\xi)| +2|\tilde\Ome|_\infty,$$
by hypothesis $|k| < {\esse}$ (in particular $\esse>1$) while $ 2|\val|_\infty, |\ome-\vgot |_\infty \leq M \varepsilon^2$.  So, by \eqref{Omeinf} and \eqref{palleK}, we have $ |f_{k,l}|\geq  |\mathfrak{n}| - ({\esse} +4) M  \varepsilon^2 > \frac12 $ and we may deduce that $R ^{{\tau}}_{k,l}$ is empty. 
Finally if  $|k|\geq {\esse}$ then we change the variables form $\xi$ to $\ome$ and study $G_{k,l}(\ome):=\langle \ome,k\rangle + (\Omega(\xi(\ome)) , l)$.
Let $\underline{e}_k$ be the versor of $k$, we may perform an orthogonal change of variables  in $\ome$ so that  $\underline{e}_k$ is the first vector in the standard basis. Then the Lipschitz norm of $\langle \ome,k\rangle$ is the absolute value of the vector $k$ which can be bounded below by  $\frac{|k|}{\sqrt{n}}\geq \frac{S_0}{\sqrt{n}}>4ML$.  Then we repeat our argument  with respect to $\ome $, indeed$$| (\Ome(\xi(\ome)), l)|^{lip}_\infty \leq |\Ome(\xi)|^{lip}_\infty|\ome^{-1}|^{lip}_\infty \leq ML $$ so that 
$$ | \frac{G_{k,l}(x, \ome_2,\dots,\ome_n)-G_{k,l}(y, \ome_2,\dots,\ome_n)}{x-y} |> 4ML  -  ML =3ML \,,$$  for all
  vectors $(x, \ome_2,\dots,\ome_n)\neq (y, \ome_2,\dots,\ome_n)$ in $\ome(\O)$.
  Thus the volume of the set where $|G_{k,l}(\ome)|<c$ can be estimated   by $  (ML)^{-1}(M\e^2)^{ (n-1)}c$.  The corresponding volume in the space of the parameters $\xi$ is therefore estimated by $  L^n(ML)^{-1}(M\e^2)^{(n-1)}c = (LM)^{n-1} M^{-1}\e^{ 2(n-1)} c.$ Since by \eqref{palleK} $ML<4, M>a$ the estimate follows. 
  \end{proof}
\begin{lemma}\label{measure}
 The set $\mathcal O_+$  is compact and  one has \begin{equation}
\label{stivo}|\mathcal O\setminus \mathcal O_+|\leq\Gamma \,  \,\frac{\gamma}{a} \varepsilon^{2(n-1)} K^{-\tau_0+n+d+1}. 
\end{equation} 
\end{lemma}
\begin{proof} 
Since $k$  runs on a finite set the functions  $| \langle \ome(\xi),k\rangle|$ are bounded on $\mathcal O$,  hence the first formula \eqref{zero} is satisfied for $|h|$ large, for instance $|h|>2|\omega|_\infty K$. So \eqref{zero} is actually implied by a finite number of inequalities.  

Formulas \eqref{first} and \eqref{second}  are a finite number of inequalities by definition. Formula \eqref{second2} is a priori an infinite list of inequalities, we note however that  $|(\ome,k)\pm (\Omega_m+\Omega_n)|>|\omega|_\infty K$  is large when $|\Omega_m+\Omega_n|>2 |\omega|_\infty K$.
 
 Next $\Omega_m$ has an integral part $\sigma(m)|\er(m)|^2$, but  $\sigma(m)=1$ as soon as $|\er(m)|> \mathtt C K^{\tau_1}$.    This implies  that  $\Omega_m+\Omega_n$ is large, and hence no condition is imposed,  except possibly for finitely many values of $m,n$. Finally \eqref{third}  is given only for finitely many elements; in fact in {\it iv)}, for each
$[v_i,p_i]_\ell^g$ and $k$, we  impose only  a fixed finite number of condition by
choosing a point $m_A $ and a type $u=\er(\bar n)-\bar n\in\mathcal Z$. Finally by Remark \ref{numero} there
are a finite number of $[v_i,p_i]_\ell^g$. Thus $\mathcal O_+$ is compact.

\medskip
 Let us prove the measure estimates. By definition  $\mathcal O_+$ is obtained from $\mathcal O$  by removing a finite list  of strips $\mathcal R ^{\tau_p}_p$ where $p$ runs in  a  suitable  set of pairs $k,l$. For a given set of indices $I$ denote by $\mathcal R_I^{\tau}:=\cup_{p\in I}\mathcal R ^{\tau}_p$ and, by Lemma \ref{misu}, we estimate $|\mathcal R_I^{\tau}|\leq |I| {\gamma}{a}^{-1}\varepsilon^{2(n-1)} {\Pappa}^{-{\tau}}$. The Lemma will thus follow from an estimate on the cardinality of $I$ for the various cases considered.  

Recall that the elements $m\in\Z^d$ with $\s(m)=-1$  are a finite set of some cardinality  depending only on $\kappa, n$ similarly their norm can be bounded by some $\bar \kappa$ of the order of $ \kappa^2$.  
 By Hypothesis $(A1)$, $|\omega|_\infty < \kappa^2+1 $, and we may assume that $\Pappa$ is large so that $\bar \kappa\leq \sqrt{2(\kappa^2+1){\Pappa}}$.\smallskip

\noindent i)\quad Is previously remarked we have  to impose  \eqref{zero} with $|h|\leq |\omega|_\infty K\leq (\kappa^2 +1)K$ we have:  
$$I_0:=\{(k,h)\,|, |k|\leq  {\Pappa}\,,\,|h|\leq (\kappa^2 +1){\Pappa}\}, \quad |I_0|\leq (\kappa^2 +1))(2\Pappa)^{n+1}$$$$ |\mathcal R_{I_0}^{\tau}|\lessdot \,\varepsilon^{2(n-1)} \gamma a^{-1} {\Pappa}^{-{\tau_0}+n+1}.$$

\noindent ii)\quad In \eqref{first}, by momentum conservation  $l=\pm  e_m$  implies that $\pm \er(m) =-\pi(k) $. Hence to impose \eqref{first} we have to remove the list indexed by $I_1$: $$I_1:=  \{(k,l)\,|, |k|\leq  {\Pappa}\,,\, l=\pm e_m ,\ \exists\, u\in\mathcal Z:   m+u =\mp\pi(k)\}\,,  \quad |I_1|\leq (2{\Pappa})^nd.$$
$$ |\mathcal R_{I_1}^{\tau}|\lessdot\, \varepsilon^{2(n-1)} \gamma a^{-1} {\Pappa}^{-{\tau_0}+n+1}.$$

\noindent iii)\quad If $l= \pm(e_m+e_n)$ and $\s(m)=\s(n)=1$, the index $(k,l)$ can contribute  only if we have  the condition
$$|\pm \langle \ome,k\rangle+ |\er(m)|^2+|\er(n)|^2+2\val_m+2\val_n+\tilde\Omega_m+\tilde\Omega_n|<\frac12. $$ From \eqref{omelip} and \eqref{Omeinf} this condition implies   $|\pm \langle \ome,k\rangle+ |\er(m)|^2+|\er(n)|^2|<1 $ hence $|\er(m)|^2+|\er(n)|^2<2|\omega|_\infty {\Pappa}$. Setting 
$$I_2:=\{(k,l)\,|, |k|\leq  {\Pappa}\,,\, l =\pm (e_m+e_n)\,,|\er(m)|\leq  \sqrt{2(\kappa^2+1){\Pappa}}\,,\,$$$$\exists u,v\in \mathcal Z:\ m+ n+v+u=\mp \pi(k)  \}, $$$$  |I_2 |<  (2\sqrt{2(\kappa^2+1){\Pappa}}+1)^dd^2,\quad   |\mathcal R_{I_2}^{\tau}|\lessdot\, \varepsilon^{2(n-1)} \gamma a^{-1} {\Pappa}^{-{\tau_0}+n+d/2}.$$  
 
 \noindent iv)\quad Setting $$ I_3:=\{(k,l)\,|, |k|\leq  {\Pappa}\,,\, l = e_m-e_n\,, |\er(m)|\leq \mathtt C {\Pappa}^{{\tau_1}}\,,\, \exists u,v\in \mathcal Z:\ m- n+u-v=\mp \pi(k)\} $$
One  has
$$|I_3|\lessdot\,{\Pappa}^{ d{\tau_1}+n }\implies |\mathcal R ^{2d{\tau_1}}_{I_3}|\lessdot \gamma a^{-1} \varepsilon^{2(n-1)}  {\Pappa}^{-d{\tau_1}+n }. $$

\noindent v)\quad To deal with the last case,  for all affine subspaces
$[v_i;p_i]_\ell\in \mathcal H_\Pappa$ with $p_\ell \leq \mathtt C \Pappa^{\frac{\tau_1}{4d}}$,  for all $|k|\leq {\Pappa}$ and for all types $u\in \mathcal Z$ we set $\bar n= \er(m)+ u+\pi(k)$ and define
\begin{eqnarray}\label{resonant set}\mathcal R _{k,[v_i;p_i]_\ell,u}:= \{\xi\;\vert\; |\langle \ome,k\rangle+\Omega_{m_I}-\Omega_{\bar n}| < 2 \gamma
\min(\Pappa^{-2d\tau_0},\mathtt c^{2d} p_\ell ^{-2d})\}\end{eqnarray}
Following Lemma \ref{misu}, $|\mathcal R _{k,[v_i;p_i]_\ell,u}| \lessdot \, \gamma a^{-1}\varepsilon^{2(n-1)} \min(\Pappa^{-2d\tau_0},\mathtt c^{2d} p_\ell ^{-2d})$. 
 
 We distinguish the two cases. First when $\min(\Pappa^{-2d\tau_0},\mathtt c^{2d} p_\ell ^{-2d})= \Pappa^{-2d\tau_0} $ we are in 
 $$I^K_1:= \{(k,[v_i;p_i]_\ell,u)\,| |k|< {\Pappa}\,, [v_i;p_i]_\ell\in\mathcal H_\Pappa:\,  p_\ell \leq \mathtt c \Pappa^{\tau_0}\,,\,u\in\mathcal Z \}$$
 when $\min(\Pappa^{-2d\tau_0},\mathtt c^{2d} p_\ell ^{-2d})= \mathtt c^{2d} p_\ell ^{-2d} $ we are in 
 $$I^K_2:= \{(k,[v_i;p_i]_\ell,u)\,| |k|< {\Pappa}\,, [v_i;p_i]_\ell\in\mathcal H_\Pappa:\, \mathtt c  \Pappa^{\tau_0}<  p_\ell \leq \mathtt C \Pappa^{{\tau_1}\over 4d}\,,\,u\in\mathcal Z \}$$
 
  By Remark
\ref{numero}
  we have $|I^K_1|\lessdot \Pappa^{(d+\tau_0)(d-1)+n}$ hence $$|\mathcal R _{I^\Pappa_1}| \lessdot\,    \gamma a^{-1}\varepsilon^{2(n-1)} \Pappa^{ d (d-1)+ \tau_0 (-d-1) +n }\leq   \gamma a^{-1} \varepsilon^{2(n-1)} \Pappa^{ -  d\tau_0    }$$
as for  $I^K_2$ we use again Remark
\ref{numero} to 
 bound with $(2\kappa \Pappa)^{  d^2} (2p)^{d-1}$ the number of subspaces with given $p      =p_\ell$ for all $\ell$.   
 $$|\mathcal R _{I^\Pappa_2}|\lessdot \, \gamma a^{-1}\varepsilon^{2(n-1)} \sum_{ p_\ell >\mathtt c  \Pappa^{\tau_0}}  p_\ell ^{-2d-1+d}\Pappa^{  d^2}{\Pappa}^{n}\lessdot \,  \gamma a^{-1}\varepsilon^{2(n-1)}    \Pappa^{- d \tau_0+  d^2+n}$$
(from \eqref{itau}).  
Summing all these contributions, since we can bound all   the factors by  $\varepsilon^{2(n-1)} \gamma a^{-1}{\Pappa}^{-{\tau_0}+n+d+1}$,  our Lemma is proved.
\end{proof}
\vskip15pt
We arrive now to the key estimate which handles {\em small denominators} and for which we have introduced all the formalism of cuts and quasi--T\"oplitz functions.
 \begin{proposition}\label{key}
For all $\xi\in\mathcal O_+$, for all $k\in \Z^n$, $|k|\leq {\Pappa}$ and
$l\in \Z^{S^c}$, $|l|\leq 2$ which satisfy momentum conservation,
we have
\begin{equation}
|\langle \ome,k\rangle +(l,\Omega)|\geq \gamma {\Pappa}^{-{{2d{\tau_1}}}}   .\end{equation}
\end{proposition}
\begin{proof}
By Definition \ref{Opiu} the cases i), ii), iii)  follow trivially since ${{2d{\tau_1}}}$ is  large
with respect to $\tau_0$.
 
We are left with the case  $\ell= e_m-e_n$  with $|\er(m)|,|\er(n)|> \mathtt C{\Pappa}^{\tau_1}$.  To start we have
$$ | \langle \ome,k\rangle+\Omega_m-\Omega_n | \geq  | \langle \ome,k\rangle+|\er(m)|^2-|\er(n)|^2 |- 4 |\val|_\infty-2|\tilde \Ome|_\infty.$$  We bound the two terms $4 |\val|_\infty+2|\tilde \Ome|_\infty|\leq M \e^2+ 4 M\varepsilon^2 $  by \eqref{omelip}. We need to estimate $ | \langle \ome,k\rangle+|\er(m)|^2-|\er(n)|^2 |$, by momentum conservation $\er(n)= \er(m)+\pi(k)$. First note that, setting $v:=m-\er(m)$ the type of $m$, we have:
\begin{equation}\label{scoccio}
\langle\ome,k\rangle+|\er(m)|^2-|\er(n)|^2=  \langle\ome,k\rangle-|\pi(k)|^2 -2\langle \pi(k),\er(m)\rangle=
\end{equation} 
$$=\langle\ome,k\rangle-|\pi(k)|^2+ 2\langle \pi(k),v\rangle -2\langle \pi(k),m\rangle.
$$
Note that $\pi(k)\in  B_K\cup\{0\}$.   Let $m\frec{\Pappa}[v_i;p_i]$, we distinguish two cases:  $ p_1 \geq  \mathtt C {\Pappa}^{4d\tau_0}$  or $ p_1 < \mathtt C {\Pappa}^{4d\tau_0}$ .
\smallskip

{\bf Case 1:\quad $ p_1 \geq  \mathtt C {\Pappa}^{4d\tau_0}$.}\quad 

If   $ p_1 \geq  \mathtt C {\Pappa}^{4d\tau_0}$ then $m$ has a cut at $\ell=0$. By Lemma \ref{dimelle}  we have that $m$ is on the open stratum and $\er(m)=m, v=0$.
 If $\pi(k)=0$ we have $m=n$  and the denominator is  covered by the bound \eqref{zero} with $h=0$.
 If  $\pi(k)\neq 0$   then by definition $|\langle \pi(k),\er(m)\rangle|=|\langle \pi(k), m )\rangle| \geq p_1$.  $(A1)$ implies $|\ome|_\infty<2\kappa$ so:$$ |  \langle \ome,k\rangle+\Omega_m-\Omega_n | \geq  |  \langle \ome,k\rangle+|\er(m)|^2-|\er(n)|^2 |- 4 |\val|_\infty-2|\tilde \Ome|_\infty>$$
$$|\langle\ome,k\rangle-|\pi(k)|^2  -2\langle \pi(k),m\rangle|- 5 M\varepsilon^2>2  \mathtt C {\Pappa}^{4d\tau_0}- 2\kappa {\Pappa}-\kappa^2{\Pappa}^2  - 5 M\varepsilon^2>1.  $$ 

\smallskip

{\bf Case 2:\quad $ p_1 <  \mathtt C {\Pappa}^{4d\tau_0}$.}\quad
  By hypothesis the point $m$ has
$|\er(m)|>\mathtt C{\Pappa}^{{\tau_1}}$  and $ p_1 < \mathtt C {\Pappa}^{4d\tau_0}$, thus    by Proposition \ref{key2} $m$ belongs to some $A^g$ where $A\frec{\Pappa}[v_i;p_i]_\ell,\ 1\leq\ell<d$. Write $m= \er(m)+v$ and  $n=\er(n) +u =\er(m)+\pi(k)+u $  for   two types $u,v\in \mathcal Z$. \smallskip

Let us first notice
that
 \eqref{zero1} with $l=e_m-e_n$ is surely satisfied  if $|(
\pi(k),\er(m))|\geq  {\Pappa}^3$ because in that case the absolute value of \eqref{scoccio} is greater than $ 2{\Pappa}^3- \kappa^2{\Pappa}^2-|\ome|{\Pappa}-8 d K >{\Pappa}^3$ by assumption \eqref{palleK} ($K> N_0 $) and since $(A1)$ implies $|\ome|_\infty<2\kappa$.

 If on the other hand  $|( \pi(k),\er(m))|<
{\Pappa}^3$, then $\pi(k)\in B_{\Pappa}\cup \{0\}$ is  in $\langle v_i\rangle_\ell$. In fact 
otherwise we would have $|( \pi(k),m)|> \mathtt c  {\Pappa}^{4d\tau_0}$ by
definition of $A^g$, hence $|( \pi(k),\er(m))|> \mathtt c  {\Pappa}^{4d\tau_0}-2d\kappa^2 {\Pappa}>{\Pappa}^3$ by   Formula \eqref{palleK} and $(A1)$,   a contradiction. 

In $A^g$ we have chosen a point $m_A$,  to the points $m,m_A$ we can apply Lemma \ref{minko}, thus they have a cut at $\ell$ for parameters $(\Pappa,\theta,\mu,\tau(p_\ell))$ where $\theta,\mu$ are only restricted to be allowable. Then they satisfy the hypotheses of Theorem \ref{Lostra} hence we have  $m_A= \er(m_A)+v$. 

   Consider then $n= m+\pi(k)-v+u$ and let
$n\frec{{\Pappa}}[w_i;q_i]$. We have $|m-n|= |\pi(k)-v+u|\leq  \kappa(\Pappa+2d)\leq  2\kappa \Pappa$. We now can impose that the allowable $\theta,\mu$ satisfy  the constraints given by Formula \eqref{palleK} hence we have the inequality \eqref{ladisa} for $n$ in place of $r$  hence, by Lemma \ref{lintor},  $n$ has an  $\ell$ cut $[w_i;q_i]_\ell$ with parameters
$N,\theta,\mu,\tau(p_\ell)$ and moreover $[w_i;q_i]_\ell=A+\pi(k)-v+u$. The same argument  shows that, setting $\bar n :=m_A+\pi(k)-v+u= m_A+n-m$, both $n$ and $\bar n$ have a cut with the same parameters and the same associated subspace $[w_i;q_i]_\ell$. 

We thus can apply again Theorem \ref{Lostra}  and see  that $ \er( \bar n)= \er(m_A)+\pi(k) $.  By $(A3)$ we know that $\Pi_{\Pappa,\theta,\mu,\tau}\sum_a\val_a |z_a|^2$ is T\"oplitz, for the chosen parameters  $\Pappa,\theta,\mu,\tau(p_\ell)$,  hence we deduce that  $\val_m= \val_{m_A}$ and   $\val_n= \val_{\bar n}$.
We deduce that if $\pi(k)\in \langle v_i\rangle_\ell$:
 $$|\er(m)|^2-|\er(n)|^2-|\er(m_A)|^2 +|\er(\bar n)|^2\!\!=\!-|\pi(k)|^2 -2\langle \pi(k),\er(m)\rangle+|\pi(k)|^2 +2\langle \pi(k),\er(m_A)\rangle\!=\!0$$
 Finally since $\val_m= \val_{m_A}$ and   $\val_n= \val_{\bar n}$ we have:
$$ |\Omega_m-\Omega_n- \Omega_{m_A}+\Omega_{\bar n}|=  |\tilde\Omega_m-\tilde\Omega_n- \tilde\Omega_{m_A}+\tilde\Omega_{\bar n}|\implies$$
\begin{equation}\label{bohh3}
 |\langle \ome,k\rangle+\Ome_m-\Ome_n|\geq |\langle \ome,k\rangle+\Ome_{m_A}-\Ome_{\bar n}|-  |\tilde\Omega_m-\tilde\Omega_n- \tilde\Omega_{m_A}+\tilde\Omega_{\bar n}|. \end{equation}

By $(A4)$  we also know that $\tilde\Omega(z):=\sum_b \tilde\Omega_b|z_b|^2 $ is quasi-T\"oplitz with parameters $\Pappa,\theta,\mu$ which satisfy \eqref{palleK}
hence, we may apply Lemma \ref{diago} with $Q(z)=\tilde\Omega(z)$.    Lemma \ref{minko}  ensures that, for any allowable $\theta,\mu$,  all $m\in A^g$ satisfy the conditions needed to obtain formula \eqref{appr1} with $N={\Pappa}, \tau= \tau(p_\ell)$, and also          the estimate \eqref{lastima} (with $\tau=\tau(p_\ell)$):
$$ |\tilde\Omega_m- {\tilde\Ome_{m_A}}|<  2 \|\tilde\Omega\|^{T  }_{\vec p } {\Pappa}^{-4d\tau(p_\ell)}.$$
Similarly we have
$$ |\tilde\Omega_n- {  \tilde\Omega}_{\bar n}|<2 \|\tilde\Omega\|^{T  }_{\vec p } {\Pappa}^{-4d\tau(p_\ell)}.$$   In conclusion  when $\pi(k)\in \langle v_i\rangle_\ell$ we have
$$ |\Omega_m-\Omega_n- \Omega_{m_A}+\Omega_{\bar n}|<  4\|\tilde\Omega\|^{T  }_{\vec p }{\Pappa}^{-4d\tau(p_\ell)},$$
where by definition  ${\Pappa}^{\tau(p)}= \max({\Pappa}^{\tau_0},\mathtt c^{-1}
 p )$. We now apply the constraint \eqref{third} and hence:
\begin{equation}\label{bohh2}
 |\langle \ome,k\rangle+\Ome_m-\Ome_n|\geq |\langle \ome,k\rangle+\Ome_{m_A}-\Ome_{\bar n}|- 4\|\tilde\Omega\|^{T  }_{\vec p } {\Pappa}^{-4d\tau(p_\ell)}\geq 
 \end{equation}
 $$ 2\gamma \min( K^{-2d\tau_0},\mathtt c^{2d} p_\ell ^{-2d})- 4\|\tilde\Omega\|^{T  }_{\vec p } \min({\Pappa}^{-4d\tau_0},\mathtt c^{4d}
 p_\ell^{-4d} ). 
  $$
By \eqref{topo}, $\|\tilde\Omega\|^{T  }_{\vec p }<\gamma$, and clearly   $4 \min({\Pappa}^{-4d\tau_0},\mathtt c^{4d}
 p_\ell^{-4d} )<   \min( K^{-2d\tau_0},\mathtt c^{2d} p_\ell ^{-2d})$.  Hence 
$$  |\langle \ome,k\rangle+\Ome_m-\Ome_n|\geq \gamma \min( K^{-2d\tau_0},\mathtt c^{2d} p_\ell ^{-2d}) ,$$  so in order  to we get the desired inequality we need to show that $\min( K^{-2d\tau_0},\mathtt c^{2d} p_\ell ^{-2d})\geq   
{\Pappa}^{-{2d\tau_1}}$ i.e. that  $ \mathtt c^{2d} p_\ell ^{-2d} \geq   
{\Pappa}^{-{2d\tau_1}}$.
Since  $p_\ell\leq \mathtt C K^{\tau_1/4d}\implies \mathtt c^{2d} p_\ell ^{-2d} \geq  (\mathtt c \mathtt C^{-1})^{2d} K^{-\tau_1/2}   $ and (cf. \eqref{itau}) $\Pappa \geq N_0> \mathtt C \mathtt c^{-1}   $ implies $(\mathtt c \mathtt C^{-1})^{2d} K^{-\tau_1/2}\geq  K^{-\tau_1/2-2d} \geq  K^{-2d\tau_1 }  $.

\end{proof}
\vskip15pt

\begin{remark}
This Proposition essentially says that, by imposing \textbf{{\em only
one}} non resonant condition \eqref{third}, we impose \textbf{all}
 the conditions
\eqref{zero1} with $l= e_m-e_n$ such that $m\in [v_i;p_i]_j^g$ and
$n=m+\pi(k)$. 
\end{remark}

\subsection{The T\"oplitz property for  the generating function $F$}\label{s:homol}

The function $F$ has been  obtained by solving the   homological equation for a hamiltonian $H$ compatible with the parameters   $({\Pappa},\theta,\mu)$ and given in Formula \eqref{laF}.  Recall we are using parameters $\vec p= (s,r,\PaPi,\theta,\mu,\lambda=\gamma^{-1}M,\mathcal O)$. We now prove: 
 
\begin{proposition}\label{submain}
For $\xi\in \mathcal O_+$
the solution of the homological equation
$F$ is quasi-T\"oplitz for  parameters $({\Pappa},\theta,\mu)$, moreover one has the bound (cf. \eqref{glieps}:
\begin{equation}\label{booo}\|X_{F^{(i)}}\|^{T  }_{\vec p\,' } \leq   \,\gamma^{-1}\eufm K \|X_{P^{(i)}}\|_{\vec p}^T= \eufm K\ix^{(i)},\quad \eufm K=  5{\Pappa}^{4d{\tau_1}+1}\,,
\end{equation} where $ \vec p\,'= (s,r,\PaPi,\theta,\mu,\lambda=\gamma^{-1}M,\mathcal O_+)$.
\end{proposition}
\begin{proof}
We have   given  in Formula  \eqref{effelambda} a better bound on the  norm $\|X_{F^{(i)}}\|_{s,r}$ hence in order to prove our statement we only need to consider the quasi-T\"oplitz norm $\|X_{F^{(i)}}\|^{(K,\theta,\mu)}_{s,r}$ and the Lipschitz norm.

 \ The quasi--T\"oplitz  property is a condition for $N\geq K$, on the
$(N,\theta,\mu,\tau)$--bilinear part of $F^{(i)}$.   
Hence if $i=0,1$ the T\"oplitz norm coincides with the usual majorant norm and \eqref{booo} follows from the bounds \eqref{effelambda}.

 We   are reduced to proving our statement on the quadratic terms:
$$ \Pi_{(N,\theta,\mu,\tau)}F^{(2)}=\!\!\!  \!\!\!  \!\!\!  \sum_{|k|\leq N \,,\atop \min(|\er(n)|,|\er(m)|)> \theta N^{\tau_1 },\ m,n\in\underline p-cut}\!\!\!  \!\!\!  \!\!\!  \!\!\!  \!\!\!  \!\!\!  \!\!\!   \!\!\!  F_{k,m,n}e^{\ii (k,x)}z_m\bar z_n +B_{k,m,n} e^{\ii (k,x)} z_mz_n  \,+\, C_{k,m,n} e^{\ii (k,x)}\bar z_m\bar z_n $$ with
\begin{equation}\label{homol}
F_{k,m,n}=  \frac{P_{k,0,e_m,e_n}}{\langle k,\omega\rangle+\Omega_m-\Omega_n}\,,\  B_{k,m,n}=  \frac{P_{k,0,e_m+e_n,0}}{\langle \ome,k\rangle+\Omega_m+\Omega_n}\,,\  C_{k,m,n}=  \frac{P_{k,0,0,e_m+e_n}}{\langle \ome,k\rangle-\Omega_m-\Omega_n}.
\end{equation}
By hypothesis $\min(|\er(m)|,|\er(n)|)> \theta N^{\tau_1 }$ so in the case of
$B_{k,m,n}$ one has
$$ |B_{k,m,n}|= \frac{|P_{k,0,e_m+e_n,0}|}{|\langle k,\omega\rangle+|\er(m)|^2+|\er(n)|^2 +2\val_m+2\val_n+\tilde\Omega_m+\tilde\Omega_n|}\leq \mathtt c^{-1}|P_{k,0,e_m+e_n,0}| N^{-\tau_1}, $$ since
$$ |\langle k,\omega\rangle+|\er(m)|^2+|\er(n)|^2+2\val_m+2\val_n +\tilde\Omega_m+\tilde\Omega_n|> 2\mathtt cN^{\tau_1 }- |\omega|N-4|\val|_\infty+2|\tilde\Omega|_\infty > \mathtt c N^{\tau_1 }.$$ 
Since $N^{4d\tau-\tau_1}<1$ this means that $\sum_{k,m,n} B_{k,m,n} e^{\ii (k,x)} z_mz_n$ is quasi-T\"oplitz, and we may take    the ``T\"oplitz approximation'' equal to zero (cf. Remark \ref{dupro}).  Since $\mathtt c{\Pappa}^{{2d{\tau_1}}}>1>\gamma $ the final estimate follows by  formula \eqref{effelambda} 
$$\|\sum_{k,m,n} B_{k,m,n}e^{\ii (k,x)} z_mz_n \|^{K,\theta,\mu,\tau}_{s,r}\leq  \mathtt \max({\Pappa}^{{2d{\tau_1}}}\gamma^{-1},\mathtt c^{-1})\|X_{P^{(i)}}\|^\lambda_{s,r}=
{\Pappa}^{{2d{\tau_1}}}\gamma^{-1}\|X_{P^{(i)}}\|^\lambda_{s,r}.$$
  Same argument for $\sum_{k,m,n} C_{k,m,n}e^{\ii (k,x)}\bar z_m\bar z_n$.\smallskip

We thus have to study $\sum_{k,m,n} F_{k,m,n}e^{\ii (k,x)}z_m\bar z_n$.  
 Take $N\geq K$, denote by $\underline p:=N,\theta,\mu,\tau$, we wish to decompose
\begin{equation}
\label{olbF}F_{k,m,n}={\mathcal F}_{k}(m-n,[v_i;p_i]_\ell)+
N^{-4d\tau}\bar F_{k,m,n},
\end{equation}   so that  ${\mathcal F}_{k}$ is the
$k$ Fourier coefficient of a T\"oplitz approximation $\mathcal F\in \mathcal T_{\underline p}$.

 By momentum conservation we have $\pi(k)+\er(m)-\er(n)=0$ hence \begin{equation}
\label{deer}|\er(m)|^2-|\er(n)|^2=   -|\pi(k)|^2- 2\langle \pi(k),\er(m)\rangle  .
\end{equation} For the denominator in the first term of \eqref{homol}  we have $$ \langle k,\omega\rangle+\Omega_m-\Omega_n= \langle k,\omega\rangle+|\er(m)|^2-|\er(n)|^2+2\val_m-2\val_n +\tilde\Omega_m-\tilde\Omega_n $$    
\begin{equation}
\label{deer1}= \langle k,\omega\rangle -|\pi(k)|^2- 2\langle \pi(k),\er(m)\rangle+2\val_m-2\val_n +\tilde\Omega_m-\tilde\Omega_n.
\end{equation} If $\Pappa$ is sufficiently large we can estimate $|\langle k,\omega\rangle -|\pi(k)|^2  +2\val_m-2\val_n +\tilde\Omega_m-\tilde\Omega_n|<\Pappa^3$. From this we see that if $2|(\pi(k),\er(m))|> \mathtt c N^{4d \tau}$  we may again set $\mathcal F_k=0$.

So we are reduced to the case in which $2|(\pi(k),\er(m))|\leq \mathtt c N^{4d \tau}$.
\smallskip

By assumption $m,n$ have a cut at $\ell$ with parameters $(N,\theta,\mu,\tau)$ and  $$
|\er(m)|,|\er(n)|\geq \theta N^{\tau_1}\,,\quad
m\frec{N}[v_i;p_i]\,,\quad n\frec{N}[w_i;q_i]\,,$$ \begin{equation}\label{condi} q_\ell,p_\ell\leq \mu N^{\tau}\,,\quad
q_{\ell+1},p_{\ell+1}\geq \theta N^{4d\tau}\,,\quad A:=[v_i;p_i]_\ell\prec B:= [w_i;q_i]_\ell\,,
\end{equation} 
   hence, by Corollary \ref{dueta}, $\langle v_1,\dots,v_\ell\rangle= \langle w_1,\dots,w_\ell\rangle$. We distinguish two cases:

{\bf Case 1:\quad $\pi(k)\notin \langle v_i\rangle_\ell$.}\quad  If $\pi(k)\notin \langle v_i\rangle_\ell$ then  by the definitions of cut \ref{cut}, and  of optimal presentation 
we have $2|(\pi(k),\er(m))|> \mathtt c N^{4d \tau}$ contrary to our hypothesis. \medskip
   
 {\bf Case 2:\quad $\pi(k)\in \langle v_i\rangle_\ell$.}\quad     We 
  recall that $\val_m$ (resp. $\val_n$) are constant on all the $m$ which have the same affine space  $A=[v_i;p_i]_\ell$ associated to its $\ell$-cut. Moreover, setting $h=n-m$, we know that $n$ has an $\ell$-cut with associated affine space $B=A+h=[w_i;q_i]_\ell$. By lemma \ref{diago}   and $\tilde\Omega$ has a  T\"oplitz approximation, $\it{\tilde \Omega}$ see Formula \eqref{pargol}. By Corollary \ref{pirlino} we can choose a point $m_A\in A^g_{\underline p}$ so that  $m_A+h\in (A+h)^g_{\underline p}$ then we may choose the T\"oplitz approximant of order one  with $\it{\tilde\Ome}(A)=\tilde\Ome_{m_A}$ and $\it{\tilde\Ome}(A+h)=\tilde\Ome_{m_A+h}$:
 \begin{equation}\label{bohH}
 |\tilde\Omega_m- \it{\tilde\Ome}(A)|\,,  |\tilde\Omega_n- \it{\tilde\Ome}(A+h)|<  2 \|\tilde\Omega\|^{T  }_{\vec p } {N}^{-4d\tau}\,.
\end{equation} 
Denote by $D_{k,m,n}=  \langle k,\omega\rangle+\Omega_m-\Omega_n$ the denominator of the term $F_{k,m,n}$ , we define 
  \begin{equation}
\label{zerotr}\mathfrak D_{k,h,A}:= D_{k,m_A,m_A+h}\,,\implies \ 
|\mathfrak D_{k,h,A}|\geq \gamma {\Pappa}^{-{{2d{\tau_1}}}},\ \eqref{zero1}.
\end{equation}
 Finally, since   $P^{(2)}$ is quasi-T\"oplitz,  we may set
  $$ {\mathcal F}_{k}(h,A)=  \frac{{\mathcal P}^{(2)}_{k}(h,A)}{\mathfrak D_{k,h,A}},\quad \bar F  = N^{ 4 d \tau}(F -{\mathcal F})$$ where
  ${\mathcal P}={\mathcal P}^{(2)} \in\mathcal T_{\underline p}$ is a piecewise--T\"oplitz approximation of $P_{k, 0,e_m, e_n}$ so that for $\bar P^{(2)}  = N^{ 4 d \tau}(P^{(2)}-{\mathcal  P}^{(2)}) $  we have the bounds  $\|\mathcal P^{(2)}\|_{r,s},\|\bar P^{(2)}\|_{r,s}\leq \|P^{(2)}\|^T_{\vec p}+\epsilon$ where $\epsilon>0$ can be taken arbitrarily small (see Lemma \ref{diago}).

We notice that by \eqref{bohH}
\begin{equation}
\label{stimO}
N^{4 d\tau}|\mathfrak D_{k,h,A} -D_{k,m,n}|=N^{4 d\tau}  |\tilde\Omega_m- \it{\tilde\Ome}(A)-\tilde\Omega_n+ \it{\tilde\Ome}(A+h)|<4 \|\tilde\Ome\|_{\vec p}^T.
\end{equation}

 If the denominators are  bounded away from  zero then (cf. \eqref{olbF}):
$$\bar F_{k,m,n} = N^{ 4 d \tau}(F_{k,n,m}-{\mathcal F}_{k}(h,A ) ) =  \frac{\bar P_{k,m,n}^{(2)}}{D_{k,m,n}}  + {\mathcal P}^{(2)}_{k}(h,A)\frac{N^{ 4 d \tau}(\mathfrak D_{k,h,A} -D_{k,m,n})}{\mathfrak D_{k,h,A} D_{k,m,n}}, $$ is bounded. 
 
Summing over the indexes $k,m,n$ such that $m$ has a cut with parameters $(N,\theta,\mu)$,  we obtain $$\| \mathcal F \| _{s,r} \leq  \frac{\| \mathcal P^{(2)} \| _{s,r}}{\inf_{k,n,m} \mathfrak D_{k,h,A}},\quad 
\| \bar F \| _{s,r} \leq \frac{ \| \bar P^{(2)} \| _{s,r} }{\inf_{k,n,m} |\mathfrak D_{k,h,A}|}+\frac{4 \|\tilde\Ome\|_{s,r}^T\| \mathcal P^{(2)} \| _{s,r}}{ \inf_{k,n,m} |D_{k,m,n}\mathfrak D_{k,h,A}|}$$

This we may rephrase as
 \begin{equation}\label{pacchia}
\| F^{(2)} \|^{T  }_{\vec p\,' } \leq    \| P^{(2)}\|_{\vec p}^T\sup_{\xi\in \O _+}\sup_{k,n,m: |k|<K}(\frac{1}{ |\mathfrak D_{k,m,n}|}+\frac{ 4\|\tilde\Ome\|_{\vec p}^T }{  |\mathfrak D_{k,m,n} D_{k,m,n}|}+ \lambda \frac{M(|k|+2)}{D_{k,m,n}^2}).
\end{equation} 
   By the { \it Smallness condition},     $(A4)$   formula \eqref{topo}, we have $\|\tilde\Ome\|_{\vec p}^T \leq \gamma$.
  The denominators $|\mathfrak D_{k,m,n}|, | D_{k,m,n}|$  are $ >\gamma {\Pappa}^{-{2d{\tau_1}}}$ uniformly in $\O _+$ by Formulas \eqref{zero1} and \eqref{zerotr}. Recalling that $\lambda= \gamma M^{-1}$, we deduce that 
$$\| F^{(2)} \|^{T  }_{\vec p\,' } \leq    \| P^{(2)}\|_{\vec p}^T\gamma^{-1}({\Pappa}^{ {2d{\tau_1}}}+ {\Pappa}^{ {4d{\tau_1}+1}}+4{\Pappa}^{ {4d{\tau_1}}})\leq   5\| P^{(2)}\|_{\vec p}^T\gamma^{-1} {\Pappa}^{ {4d{\tau_1}+1}}. $$ 

\end{proof}
 \subsection{The new Hamiltonian $H^+$\label{s:Kstep}} Recall we have set $\vec p=(s,r,\PaPi,\theta,\mu,\lambda,\mathcal O)$,  $\vec p_+=(s_+,r_+,\PaPi,\theta_+,\mu_+,\lambda_+,\mathcal O_+)$. By Propositions \ref{main} and \ref{submain}, $F$ defines a M-analytic symplectic quasi-Toplitz  change of variables from $D(s_+,r_+)$ to $D(s,r)$, where $r_+,s_+$ are determined by \eqref{rpiu}. 

The change of variables is of the form $\Phi= I+ \Psi$ with the bounds (cfr. \eqref{flussoT1}) 
\begin{equation}\label{babanu}
\| \Psi \|^T_{\vec p_+} \leq  2\|X_F\|_{\vec { p }}^T   
\end{equation} 
and for any function $ f \in\mathcal T_{\underline p}$ we have that $ e^{ad(F)} f\in\mathcal T_{\vec p_+}$ where $\vec p_+$ satisfy  \eqref{nubo}.
\smallskip
 
We now analyze $H^+:= e^{ad(F)} (H)$,  recall that by definition $ad(F)(\mathcal N)= - P^{\leq 2}_{\leq {\Pappa}}+[P^{\leq 2}]$.
\begin{equation}\label{Pnupiu}
H^+:= e^{ad(F)} (\mathcal N+P)= H+ ad(F)H+\sum_{j\geq 2} \frac{ad(F)^{j }}{j!}(H)= \end{equation}
$$ = \mathcal N+P  - P^{\leq 2}_{\leq {\Pappa}}+[P^{\leq 2}]+ad(F)P+ \sum_{j\geq 2} \frac{ad(F)^{j }}{j!}(H)$$
 
We call $\mathcal N+[P^{\leq 2}] := \mathcal N^+$ and the rest of the Hamiltonian $P_+$, so that
\begin{equation}
\label{Ppi}P_+:= (P-P^{\leq 2}_{\leq {\Pappa}}) + \{F,P\}+
 \sum_{j\geq 2} \frac{ad(F)^j}{j!} P  +  \sum_{j\geq 2} \frac{ad(F)^{j-1}}{j!}(-P^{\leq 2}_{\leq {\Pappa}}+[P^{\leq 2}_{\leq {\Pappa}}]).
\end{equation} 

By formula \eqref{nucleo}, 
\begin{equation}\label{opiu}
\ome^+:= \ome +  P^{1,0}_0 \,,\quad  \tilde\Ome_n^+:= \tilde\Ome_n + P^{0,2}_{0,n,+,n,-}\,,$$$$ \Ome^+_n=\s(n)|\er(n)|^2+ 2\val_n +\tilde\Ome^+_n= \Ome_n+  P^{0,2}_{0,n,+,n,-}\,,\quad\mathcal C^+=\mathcal C+\mathcal P.
\end{equation}    
Recall that $\mathcal C$ is the finite complex part of the normal form, same for $\mathcal P$ as defined in \eqref{nucleo}.
We need to
 \begin{itemize}
\item Prove that $H^+$ satisfies conditions $(A1)-(A5)$.

\item Estimate all the new parameters $\vec p_+$.
\end{itemize}
 
  \subsubsection{Lipschitz estimates $M_+,L_+,\lambda_+ $}     
We have
 $$ |\ome^+|^{lip}+|\Ome_+|^{lip}=|\ome +  P^{1,0}_0|^{lip}+ \sup_{n}|\Ome_n+ P^{0,2}_{0,n,+,n,-}|^{lip} \leq$$ $$ |\ome|^{lip}+ \sup_{n}|\Ome_n|^{lip} + |P^{1,0}_0|^{lip}+ \sup_n |P^{0,2}_{0,n,+,n,-}|^{lip} \leq  M+ |P^{1,0}_0|^{lip}+ \sup_n |P^{0,2}_{0,n,+,n,-}|^{lip}, $$  by definition of $M$. 
 
 In the same way
 $$ (\ome^+)^{-1} =(\ome + P^{1,0}_0)^{-1}=(\ome\circ (Id + P^{1,0}_0\circ \ome^{-1}))^{-1}= (Id + P^{1,0}_0\circ \ome^{-1})^{-1}\circ \ome^{-1}, $$ so that $\ome^+$ is invertible as a Lipschitz function provided that $ L |P^{1,0}_0|^{lip}<1$ with the bound
 \begin{equation}
\label{ompl}|(\ome^+)^{-1} |^{lip}\leq \frac{L}{1- L |P^{1,0}_0|^{lip}}. 
\end{equation} 
In order to estimate $a_+$, defined in Formula  \eqref{pota}, we note 
    \begin{equation}\label{pota1}
   | \Delta_{\xi,\varrho}(\langle \ome^+(\xi ) , k\rangle + ( \Ome_+ (\xi ) , l ))| \end{equation}$$\geq  | \Delta_{\xi,\varrho}(\langle \ome (\xi ) , k\rangle + ( \Ome  (\xi ) , l ))| - |\Delta_{\xi,\varrho}( \langle P^{1,0}_0(\xi ), k \rangle +  (P^{0,2}_{0,n,+,n,-} , l ))| 
$$
$$
 \geq a-  {\esse} | P^{1,0}_0(\xi ) |^{lip} + 2 | P^{0,2}_{0,n,+,n,-}   |^{lip} 
$$

We recall that
  $\|\cdot \|^\lambda = \| \cdot\| +\lambda \| \cdot \|^{lip},$ hence  $$  |P^{1,0}_0|^{lip}, |P^{0,2}_{0,n,+,n,-}|^{lip} \leq \lambda^{-1} \| P^{(2)}\|^\lambda_{s,r}\leq M |\vec\ix|\,$$
since  $\lambda = \gamma M^{-1}$. We define
\begin{equation}\label{emmepiu}
M_+ :=  M( 1+2 |\vec\ix|)  \,, \quad L_+ := 
L( 1- ML |\vec\ix|)^{-1},\quad   a_+:= a- (S_0+2)M |\vec\ix|
\end{equation}
 notice that $L_+$ is well defined since by (A4)  $LM |\vec\ix|<1$. By construction
 $$  |\ome^+|^{lip}+|\Ome_+|^{lip}\leq M_+\,,\quad  | (\ome^+)^{-1} |^{lip} \leq L_+\,, $$
 finally we have:
 $$ |\ome^+ -\vgot|\leq |\ome-\vgot| +|P^{1,0}_0| \leq M \varepsilon^2 + \gamma  |\vec\ix|  \leq M_+ \varepsilon^2,$$ 
 since $\gamma < 2 M \varepsilon^2$. We finally set $\lambda_+ := \gamma (M_+)^{-1}$ we have $\lambda_+< \lambda$ since $M_+>    M          $.

  \subsubsection{T\"oplitz estimates $\vec\ix_+,\Theta_+$}
  We wish to show $P_+$ is quasi--T\"oplitz and   bound $$ \ix^{(h)}_+:= \gamma^{-1} \| X_{P^{h}_+}\|^{T  }_{\vec p_+} \,,\quad {\rm for}\; h=0,1,2 \, \,;\quad \Theta_+:=  \gamma^{-1}\|X_{P_+}\|^{T  }_{\vec p_+}.$$

 We  have that (cf. \eqref{Ppi}) $P_+= (P-P^{\leq 2}_{\leq {\Pappa}})    +A+B$  where
 $$A:= \sum_{j\geq 2} \frac{ad(F)^{j-1}}{j!}(-P^{\leq 2}_{\leq {\Pappa}}+[P^{\leq 2}_{\leq {\Pappa}}]),\quad B= \{F,P\}+
 \sum_{j\geq 2} \frac{ad(F)^j}{j!} P  ,$$
  We argue as in Proposition 5 of \cite{PX} or in Proposition \ref{main}. By Formula \eqref{booo} $\|X_{F }\|^{T  }_{\vec p } \leq    5{\Pappa}^{4d{\tau_1}}|\vec \ix |, $ the hypothesis  \eqref{rpiu} implies that we have the conditions of \eqref{funm}, that is 
$ 2^{2n+14}     \delta^{-1}    \|X_{F}\|^T_{\vec { p}}< 1/2 $, where $\delta = \min(1-\frac{s_+}{s},1-\frac{r_+}{r})$. Therefore 
we have that  
\begin{equation}
\label{pppr}\|X_A\|^{T  }_{\vec p_+}\leq   2  \delta^{-1} \|X_F\|^{T  }_{\vec p }\|X_{P^{\leq 2}}
\|^{T  }_{\vec p },\quad \|X_{B}\|^{T  }_{\vec p_+}\leq  2 \delta^{-1} \|X_F\|^{T  }_{\vec p }\|X_{P}\|^{T  }_{\vec p } 
\end{equation} Using \eqref{booo} we rewrite \eqref{pppr}    as 
 \begin{equation}
\label{Pppr} \gamma^{-1}\|X_A\|^{T  }_{\vec p_+} \lessdot \delta^{-1} \eufm K|\ix|^2\,,\quad  \gamma^{-1} \|X_{B}\|^{T  }_{\vec p_+}\lessdot  \delta^{-1} \eufm K|\ix|\Theta\,,
\end{equation}

  We obtain:
\begin{equation}\label{thetanu} \| X_{P_+}\|_{\vec p_+}^{T }\leq \|X_P\|_{\vec p }^{T } + 4   \delta^{-1} \|X_F\|_{\vec p }^{T } \|X_P\|_{\vec p }^{T }\quad i.e.\quad \Theta_+ - \Theta  \lessdot   \delta^{-1}\eufm K  |\vec \ix| \Theta \,.
\end{equation}
Let us now compute the terms of order $\leq 2$ in $P^+$,
we have:
$$P^{(h)}_+=  P^{(h)}_{> {\Pappa}}+( A+ \{F,P^{>2}\} +\{F, P^{\leq 2}\}+  \sum_{j\geq 2}
 \frac{ad(F)^{j}}{j!} P)^{(h)}. $$ 
Again from \eqref{booo}, denoting $E= \sum_{j\geq 2}
 \frac{ad(F)^{j}}{j!} P$, we have the bounds:
$$\|X_{ \{F, P^{\leq 2}\}}\|^T_{\vec p_+}\leq 2 \delta^{-1} \|X_{ F}\|^{T  }_{\vec p }\ \|X_{  P^{\leq 2}}\|^{T  }_{\vec p }\lessdot \delta^{-1} \eufm K \gamma |\vec \ix|^2,  $$
$$  \|X_E\|^T_{s_+,r_+}\lessdot \delta^{-2} (\|X_F\|^{T  }_{\vec p })^2
\|X_{P}\|^{T  }_{\vec p } \lessdot \delta^{-2}  \eufm K^2|\vec \ix|^2
\Theta $$ 
The contributions from $\{F,P^{>2}\}$ are
$$\Pi_{0} \{F,P^{>2}\}=  0 \,,
\quad \Pi_{1}\{F,P^{>2}\}= \{F^{(0)},P^{(3)}\} \,,$$
$$  \Pi_{2}\{F,P^{>2}\}= \{F^{0},P^{(4)}\}+
 \{F^{(1)},P^{(3)}\}\,,$$

so applying the Cauchy estimates we have, setting $ \mathfrak z:= {\rm const.}(\delta^{-1}  \eufm K)^2$:
$$ \ix^{(0)}_+ \leq \qquad \qquad \qquad\ \ \qquad  {\mathfrak z}|{\vec \ix}|^2(1+ \Theta)+ 2\ix^{(0)}   e^{-{(s-s_+){\Pappa}}}  $$
$$\ix^{(1)}_+  \leq {\mathfrak z}  \big( \Theta\,  \ix^{(0)} +\qquad\quad\quad  \ |{\vec \ix}|^2 (1+ \Theta)\big) +2 \ix^{(1)}  e^{-(s-s_+)\Pappa }$$
$$ \ix^{(2)}_+  \leq  {\mathfrak z}   \big( \Theta ( \ix^{(0)}+\ix^{(1)})  +|{\vec \ix}|^2 (1+ \Theta)\big) \,+2 \ix^{(2)}  e^{-(s-s_+)\Pappa }.$$
Note that the terms $2 \ix^{(h)}  e^{-(s-s_+)\Pappa }$ come from $P^{(h)}_{> {\Pappa}}$ via the smoothing estimates \ref{smoothT}.

We write in matrix form, denoting by $\vec{\ix}$ the three dimensional column vector of coordinates $({\ix}^{(0)},{\ix}^{(1)},{\ix}^{(2)})$ and $\underline 1:=(1,1,1)$ we have
\begin{equation}\label{sonasega}
\begin{array} {lll}\vec{\ix}_+ &\leq &  {\mathfrak z} \big( \Theta \eufm L \vec{\ix} + |\vec \ix|^2(1+\Theta)\underline 1\big)+   2 e^{-(s-s_+)\Pappa } \vec \ix   \\ \Theta_+ & \leq&  \Theta+ {\mathfrak z}  \Theta |\vec \ix| \end{array}
\end{equation} 
where  the matrix $\eufm L$ is
\begin{equation}
\label{lael}\eufm L=\begin{pmatrix} 0&0&0 \\ 1 &0&0 \\ 1 &1&0
\end{pmatrix}\implies \eufm L^2=\begin{pmatrix} 0&0&0 \\ 0 &0&0 \\ 2 &0&0
\end{pmatrix},\quad \eufm L^3=0
\end{equation} 
and the vector inequality means that the inequality is true for all three coordinates.

\subsection{Iteration}

\subsubsection{A useful inequality} Now we need to be able to handle in a recursive way the inequalities obtained so far, we start with a formal inequality   which is a variation of Lemma 5.8 of \cite{BBP1}.\smallskip

We \textsl{fix} $\chi$ such that
\begin{equation}\label{chi}
1 < \chi < 2^{\frac{1}{3}} \, .
\end{equation} and $\mathfrak a,\mathfrak b,\mathfrak c$ three positive numbers satisfying:
\begin{equation}\label{limiti}
\mathfrak a,\mathfrak b,\mathfrak c>1\,,\quad 12  {\mathfrak b } ^2    \leq \min_{i\in\N}  \frac{e^{  {\mathfrak a } 2^{i-2}+ \chi^{i-2}- \chi^{i+1}}}{\mathfrak c^{2i-1} }\,,
\end{equation} and set 
\begin{equation}\label{mini}
\mathfrak i:=\min_{i\in\N} (\frac{e^{\chi^i}}{2{\mathfrak b }  (2\mathfrak c) ^i}, \frac{e^{(2-\chi^3)\chi^{i-2}}}{32 {\mathfrak b } ^3 \mathfrak c^{3i-3}} ) >0.
\end{equation}
 \begin{lemma}\label{pomodorino}  
For $j\in \mathbb N$ consider a  sequence $(\vec{\ix}_j,\Theta_j)$ with $ \vec{\ix}_j:=  (\ix^{(0)}_j,\ix^{(1)}_j,\ix^{(2)}_j)$ a vector and $ \Theta_j $ a number, all with positive components. Set $ |\vec{\ix}_j|:= \ix^{(0)}_j+\ix^{(1)}_j+\ix^{(2)}_j$.

Suppose  that for (\ $\eufm L$ as in \eqref{lael})   we have:\begin{equation}\label{passoj}
\left\{ \begin{array} {lll}\vec{\ix}_{j+1} &\leq &  {\mathfrak b } \mathfrak c^{j}( \Theta_{j} \eufm L \vec{\ix}_{j}  +|\vec{\ix}_{j}|^2(1+\Theta_{j})\underline 1)+  e^{-{\mathfrak a } 2^j} \vec \ix_j \\ \Theta_{j+1} & \leq&  \Theta_{j} +  {\mathfrak b } \mathfrak c^{j} \Theta_{j} |\vec{\ix}_j| . \end{array}\right.
\end{equation}

There exist  $   \Art := \Art (\mathfrak c,\chi, {\mathfrak a } ,{\mathfrak b } )>1$ such that for all $   {\art} >0  $
satisfying \begin{equation}
\label{chee} 2{\art} \Art <\min(1/3,\mathfrak i)
\end{equation}  we have that \begin{equation}\label{Cgotbis}
1/3|\vec{\ix}_0|,  \Theta_0<{\art} \ \ \Longrightarrow \ \
 |\vec{\ix}_j| \leq  \, \Art |\vec{\ix}_0|  e^{-\chi^j} \, ,\quad  \Theta_j  \leq \Theta_0 (1+ \Art  \sum_{0<l\leq j}2^{-l}) \,,\quad  \forall\,  j\geq 0.
\end{equation}
\end{lemma}   Let $i_0$ be the value of $i$ for which the minimum $\mathfrak i$ in \eqref{mini} is achieved. One easily sees that, since $|\vec{\ix}_0|,\Theta_0 \leq 3 \art<1$  we can find a value $ \Art $ (depending only on $\mathfrak a,\mathfrak b,\mathfrak c,\chi$)  for which both relations \eqref{Cgotbis} hold for all $i\leq i_0 $ and  $|\vec{\ix}_0| <1$.

We now work by induction and suppose that both relations hold up to some $i\geq i_0$. Then $\Theta_i \leq \Theta_0 (1+ \Art  \sum_{1\leq l\leq i}2^{-l})<2\Theta_0   \Art $ and, assuming $2\Theta_0\Art  \leq 2{\art}   \Art \leq \mathfrak i   $, we have ${\mathfrak b } \mathfrak c^i  2^{i+1}  e^{-\chi^i}\leq  \mathfrak i ^{-1}\leq  (2\Theta_0\Art )^{-1}  $ estimate :
$$\Theta_{i+1} \leq  \Theta_0 (1+ \Art  \sum_{l\leq i}2^{-l}  + {\mathfrak b } \mathfrak c^i  2\Theta_0 A^2_0 e^{-\chi^i})\leq  \Theta_0 (1+ \Art  \sum_{1\leq l\leq i+1}2^{-l}). $$  Notice that the constraint \eqref{chee}  and the inequality \eqref{Cgotbis} imply  $\Theta_j<1$ for all $j$.
We now   substitute $\Theta_j<1$ for all $j\leq i$ in the first relation and  get 
$$\vec\ix_{j+1}\leq {\mathfrak b } \mathfrak c^{j}(  \eufm L \vec{\ix}_{j}  +2|\vec{\ix}_{j}|^2\underline 1)+ e^{-{\mathfrak a } 2^j}  \vec\ix_{j}  $$

we obtain a  bound for
$\vec\ix_{i+1}$ in terms of  $\vec\ix_i$, $\vec\ix_{i-1}$ and $\vec\ix_{i-2}$.

We now assume by induction that the bounds in   \eqref{Cgotbis} are satisfied for all $j\leq i$ and then: 
$$\vec\ix_{i+1}\leq  {\mathfrak b } ^2\mathfrak c^{2i-1}  (\eufm L^2  \vec{\ix}_{i-1}  + 2|\vec{\ix}_{i-1}|^2 \eufm L \underline 1) + {\mathfrak b }  \mathfrak c^i  e^{-{\mathfrak a } 2^{i-1}} \eufm L \vec \ix_{i-1}  +2{\mathfrak b } \mathfrak c^i |\vec{\ix}_{i}|^2\underline 1+  e^{-{\mathfrak a } 2^i}  \vec\ix_{i} \leq $$
$$\vec\ix_{i+1}\leq  {\mathfrak b } ^2\mathfrak c^{2i-1}  \eufm L^2  \vec{\ix}_{i-1}  +2 {\mathfrak b } ^2\mathfrak c^{2i-1}  |\vec{\ix}_{i-1}|^2 \eufm L \underline 1 + {\mathfrak b }  \mathfrak c^i  e^{-{\mathfrak a } 2^{i-1}} \eufm L \vec \ix_{i-1}  +2{\mathfrak b } \mathfrak c^i |\vec{\ix}_{i}|^2\underline 1+  e^{-{\mathfrak a } 2^i}  \vec\ix_{i} \leq $$
$$   2{\mathfrak b } ^3\mathfrak c^{3i-3}   |\vec{\ix}_{i-2}|^2\eufm L^2\underline 1+{\mathfrak b } ^2\mathfrak c^{2i-1}  e^{-{\mathfrak a } 2^{i-2}}  \eufm L^2   \vec\ix_{i-2}+2 {\mathfrak b } ^2\mathfrak c^{2i-1}  |\vec{\ix}_{i-1}|^2 \eufm L \underline 1   $$
$$ + {\mathfrak b }  \mathfrak c^i  e^{-{\mathfrak a } 2^{i-1}} \eufm L \vec \ix_{i-1}  +2{\mathfrak b } \mathfrak c^i |\vec{\ix}_{i}|^2\underline 1+  e^{-{\mathfrak a } 2^i}  \vec\ix_{i}. $$
This in turn implies, since $|\eufm L|=3, |\eufm L^2|=2$, that
$$ | \vec \ix_{i+1}| \leq  4{\mathfrak b } ^3\mathfrak c^{3i-3}   |\vec{\ix}_{i-2}|^2 +6 {\mathfrak b } ^2\mathfrak c^{2i-1}  |\vec{\ix}_{i-1}|^2  +6{\mathfrak b } \mathfrak c^i |\vec{\ix}_{i}|^2 $$
$$ +2{\mathfrak b } ^2\mathfrak c^{2i-1}  e^{-{\mathfrak a } 2^{i-2}} |  \vec\ix_{i-2}|+3 {\mathfrak b }  \mathfrak c^i  e^{-{\mathfrak a } 2^{i-1}} | \vec \ix_{i-1} |  +  e^{-{\mathfrak a } 2^i}  |\vec\ix_{i}| \leq $$
$$| \vec \ix_{i+1}| \leq  |\vec\ix_0|^2 \Art ^2(4 {\mathfrak b } ^3\mathfrak c^{3i-3}  e^{-2 \chi^{i-2}} + 6 {\mathfrak b } ^2\mathfrak c^{2i-1}e^{-2 \chi^{i-1}}+6{\mathfrak b } \mathfrak c^{i}e^{-2 \chi^{i}})+ $$
$$ |\vec\ix_0|\Art (2 {\mathfrak b } ^2\mathfrak c^{2i-1}  e^{-{\mathfrak a } 2^{i-2}- \chi^{i-2}}  + 3{\mathfrak b } \mathfrak c^{i}  e^{-{\mathfrak a } 2^{i-1}- \chi^{i-1}} + e^{-{\mathfrak a } 2^{i}-\chi^i})\leq$$
$$ |\vec\ix_0|  \Art [ 16 |\vec\ix_0|  \Art   {\mathfrak b } ^3\mathfrak c^{3i-3}  e^{-2 \chi^{i-2}} +
 6  {\mathfrak b } ^2\mathfrak c^{2i-1}  e^{-{\mathfrak a } 2^{i-2}- \chi^{i-2}}]  \leq |\vec\ix_0|\Art  e^{-  \chi^{i+1}}$$
 This is achieved provided that,  
 $$[ 16{\art}  \Art   {\mathfrak b } ^3\mathfrak c^{3i-3}  e^{-2 \chi^{i-2}} +
 6  {\mathfrak b } ^2\mathfrak c^{2i-1}  e^{-{\mathfrak a } 2^{i-2}- \chi^{i-2}}]\leq e^{-  \chi^{i+1}}$$ and this in turn is valid if we assume the constraint \eqref{limiti}.
 \bigskip

We will apply this to
\begin{equation}\label{kappa}
 \mathfrak c := 4^{1+4d\tau_1}\,,\quad {\mathfrak b } :=   cost\,  {\Pappa}^{8d\tau_1}\,, \quad \mathfrak a:=    \frac{\Pappa s_0}{32}\,.
\end{equation}
We note that, with this choice of parameters, condition \eqref{limiti}  amounts to a largeness condition on ${\Pappa}$.

\subsubsection{Parameters in the iteration}

Let  $ H_0 = \mathcal N_0 + P_0 : D_0 \times\O_{0} \to {\mathbb C} $ be as in Theorem \ref{KAM}. Define
\begin{equation}
\label{parini}\ix^{(h)}_0 :=  \frac{\|X_{P^{(h)}_0}\|^T_{\vec p_0}}{\gamma} \leq \Theta_0:=  \frac{\|X_{P_0}\|^T_{\vec p_0}} {\gamma}= \frac{\|X_{P_0}\|_0} {\gamma} .
\end{equation}

 We have the estimate  $e^{-1}\leq   \max_\nu   e^{ -\chi^\nu}2^{\nu(1+4d\tau_1 )}<\infty$  since, for any $p>0$, we have that $\lim_{\nu\to\infty}2^{ p\,\nu } e^{ -\chi^\nu}=0$. We define: 
 \begin{equation}\label{cstar}
C_\star=  2^{2n+10}e \Art  M_0 a_0^{-1} \kappa {\Pappa}_0 ^{4d\tau_1}  \max_\nu   e^{ -\chi^\nu}2^{\nu(1+4d\tau_1 )} \,,
\end{equation}
here  $\Art $ is the constant of Lemma \ref{pomodorino} with the choice of parameters \eqref{kappa}, note that it  depends only on $\Pappa_0,\kappa,d,n,\tau_1,\tau_0,s_0$ and $ M_0 a_0^{-1}$. Recall that, as we have stated in Remark \ref{eind}  $M_0 a_0^{-1}>1$ is an $\e$ independent constant of the problem.

We now fix   $\art:= \art(\Pappa_0,\kappa,d,n,\tau_1,\tau_0,s_0,a_0)$ in order to ensure the {\em smallness conditions}:\begin{equation}\label{alcos}
    C_\star \art<(12\,e)^{-1} ,\quad\prod_\nu (1+ C_\star \art e^{-\chi^{\nu-1}})<\sqrt{2}.\end{equation}
 together with the condition \eqref{chee} of Lemma \ref{pomodorino}.
 
We now need to estimate all the parameters in the iteration. For the parameters which increase we exhibit bounds from above and for the ones which decrease bounds from below. Thus for $ \nu \in \mathbb{N}$  we define

\begin{itemize}

\item  \ $  \d_\nu := 2^{-\nu-3} \,, \quad r_{\nu+1} := (1-\d_\nu)r_\nu \,,\quad
s_{\nu+1} := (1-\d_\nu)s_\nu  \,, \quad
D_\nu := D(s_\nu,r_\nu)  \,,$ \vskip15pt

\item   \
$ M_\nu:=M_{\nu-1} (1+ C_\star|\vec\ix_0| e^{-\chi^{\nu-1}})\leq \sqrt 2 M_0$ , $\lambda_\nu:=\frac{\gamma}{M_\nu}$,

\item \  $ L_\nu:=L_{\nu-1} (1- C_\star |\vec\ix_0| e^{-\chi^{\nu-1}})^{-1} \leq \sqrt 2 L_0$ \vskip15pt

\item  \
$ {\Pappa}_\nu:=4^\nu{\Pappa_0} \,,\,,\quad  \theta_\nu = \theta_0 (1+ \sum_{j\leq \nu}2^{-j})\,,\quad   \mu_\nu = \mu_0 (1-  \sum_{j\leq \nu}2^{-j})\,,$\vskip15pt

\item  \
 $a_\nu:=a_0(1-  C_\star |\vec \ix_0| \sum_{j\leq \nu} 2^{-j})$,
\vskip15pt

\item  \ $ \mathfrak z_\nu=  {\rm  const. }  \delta_\nu^{-2}  \eufm K_\nu^2= {\mathfrak b } \mathfrak c^\nu$.\end{itemize}

\vskip35pt

We have made our definitions so that $\min( 1- \frac{r_{\nu+1}}{r_\nu}, 1-\frac{s_\nu}{s_{\nu+1}})=\delta_\nu$. 
Note that  $     r_{\nu+1}  \searrow r_0
\prod_{\nu=0}^\infty (1-\d_\nu)> \frac{r_0}{2} \,, \quad
s_{\nu+1}  \searrow s_0
\prod_{\nu=0}^\infty (1-\d_\nu)> \frac{s_0}{2}$,
$\theta_\nu \nearrow 3\theta_0/2< \mathtt C= 2\theta_0 \,,\quad   \mu_\nu  \searrow 2 \mu_0/3>\mathtt c=\mu_0/2$.
\vskip15pt

For compactness of notation we will denote  
\begin{equation}\label{normaj}
\| \cdot \|_j:= \| \cdot \|^T_{\vec p_j} \,,\quad \underline p_j:= (r_j,s_j,\O_j,{\Pappa}_j,{\esse}=16\sqrt{n},\theta_j,\mu_j, a_j,M_j, L_j,\mathtt c,\mathtt C),
\end{equation} 
where $\O_j$ is defined in 
the course of the proof of Lemma \ref{lem:iter} .

\begin{lemma}{\bf (Iterative Lemma)}\label{lem:iter}
 Let $\Art,C_\star,\art$ be  fixed as in formulas  \eqref{cstar} and \eqref{alcos}. Let $\Gamma$ be as in Formula \eqref{stimis} and $B=4\Gamma K_0^{-\tau_0+n+d/2}$. If for the Hamiltonian $H_0$ we can choose $\gamma$ so that if for $\Theta_0 $ defined in \eqref{parini} we have: 
\begin{equation}\label{tolkien}
\Theta_0 \leq {\art}\,, B\gamma \e^{2n-2}a_0^{-1} <|\O_0|  \end{equation}
are satisfied, then
we can construct recursively      sets $\O_j\subset \O$ and a Hamiltonian   $ H_j = \mathcal N_j + P_j : D_j\times\O_{j} \to {\mathbb C} ,\quad \mathcal N_j:= (\omega^{(j)}(\xi),y)+\sum_{k\in S^c} \Omega^{(j)}_k(\xi) |z_k|^2$ with $\Omega^{(j)} _n(\xi)=\s(n)(|\er(n)|^2+2\val_n(\xi))+ \tilde
\Omega^{(j)}_n(\xi).$ So that, if we 
define\begin{equation}\label{rosamystica}
\ix^{(h)}_j :=  \frac{\|X_{P^{(h)}_{j}}\|_j}{\gamma} \,, \vec\ix_j :=(\ix^{(0)}_j ,\ix^{(1)}_j ,\ix^{(2)}_j ),\, \quad \Theta_j=  \frac{\|X_{P_j}\|_j}{\gamma} ,
\end{equation}
the following properties are satisfied for all $j$:
\vskip15pt
$ {\bf (S1)_{j}} $  For $j>0$, $\O_j\subset \O_{j-1}$ is defined by  
\eqref{zero}-\eqref{third} with $\ome\rightsquigarrow\ome^{(j-1)}$ and $\Ome_n\rightsquigarrow \Ome^{(j-1)}_n$. We have that
$ H_j = H_{j-1} \circ \Phi_j $ where
$ \Phi_j\, : \,D_j \times \O_{j} \to D_{j-1} $
is a Lipschitz family of real analytic symplectic maps of the form 
$  \Phi_j = I + \Psi_j $ with $ \|\Psi_j\|^{\lambda_j}_{D_j}< C_\star {\bar\Theta} 2^{-j} $.
 \vskip15pt
 
$ {\bf (S2)_{j}} $ The Hamiltonian $H_j$ is  compatible with the parameters $\underline p_j$ (Definition \ref{buone}). The parameters   $r_+= r_{j+1},\, s_+= s_{j+1} $ satisfy the hypotheses \eqref{rpiu} of the KAM step
 and the set $\O_{j+1}\subset\O_{j }$ satisfies Formula \eqref{stimis}, namely, using \eqref{itau}:
$$| \O_j\setminus \O_{j+1} |\leq \Gamma  \gamma a_j^{-1}\varepsilon^{2(n-1) }  {\Pappa}_{j }^{-\tau_0+n+d/2}\implies | \O_0\setminus \O_{j+1} |\leq B  \gamma a_0^{-1}\varepsilon^{2(n-1) }  .  $$ 
\vskip15pt
$ {\bf (S3)_{j}} $ 
There exist Lipschitz extensions 
$ {\hat \o}^{(j)} $, $  {\hat \Ome}^{(j)} $ of $ \o^{(j)} $, $ \tilde \Ome^{(j)} $ defined on $ \O_0 $ and, for $ j \geq 1$:
\be\label{hatomeghini}
| {\hat \omega}^{(j)} - {\hat \omega}^{(j-1)} |+ \lambda_j | {\hat \omega}^{(j)} - {\hat \omega}^{(j-1)} |^{\rm lip}  \leq \gamma \ix^{(2)}_j \, , \
\| {\hat \Omega}^{(j)} - {\hat\Omega}^{(j-1)} \|_\infty + \lambda_j
\| {\hat \Omega}^{(j)} - {\hat \Omega}^{(j-1)} \|_\infty^{\rm lip}
\leq \gamma \ix^{(2)}_j
\end{equation}
\be\label{hatomeghini1}
\ |   {\hat\o}^{(j)} |^{\rm lip} + \|   {\hat\Ome}^{(j)} \|^{\rm lip}_\infty\leq M_j \,,\quad | {\hat \omega}^{(j)} - \vgot |<M_j\e^2.
\end{equation}
\vskip15pt
$ {\bf (S4)_{j}} $ \quad $\vec\ix_j ,\Theta_j$
satisfy  (\ref{passoj})
and \eqref{Cgotbis} hence $|\vec{\ix}_j| \leq  \, \Art |\vec{\ix}_0| e^{-\chi^j}$ .
\vskip15pt
$ {\bf (S5)_{j}} $ \quad For $j>0$
the sequence of composed maps $ \tilde\Phi_j:=\Phi_1\circ\Phi_2\circ\cdots\circ\Phi_j  = I + \tilde\Psi_j $ satisfies $  \|\tilde \Psi_{\nu+1}- \tilde \Psi_\nu\|^\lambda_{D_j}\leq C_\star|\vec\ix _0|2^{-\nu } ,\quad\| \tilde\Psi_j \|^\lambda_{D_j} \leq  2 C_\star |\vec\ix _0| $.
\end{lemma}

\begin{proof}
We proceed by induction the conditions $(\bf Si)_0$ are satisfied by the hypotheses of Theorem \ref{KAM} except that, for  $(\bf S2)_0$, once we have chosen $\Pappa$ satisfying the constraints of the previous Lemmas,  we have to impose a further smallness condition  on $\bar\Theta$ deduced by formula  \eqref{rpiu}.

Then, by induction, we prove the statements
$ \bf(Si)_{\nu+1} $, $ i = 1, \ldots , 5 $, by assuming the validity of $ \bf (Si)_{j} $ for $j\leq \nu$.
\vskip15pt

{\sc  $(\bf S1)_{\nu+1}$.} We apply the KAM step with $H= H_\nu$. By $(\bf S2)_\nu$, we have that $H_\nu$ is compatible with the parameters $p_\nu$.  
In order to implement the KAM step, and deduce $(\bf S1)_{\nu+1}$ we need to verify  the constraints of Formulas \eqref{pois1}, \eqref{nubo} and \eqref{rpiu} are satisfied for $ r_+,s_+,{\Pappa}_+,\theta_+,\mu_+ = r_{\nu+1},s_{\nu+1},{\Pappa}_{\nu+1},\theta_{\nu+1},\mu_{\nu+1} , {\Pappa}={\Pappa}_\nu,\vec\epsilon=\vec\ix_\nu  $.  Substituting in \eqref{pois1}, \eqref{nubo}  we easily see that this amounts to a lower bound on ${\Pappa}$, depending only on $\tau_0,\tau_1$ and the remaining parameters in $p_0$. This we have imposed at the beginning of the algorithm, as we explained in  Remark \ref{lapa2}.

As for \eqref{rpiu},   we have by induction the inequality on   $|\vec\ix_\nu|$ and we have to verify 
 $
2\kappa  \delta_\nu ^{-1}e \Art  \art e^{-\chi^\nu} {\Pappa_\nu}^{4d\tau_1}< \frac12\,,
$  which is contained in the constraints \eqref{alcos}.

Then  following the KAM step we construct the set $\O_+$ which coincides by definition with $\O_{\nu+1}$.  On $\O_{\nu+1}$, we define the generating function $F=F_{\nu+1}$.  Then we construct  the real analytic symplectic map
$ \Phi_{\nu+1}:D_{\nu+1}\times \O_{\nu+1} \to D_\nu $, Lipschitz in $ \O_{\nu+1} $, generated by $F$. We have:
$$
H_{\nu+1} := H_+= H^\nu\circ \Phi_{\nu+1}=:\mathcal N_{\nu+1}+P_{\nu+1} \, , \quad  \mathcal N_{\nu+1}:=\mathcal N_\nu+ [P_\nu] \, .
$$
and $P_{\nu+1}:= P_+$ defined in \eqref{Pnupiu}. \vskip15pt

{\sc  $(\bf S2)_{\nu+1}$.}  By construction  $  H_{\nu+1}$ is of the form given by Formula \eqref{buone}.   We want to apply the results of \S \ref{s:Kstep} in order to prove that, $ \forall \xi \in \O_{\nu+1} $, the Hamiltonian $ H_{\nu+1},$ is compatible with the parameters $p_{\nu+1}$.  For this it is enough to show that the constraints found on $\underline p_+$ in that section are in this case valid for the parameters $p_{\nu+1}$. First we need to verify \eqref{nubo} which is  a largeness condition on $\Pappa_0$:\begin{equation}   e^{-  s_\nu\frac{ 2^\nu{\Pappa_0} }{(\ln 4^\nu{\Pappa_0})^2}} {4^\nu{\Pappa_0}}^{{\tau_1}}<1\,,\  {  \kappa} < \mu_0 2^{ 3\nu } { {\Pappa_0}}^{2} \ln(4^\nu{\Pappa_0})^{-2}\,,\;    \kappa\mathtt C  <   \theta_0 2^{3\nu+1} {  {\Pappa_0}}^{ 4d\tau_0-4}\ln(4^\nu{\Pappa_0})^{-2}.
\end{equation}

 From Formula \eqref{alcos} we can bound uniformly $M_\nu\leq \sqrt{2} M_0,L_\nu\leq \sqrt{2}  L_0$ so that we have for all $\nu$  that $\esse=8\sqrt{n}M_0L_0>4\sqrt{n} M_\nu L_\nu$. By exploiting \eqref{emmepiu} and  $|\vec\ix_\nu|\leq \Art |\vec\ix_0| e^{-\chi^\nu}$, since $ 2 \Art , 2M_0L_0 \Art  \leq C_\star , $ we verify that $M_+ \leq M_{\nu+1}$ and $L_+\leq L_{\nu+1}$:
 $$ 
M_{+} :=  M_\nu( 1+2 |\vec\ix_\nu|)  \leq M_\nu( 1+2 \Art |\vec\ix_0| e^{-\chi^\nu})\leq M_{\nu+1}$$ $$ L_{+} := 
L_\nu( 1- M_\nu L_\nu \Art |\vec\ix_0| e^{-\chi^\nu})^{-1}\leq L_{\nu+1}.
$$
Finally in order  to prove that $\ome_{\nu+1}$ is a lipeomorphism we argue as for  \eqref{ompl}  since $ L_{\nu+1} |P^{1,0}_0|^{lip}\leq \sqrt{2}L_0 \Art |\vec\ix_0| e^{-\chi^{\n+1}u}<1$.\smallskip

 The estimate on $\O_{\nu+1}$ follows from Lemma \eqref{measure}. 
 We finally show that $a_+$,  defined in Formula \eqref{emmepiu}  is $\geq a_{\nu+1}$, by noting that 
 $$a_\nu-a_+\leq 4   {\esse}M_\nu  |\vec\ix _\nu|\stackrel {\eqref{cstar}}{\leq }  4 \sqrt 2 \esse M_0  \Art |\vec\ix _0| e^{-\chi^\nu}\leq a_0 C_\star |\vec \ix_0| 2^{-\nu-1} = a_\nu- a_{\nu+1}.$$ 
%

\vskip15pt
{\sc  $(\bf S3)_{\nu+1} $.} The frequency maps $ \om^{(\nu+1)} $, 
$ \Omega^{(\nu+1)} $ are  defined on $ \O_{\nu+1} $ and, as we have discussed in the previous item, have Lipschitz seminorm bounded by $M_{\nu+1}$.
Then we may apply Formula \eqref{opiu} to deduce the bound \eqref{hatomeghini} (recall that $\ix^{(2)}_\nu= \gamma^{-1}\| P^{(2)}_\nu\|_\nu$),
for $\ome^{(\nu+1)}$ and $\tilde\Ome^{(\nu+1)}$.
By the Kirszbraun theorem (see e.g. \cite{KP}), used component--wise,  $\ome^{(\nu+1)}$ and $\tilde\Ome^{(\nu+1)}$ can be extended to maps-- which we denote by
$ {\hat \omega}^{(\nu+1)} $, $ {\hat \Omega}^{(\nu+1)}  $-- defined on the whole $ \O_{0} $
and preserving the same  sup-norm and Lipschitz seminorms,\eqref{hatomeghini1} follows.
 Moreover this extension may be performed so that ${\hat \omega}^{(\nu+1)}= \hat \omega^{(\nu )} + {\hat P}^{(1,0)}$ where ${\hat P}^{(1,0)}$ is an extension of $P^{(1,0)}$ which preserves the $\lambda$-norm (same for $\Ome^{(\nu+1)}$); this verifies \eqref{hatomeghini}. 

 \vskip15pt
{\sc  $(\bf S4)_{\nu+1}$}  $\vec\ix_{\nu+1}= \vec\ix_+$  satisfies  (\ref{sonasega}), with $\zeta = \zeta_\nu $ and $(s-s_+){\Pappa}/2= \s_\nu {\Pappa}_\nu $. Recalling the    definition of $ {\mathfrak b } ,{\mathfrak a } ,\mathfrak c $ we have that $\vec\ix_{\nu+1},\Theta_{\nu+1}$ satisfy the inequality of \eqref{passoj}.
We are in the Hypotheses of Lemma \ref{pomodorino}, so that also the bounds \eqref{Cgotbis} holds.
\vskip15pt
{\sc  $(\bf S4)_{\nu+1}$}
follows by  (\ref{tolkien}), $\bf (S3)_{\nu}$ and Lemma \ref{pomodorino}.
\vskip15pt
{\sc  $(\bf S5)_{\nu+1}$.} Let us denote by $\mathcal H_{s,r}^{\lambda,\mathcal O}$ the normed space of functions in $\mathcal H_{s,r}$ which depend in a Lipschitz way from parameters in $\mathcal O$ with finite $\|.\|^\lambda$ norm.  The estimate of the norm of the map $\Psi_{\nu+1}:\mathcal H_{s,r}^{\lambda,\mathcal O}\to \mathcal H_{s_\nu,r_\nu}^{\lambda_\nu,\mathcal O_\nu}$ follows, using \eqref{booo}, $\|X_{F _\nu }\|_\nu  \leq     \eufm K_\nu |\vec \ix_\nu | $ with $\eufm K_\nu =5(4^\nu\Pappa)^{4d{\tau_1}} $  and hence from \eqref{Cgotbis}  \eqref{cstar} and $$\|X_{F _\nu }\|_\nu  \leq     5  (4^\nu {\Pappa})^{4d\tau_1}   \Art |\vec\ix _0|  e^{-\chi^\nu}\leq C_\star |\vec\ix _0| 2^{-\nu-2n-7}.$$ From \eqref{alcos} and \eqref{funm}
$$12\,2^{2n+6} e  \delta_\nu^{-1}\|X_{F _\nu }\|_\nu  \leq \frac{1}{2}  \implies \quad \|\Psi_{\nu+1}\|\leq2 \norma  X_{F_\nu} \norma_{s_\nu,r_\nu}^{\l_\nu} . $$ 
  Now    we can estimate 
  $$1+\tilde \Psi_\nu=\prod_{i=1}^\nu(1+\Psi_i),\quad  \|1+\tilde \Psi_\nu\|\leq \prod_{i=1}^\nu(1+\norma \Psi_i\norma ) \leq  \prod_{i=1}^\nu(1+  2^{-i })\leq 2$$\begin{equation}
\label{itps} \tilde \Psi_{\nu+1}= \tilde \Psi_\nu+(1+\tilde \Psi_\nu)\Psi_\nu\implies \|\tilde \Psi_{\nu+1}- \tilde \Psi_\nu\|\leq C_\star|\vec\ix _0| 2^{-\nu }.
\end{equation}

Notice that $\Psi_{\nu}$ also maps $\mathcal Q^T_{\underline p_{\nu-1}}$   
to $\mathcal Q^T_{\underline p_\nu}$ and we have   similar estimates using the T\"oplitz norms.
\end{proof}

\begin{corollary}\label{cor:finale}
For all $ \xi \in \O_{\infty} := \cap_{\nu \geq 0} \O_\nu $ the sequence
$ \tilde \Phi_\nu = I + \tilde\Psi_\nu $ converges uniformly on
$ D(s_0/2,r_0/2)$ to an analytic symplectic map
$ \Phi = I + \Psi $  
such that the essential part of the perturbation
$ P^\infty_{\leq 2}( \cdot , \xi ) = 0 $.
Moreover we have
$$ | \O_0\setminus \O_{\infty}|\leq cost \, \g \varepsilon^{2n-2 }. $$
\end{corollary}

\begin{proof}
The fact that the $\tilde \Psi_\nu$ give a Cauchy sequence follows from Formula \eqref{itps}, therefore the sequence $ \tilde \Phi_\nu $ converges  as a sequence of Poisson bracket preserving homomorphisms from $\mathcal H_{s,r}$  to $\mathcal H_{s_\infty,r_\infty}$ to a Poisson bracket preserving homomorphism $\tilde \Phi_\infty$. The fact that this is induced by a coordinate change  follows from the fact that we can construct the {\em local inverse}, $\lim \tilde \Theta_\nu$ where   $\tilde \Theta_\nu= \Theta_\nu\circ\Theta_{\nu-1}\circ\cdots\circ\Theta_1  $ and $ \Theta_\nu$ is the flux  at time $-1$.

Finally $ P^\infty_{\leq 2}( \cdot , \xi ) = 0 $, $ \forall \xi \in \O_\infty $,
follows by  (\ref{rosamystica}) and $ (S4)_{\nu} $.
\end{proof}

With this Corollary we finish the proof of Theorem \ref{KAM}.

We  can finally conclude that
\begin{theorem}\label{gacom}
If for the Hamiltonian $H_0$ we have on the domains  $\ro \mathfrak K_\alpha\times D(s,r)$ a uniform  estimate for the perturbation as in  \eqref{bonls}, i.e. $\|X_{P_0}\|^T_{\vec p_0}, \| X_{ \tilde\Omega^0   }\|^{T}_{\vec p} \leq C\e^{\beta}$ with $\beta>2$, then for  $\e$ sufficiently small, the conditions on $\gamma$ of the the iterative Lemma can be satisfied.  
\end{theorem}
\begin{proof}The conditions we have imposed on $\g$ are:
$\g<1, \gamma \leq 2 \varepsilon^2 M,\quad  \| X_{ \tilde\Omega^0   }\|^{T}_{\vec p} \leq \gamma$
 $B  \varepsilon^{ 2(n-1)} \g <|\O_0|a_0$,   
$\Theta_0:=  \frac{\|X_{P_0}\|_0}{\gamma} <\bar\Theta$. We have taken $ \O_0 =\ro \mathfrak K_\alpha$ (for some component of the complement of the discriminant)  hence $|\O_0|$ can be estimated by   $C_1\varepsilon^{2n}$, hence we impose on $\g$:
\begin{equation}
\label{smeg}( \bar\Theta)^{-1} C \e^{\beta -2}<   \gamma <  \min(2 M ,B^{-1}a_0C_1)      
\end{equation}  as soon as $  \e^{\beta-2}<     \bar\Theta  C^{-1}\min(2 M ,B^{-1}a_0C_1)      $.
\end{proof}

\part{The   NLS}
{\em In this final part we prove that the NLS is a compatible Hamiltonian (in suitable coordinates) according to Definition \ref{buone} and therefore we can apply to it the KAM algorithm and arrive at the conclusions of Theorem \ref{KAM}.  Most of our work will be in showing the T\"oplitz  property of the NLS.}
\section{The T\"oplitz  property of the NLS\label{tpnls}}

In fact  $A_d$ is  T\"oplitz  so  $\|X_{A_d}\|^T_{(R,K,\theta,\mu)}= \|X_{A_d}\|_{R} $ follows from Remark \ref{dupro}.

\subsection{Semi normal form}
We now analyze the Birkhoff normal form change of variables defined in \eqref{birkof} with the purpose of proving that it maintains the quasi-T\"oplitz property.  Note that the initial variables $u,\bar u \in \ell^{a,p}_{S=\emptyset}= \bar\ell^{a,p}\times \bar\ell^{a,p} $. So all the definitions of quasi-T\"oplitz functions of Part. 2  hold with $S=\emptyset$ and hence $n=0$ (i.e. there are no action-angle variables $x,y$). Moreover at this step we assume that $\er(m)=m$ for all $m$.

We fix the parameters $N_0,\mathtt c,\mathtt C$ as in \eqref{itau}.
We start by noting that for all $d>0$ 
\begin{equation}\label{albet}
A_d:=\sum_{k_i\in \Z^n:\; \sum (-1)^i k_i=0} u_{k_1}\bar u_{k_2}u_{k_3}\bar u_{k_4}\ldots u_{k_{2d-1}}\bar u_{k_{2d}}= \sum_{\alpha,\beta\in (\Z^n)^\N: \atop {|\alpha|=|\beta|=d}} \binom{d}{\alpha}\binom{d}{\beta}u^\alpha\bar u^\beta, 
\end{equation}   is quasi T\"oplitz for all allowable $(K,\theta,\mu)$. 

In fact  $A_d$ is  T\"oplitz  so  

\begin{equation}
\|X_{A_d}\|^T_{(R,K,\theta,\mu)}= \|X_{A_d}\|_{R} \leq C(d) R^{d-2}
\end{equation}
 follows from Remark \ref{dupro} and usual dimensional arguments, see \eqref{degrl}.
   Notice that now the parameters $\vec p=(R,K,\theta,\mu)$ do not involve  $\lambda, s,\mathcal O$.
\begin{proposition}\label{qtop00}    
For all choices of parameters $0<R<\epsilon_0$  and  for allowable $(K,\theta,\mu)$, the generating function $F_{Birk}$ defined in \eqref{birkof}   is quasi--T\"oplitz with $\|X_{F_{Birk}}\|^T_{(R/2,K,\theta,\mu)}\leq  \|X_{A_{2}}\|^T_{(R,K,\theta,\mu)} \leq$ const $ R^2$.  Then for all $(K',\theta',\mu')$ which respect \eqref{nubo} with  $(K,\theta,\mu)$ we have that $H \circ \Psi^{(1)}=  H_{Birk}+ P^{(4)} + P^{(6)}$ with $\|X_{P^{(i)}}\|^T_{(R/4,K,\theta,\mu)} \leq$ const $ R^{i-2}$,  $i=4,6$ and $\|X_{H_{Birk}-\mathbb K}\|^T_{(R/4,K,\theta,\mu)} \leq$ const $ R^2$ (recall that $\mathbb K$ is defined in \eqref{kappabb})  \end{proposition}
\begin{proof}
 We need to compute the projection of $F_{Birk}$ on the space  of $N,\theta,\mu,\tau$--bilinear functions, namely following formula \eqref{pri} we compute  $F^{\pm,\pm}_{m,n}$ for all
  $m,n$ such that $m\frec{N} [v_i;p_i]$ and $m,n$ have the cut $\ell$ with parameters $\theta,\mu,\tau$.
  
   By symmetry and reality we may consider just $+,+$ and $+,-$. We need to exhibit for them T\"oplitz approximations $\mathcal F^{+.+},\mathcal F^{+.-}$. 
   
   In the case $+,+$  we write $\alpha=\alpha_0+e_m+e_n,\beta=\beta_0$, in the case $+,-$  we write  $\alpha=\alpha_0+e_m,\beta=\beta_0+e_n$, by definition $\alpha_0,\beta_0$ are the exponents of the low variables, and in our case, since $|\alpha_2|+|\beta|_2=2$, the support of $\alpha_0,\beta_0$ is in $S$ and 
we have
\begin{equation}\label{piffero1}
F^{+,\pm }_{m,n}= -\ii \sum^*_{\alpha^0,\beta^0\in \N^S} c_{\alpha,\beta}\frac{u^\alpha\bar u^\beta}{\sum_{\mathtt j \in S} (\alpha^0_{\mathtt j}-\beta^0_{\mathtt j})|\mathtt j |^2+|m|^2\pm |n|^2}.  
\end{equation}
here $c_{\alpha,\beta}:=\binom{2}{\alpha}\binom{2}{\beta}$  for simplicity of notation, while 
the symbol $\sum^*$ summarizes the  conditions of \eqref{birkof}, namely:
\begin{equation}\label{lodicoio}
\sum_{\mathtt j \in S}(\alpha^0_{\mathtt j}-\beta^0_{\mathtt j})|\mathtt j |^2+|m|^2\pm |n|^2\neq 0\,,\quad \sum_{\mathtt j \in S}(\alpha^0_{\mathtt j}-\beta^0_{\mathtt j})\mathtt j +m\pm n=0\,,\quad |\alpha_0|+|\beta_0|=2
\end{equation}

In the case $+,+$ we claim that the denominator is {\em big}  so that we can choose $\mathcal F^{+,+}_{m,n}=0$.

Indeed we have $|m|^2+ |n|^2>2\mathtt c N^{\tau_1}$  while $|\sum_{\mathtt j \in S} (\alpha^0_{\mathtt j}-\beta^0_{\mathtt j})|\mathtt j |^2|< 2 \kappa^2$ where $\kappa:=\sup_{\pluto \in S}|\pluto|$. Since $ N$ is large all these denominators are bounded below by $\mathtt  c N^{\tau_1}.$ So for $\bar F^{+,+}_{m,n}:=N^{4d\tau} F^{+,+}_{m,n}$ we bound $\|X_{\bar F}\|_R\leq N^{4d\tau-\tau_1}\mathtt  c^{-1}  \|X_{A_2}\|_R\leq  \|X_{A_2}\|_R$.  

In the case $+,-$ we  notice that $n-m=\pi(\alpha^0,\beta^0):=\sum_{\mathtt j \in S} (\alpha^0_{\mathtt j}-\beta^0_{\mathtt j})\mathtt j \in B_N\cup\{0\}$. If  $m=n$  the denomiators in \eqref{piffero1} are $m,n$ independent, we can take $\mathcal F^{+,+}_{m,n}=F^{+,+}_{m,n}$.  When $m\neq n$   we write 
\begin{equation}\label{ovvio}
|m|^2-|n|^2= (m-n,m+n)=  (m-n,2m+n-m)=-2( \pi(\alpha^0,\beta^0),m)- |\pi(\alpha^0,\beta^0)|^2.
\end{equation}  We have to distinguish two types of terms in the sum, that is the ones in which $\pi(\alpha^0,\beta^0):=\sum_{\mathtt j \in S} (\alpha^0_{\mathtt j}-\beta^0_{\mathtt j})\mathtt j  \notin \langle v_i\rangle_\ell$  and the other terms.

$$ F^{+-}_{m,n}=  -\ii \sum^*_{\alpha^0,\beta^0\in \N^S\atop \pi(\alpha^0,\beta^0)\notin \langle v_i\rangle_\ell } c_{\alpha,\beta}\frac{u^\alpha\bar u^\beta}{\sum_{\mathtt j \in S} (\alpha^0_{\mathtt j}-\beta^0_{\mathtt j})|\mathtt j |^2-2( \pi(\alpha^0,\beta^0),m)- |\pi(\alpha^0,\beta^0)|^2}$$
\begin{equation}\label{nonome}
-\ii \sum^*_{\alpha^0,\beta^0\in \N^S\atop \pi(\alpha^0,\beta^0)\in \langle v_i\rangle_\ell } c_{\alpha,\beta}\frac{u^\alpha\bar u^\beta}{\sum_{\mathtt j \in S} (\alpha^0_{\mathtt j}-\beta^0_{\mathtt j})|\mathtt j |^2-2( \pi(\alpha^0,\beta^0),m)- |\pi(\alpha^0,\beta^0)|^2}  . 
\end{equation} 
In the first terms, since $m$ has a cut at $\ell$, we have, by Remark \ref{oboe},  $|(v,m)|>\theta N^{4d\tau}  $  for all $v\notin  \langle v_i\rangle_\ell$ hence the denominator is big and
we proceed as in the case of $F^{+,+}$.

In the second terms, the  right hand side of formula \eqref{ovvio}  depends  only upon $m-n=\pi(\alpha^0,\beta^0)$ and on the cut $ [v_i;p_i]_\ell$. This implies that the constraints in the sum \eqref{lodicoio} and the denominators in Formula \eqref{nonome}  depend only on   $m-n$ and on the cut $ [v_i;p_i]_\ell$ so  the second summand of formula 	\eqref{nonome}  is  in $\mathcal T_{N,\theta,\mu}$.  The bounds follow by recalling  that the denominators are non-zero  integers. Then the bounds on the transformed Hamiltonian follow from Proposition \ref{main} and by the degree considerations \eqref{smoothT2}.
\end{proof}

We fix 
\begin{equation}\label{primomu}
 \theta= \mathtt C (1-\frac 1{16})\,,\quad \theta'= \mathtt C (1-\frac 1{8})\,,\quad\mu=  \mathtt c (1+\frac 1{16})\,,\quad \mu'= \mathtt c (1+\frac 1{8})
\end{equation}
 so that \eqref{nubo} holds for all $K=K'> N_0$.
\subsection{Action angle variables}  The results we need are mostly contained in \cite{BBP1}, although there are some small notational differences and the results in that paper are stated for $\Z$ instead of $\Z^d$, but the proofs follow verbatim in our case.\medskip

We introduce action-angle variables on the  tangential sites
$ {S } := \{\pluto_1,\dots, \pluto_n \} $  via the
analytic and symplectic map
\be\label{variableAA}
\Phi_\xi (x,y,z,\bar z ) := (u,\bar u)
\end{equation}
defined by
\be\label{actionangle}
u_{\pluto_l } := \sqrt{\xi_l +y_l} \, e^{{\rm i} x_l}, \, {\bar u}_{\pluto_l } := \sqrt{\xi_l +y_l} \, e^{- {\rm i} x_l}, \,
 l =1, \dots, n \, , \ \ u_j := z_j \, ,  \  \bar u_j := \bar z_j \, ,
\, j \in \Z^d \setminus {S} \, .
\end{equation}
Let  us consider for $\ro>0$  the set $\ro \mathfrak K_\alpha$ as in Theorem \ref{xixi}

\begin{lemma} {\bf (Domains)}
Let $s,r,\e,R > 0 $ satisfy
\begin{equation}\label{condro}
2c_1 r < \varepsilon\,,\quad  \
R= C_* \varepsilon    \quad {\rm with} \quad C_* := 4 c_2\sqrt{n} \kappa^{p} e^{ ( s+ a\kappa)} \, .
\end{equation} 
Then, for all $\xi \in \ro \mathfrak K_\alpha \cup 2\ro \mathfrak K_\alpha $, the map
\be\label{inclusio}
\Phi_\xi(\,\cdot\,;\xi) : D(s,2r) \to {\mathcal D}(R/2) := B_{R/2} \times B_{R/2}  \subset \bar\ell^{a,p} \times \bar\ell^{a,p}
\end{equation}
is well defined
and analytic (recall  that $D(s,2r)$ is defined in \eqref{domain} and $\kappa:=\sup_{\pluto \in S}|\pluto|$$)$.
\end{lemma}
For the proof see \cite{BBP1} Lemma 7.5.
\smallskip

Given a function $ F : \, {\mathcal D}(R/2)\to \mathbb C $,
the previous Lemma shows that the composite map  $ F \circ \Phi_\xi :
D(s, 2r) \to \mathbb C $ is well defined and  regular. 
 The main result of  this section is Proposition
\ref{qtop0}:
if $ F $ is  quasi-T\"oplitz in the variables $ (u, \bar u) $ then
the composite $ F \circ \Phi_\xi $ is  quasi-T\"oplitz in the
variables $ (x,y,z, \bar z ) $  (see Definition \ref{topbis}).

We write
\be\label{defFab} F = \sum_{\a,\b} F_{\a,\b} \mathfrak m_{\a,\b} \, ,
\quad
\mathfrak m_{\a,\b}  :=
(u^{(1)})^{\alpha^{(1)}}(\bar u^{(1)})^{\beta^{(1)}}(u^{(2)})^{\alpha^{(2)}}(\bar u^{(2)})^{\beta^{(2)}} \, ,
\end{equation}
where
$$
u= (u^{(1)}, u^{(2)}) \, , \quad  u^{(1)}:= \{u_j\}_{j\in S}
\, , \ u^{(2)}:= \{u_j\}_{j\in \Z^d\setminus S}  \, , \quad {\rm similarly \ for}  \ \bar u \, ,
$$
\be\label{defa1b1}
 (\a^{(1)},\b^{(1)}):= \{\a_j,\b_j\}_{j\in S} \, , \ \
(\a^{(2)},\b^{(2)}):= \{\a_j,\b_j\}_{j\in S^c} \, .
\end{equation}
We define 
\begin{equation}\label{ananas}
{\mathcal H}^D_R := \Big\{  F\in \mathcal H_R\  :  \ F
= \sum_{ \a^{(2)}+\b^{(2)}  \geq D} F_{\a,\b}  u^\a\bar u^\b \Big\}\, .
\end{equation}

\begin{proposition}{\bf (Quasi--T\"oplitz)} \label{qtop0}
Let  $\vec p= (r,s,\PaPi,\theta,\mu,\lambda, \ro \mathfrak K_\alpha) $, with $ \PaPi,\theta,\mu,\mu' $ be admissible parameters and
\begin{equation}\label{fuso}
(\mu'-\mu){\PaPi}^3 > \PaPi\,,\qquad  \PaPi^{\tau_1}
2^{-\frac{\PaPi}{2\kappa}+1}<1 \, .
\end{equation}
If
$ F \in \mathcal Q^T_{R/2,\PaPi,\theta,\mu'} \cap {\mathcal H}_{R/2}^{D} $,
then $ f := F \circ \Phi_\xi\in \mathcal Q_{\vec p}^T$
and
\begin{equation}\label{topstim}
 \|X_f\|_{\vec p}^T
 \lessdot (  8r/R )^{D-2}\frac{\lambda}{\e^2}\|X_F\|_{R/2, \PaPi,\theta,\mu'}^T \, .
 \end{equation}
\end{proposition}
For the proof we will need several Lemmas. First let us deduce the main consequence of Proposition \ref{qtop0}. 
\begin{corollary}\label{NLS00}
For all $\e>0$,  $c_1\e/2> r >\e^3$and $s>0$ satisfying \eqref{condro}, the perturbation $P$ of Definition \ref{puzza} is in $\mathcal Q^T_{\vec p}$  for the parameters $\vec p= (s,r,\theta= \mathtt C (1-\frac14), \mu= \mathtt c(1+\frac14), \lambda, \ro \mathfrak K_\alpha)$.
Moreover $P$ satisfies the bounds
 \begin{equation}\label{scossa}
 \norma X_P\norma_{\vec p}^T \leq  C\frac{\lambda}{\e^2}(\e r + \e^5 r^{-1})\,, 
\end{equation}
where $C$ does not depend on $\e,r$. 
\end{corollary}
\begin{proof}
 We choose $\mu'= \mathtt c(1+\frac18)$, the perturbation $P$ has contributions from three terms:  1) The term $  P^{(4)}\circ \Phi_\xi$,  2) The term  $P^{(6)} \circ \Phi_\xi$ and finally 3) the terms of degree $>2$ in $(H_{Birk}-\mathbb K)\circ \Phi_\xi $. By proposition \ref{qtop00} with the choice of parameters \eqref{primomu} all the terms above are quasi-T\"oplitz with the parameter $\mu'$ and the Bounds of Proposition \ref{qtop00} hold. 
 
 In item 1) we note that, by definition, $P^{(4)}\in \mathcal H^3_R$ so, by Propositions \eqref{qtop0} and  \eqref{qtop00}, we have $\frac{\e^2}{\l}\norma X_{P^{(4)}}\norma_{\vec p}^T \leq
 (r/R) R^2\leq C  \e \, r $.  In item 2)  we recall that by momentum conservation the first term of $ P^{(6)}$ of degree $D=0$ ( $P^{(6)}\in \mathcal H^0_R\setminus \mathcal H^1_R$) is actually of degree at least $8$.  Then we divide  $ P^{(6)}=  R + Q$ where $Q\in \mathcal H^1_R$ and $R$ is of degree at least $8$ in $u,\bar u$.  
 By Propositions \eqref{qtop0} and  \eqref{qtop00}, we have $\norma X_{R}\norma_{\vec p}^T \leq
 (r/R)^{-2} R^6 \leq C  \e^{8} \, r^{-2} \leq C \e^5 r^{-1}$. In the same way $\norma X_{Q}\norma_{\vec p}^T \leq
 (r/R)^{-1} R^4 \leq C  \e^{5} \, r^{-1} $.
 
 Finally in 3) we collect the terms of degree  $ 3$  and $4$ in  formula \eqref{Ham2}, we get the estimates  $ \frac{\e^2}{\l} |X_{\Pi^{\geq 3}(H_{Birk}-\mathbb K)}|_{\vec p\,{}'}^T \leq  C\e r $.  
   \end{proof}
The rest of this section is devoted to the proof of Proposition  \ref{qtop0}.
Introducing the action-angle variables \eqref{actionangle}
 in \eqref{defFab}, and using the Taylor expansion
\be\label{fractional}
(1 + t )^g = \sum_{h \geq 0} \binom{g}{h} t^h \, , \quad   \binom{g}{0} := 1 \, ,
\ \ \binom{g}{h} := \frac{g (g-1) \ldots (g - h + 1)}{h!} \, , \ h \geq 1 \, ,
\end{equation}
we get
\be\label{fnewc}
 f := F\circ \Phi_\xi  = \sum_{k,i,\a^{(2)},\b^{(2)}}f_{k,i,\a^{(2)},\b^{(2)}} e^{ \ii (k
, x)} y^i z^{\a^{(2)}} \bar z^{\b^{(2)}} \end{equation}
with   Taylor--Fourier coefficients
\be\label{coeffit} f_{k,i,\a^{(2)},\b^{(2)}} :=
\sum_{\a^{(1)}-\b^{(1)}=k} F_{\a,\b}
\prod_{l=1}^n\xi_l^{\frac{\a^{(1)}_l+\b^{(1)}_l}{2} -i_l}
\binom{\frac{\a^{(1)}_l +\b^{(1)}_l}{2}}{i_l} \, .
\end{equation}

\begin{lemma}\label{artaserse}
{\bf ($ M$-regularity)} If $ F \in {\mathcal H}_{R/2}^D $ then $ f := F
\circ \Phi_\xi\in \mathcal H_{s,2r}$  and
\begin{equation}\label{coord}
\|X_f\|_{s,2r,\ro \mathfrak K_\alpha\cup {2\ro}\mathfrak K_\alpha }
\lessdot  (8r/R)^{D-2} \|X_F\|_{R/2} \,, \quad \|X_f\|^{lip}_{s,2r,\ro \mathfrak K_\alpha}
\lessdot  \e^{-2}(8r/R)^{D-2} \|X_F\|_{R/2} \, .
\end{equation}
Moreover if $F$ preserves momentum then so does $F\circ \Phi_\xi$.
\end{lemma}

\begin{proof} See \cite{BBP1} Lemma 7.7.
\end{proof}
 
\begin{definition}
For a monomial
$ \mathfrak m_{\a,\b}  :=
(u^{(1)})^{\alpha^{(1)}}(\bar u^{(1)})^{\beta^{(1)}}(u^{(2)})^{\alpha^{(2)}}(\bar u^{(2)})^{\beta^{(2)}} $
$($as in \eqref{defFab}$)$
we set
\be\label{tange}
 {\mathfrak p} ( \mathfrak m_{\a,\b} ):= \sum_{l = 1}^n
 \langle \pluto_l \rangle (\alpha^{(1)}_{\pluto_l} +  \beta^{(1)}_{\pluto_l} ) \,  , \quad \langle j \rangle :=
  \max \{1, |j| \} \, .
\end{equation}
For any  $ F $ as in \eqref{defFab}, $ K \in \N $, we define the projection
\be\label{protang}
\Pi_{\mathfrak p \geq  K} F :=
\sum_{{\mathfrak p} ( \mathfrak m_{\a,\b} ) \geq K} F_{\a,\b}  \mathfrak
m_{\a,\b} \, , \quad \Pi_{\mathfrak p <  K} := I - \Pi_{\mathfrak p \geq  K} \, .
\end{equation}
\end{definition}

\begin{lemma}\label{cometichiami}
Let $ F \in {\mathcal H}_{R/2} $. Then
\be\label{stimaPX1}
\| X_{(\Pi_{\mathfrak p \geq  K}F)\circ
\Phi_\xi}\|_{s,r,\ro \mathfrak K_\alpha}
\leq
2^{- \frac{K}{2\kappa }+1} \|X_{F\circ \Phi_\xi}\|_{s,2r,{2\ro}\mathfrak K_\alpha }\,.
\end{equation}
\end{lemma}

\begin{proof}
See \cite{BBP1} Lemma 7.8. \end{proof}

Let $ \PaPi, \theta, \mu,\mu',\tau$ be as in Proposition \ref{qtop0}.
For $ N \geq \PaPi $  and $ F \in {\mathcal H}_{R/2} $ we set
\begin{equation}\label{adamo}
 f^*:= \Pi_{N,\theta,\mu,\tau}\Big( (F-\Pi_{N,\theta,\mu',\tau }F)\circ
\Phi_\xi \Big) \, .
\end{equation}
Note that $\Pi_{N,\theta,\mu',\tau}$ is the projection on the
bilinear functions in the variables $u,\bar u$,
while $\Pi_{N,\theta,\mu,\tau}$ 
in the variables $x,y,z,\bar z $.

The next Lemma corresponds to  Lemma 7.9 of \cite{BBP1}.
\begin{lemma}\label{ciccio} We have
\begin{equation}\label{cojo2}
\| X_{ f^* }\|_{s,r,{\ro}\mathfrak K_\alpha }
 \leq
2^{- \frac{N}{2 \kappa} + 1 }  \|X_{F\circ\Phi_\xi}\|_{s,2r,{2\ro}\mathfrak K_\alpha } \, .
\end{equation}

\end{lemma}
\begin{proof}
We first claim that
if $ F =  \mathfrak m_{\a,\b} $  is  a  monomial as in \eqref{defFab} with
$ \mathfrak p( \mathfrak m_{\a,\b} ) < N $ then   
$ f^* = 0 $.

{\sc Case $ 1 $}: $  \mathfrak m_{\a,\b} $ is $( N,\theta,\mu',\tau)$--bilinear, see Definition \ref{taubilinear}. Then
$ \Pi_{N,\theta,\mu',\tau }  \mathfrak m_{\a,\b} =  \mathfrak m_{\a,\b} $ and $ f^*= 0 $, see \eqref{adamo}.

{\sc Case $ 2 $}:  $  \mathfrak m_{\a,\b} $ is {\em not }$(N,\theta,\mu',\tau)$--bilinear.
Then we have $\Pi_{N,\theta,\mu',\tau }  \mathfrak m_{\a,\b} = 0 $ and hence $ f^*=\Pi_{N,\theta,\mu,\tau}(  \mathfrak m_{\a,\b} \circ \Phi_\xi) $,
see \eqref{adamo}.
We claim that $ \mathfrak m_{\a,\b} \circ \Phi_\xi $
is not $(N,\theta,\mu,\tau)$--bilinear, and so $  f^* = \Pi_{N,\theta,\mu,\tau}(  \mathfrak m_{\a,\b} \circ \Phi_\xi) = 0 $.
Indeed,
\be\label{mabco}
\mathfrak m_{\a,\b} \circ \Phi_\xi = (\xi + y)^{\frac{\alpha^{(1)} + \beta^{(1)}}{2}}
 e^{\ii ((\alpha^{(1)}-\beta^{(1)}), x)} z^{\alpha^{(2)}} \bar
z^{\beta^{(2)}}
\end{equation}
is  $(N,\theta,\mu,\tau)$--bilinear if and only if  (see  Definitions \ref{taubilinear} and \ref{LM})
$$
z^{\alpha^{(2)}} {\bar z}^{\beta^{(2)}}=
z^{\tilde \alpha^{(2)}} {\bar z }^{\tilde \beta^{(2)}} z_m^\sigma
z_n^{\sigma '} \, ,
$$
\begin{equation}\label{rompic}
\sum_{j\in \Z^d\setminus S} |j| (\tilde \alpha^{(2)}_j+ \tilde \beta^{(2)}_j)<
\mu N^3 \,,\quad  |m|,|n|> \theta N^{\tau_1}\,,
\quad
|\a^{(1)}-\b^{(1)}|< N\,,
\end{equation}
and $n,m$ have a cut at $\ell$ with parameters $\theta,\mu,\tau$.

We deduce the contradiction that  $ \mathfrak m_{\a,\b} =
(u^{(1)})^{\alpha^{(1)}}(\bar u^{(1)})^{\beta^{(1)}}
(u^{(2)})^{\tilde \alpha^{(2)}}(\bar u^{(2)})^{\tilde \beta^{(2)}} u_m^\sigma u_n^{\sigma '} $
is  $(N,\theta,\mu',\tau)$-bilinear  
because (recall that we suppose $ \mathfrak p(  \mathfrak m_{\a,\b} ) < N $)
$$
\sum_{l =1}^n |\pluto_l |(\alpha^{(1)}_{\pluto_l}
+ \beta^{(1)}_{\pluto_l})+ \sum_{j\in \Z^d\setminus
S} |j| (\tilde \alpha^{(2)}_j +  \tilde \beta^{(2)}_j)\!\! \stackrel{\eqref{tange}, \eqref{rompic}}{<} \!\!
\mathfrak p(  \mathfrak m_{\a,\b} )+ \mu N^3 < N + \mu N^3
\stackrel{\eqref{fuso}} < \mu' N^3 \, .
$$
For the general case, we divide $ F= \Pi_{\mathfrak p <  N} F + \Pi_{\mathfrak p \geq   N}F
$. By the above claim
$$
 f^*=\Pi_{N,\theta,\mu,\tau}\Big(
\big((Id-\Pi_{N,\theta,\mu',\tau})\Pi_{\mathfrak p \geq   N} F\big)\circ
\Phi_\xi\Big)=\Pi_{N,\theta,\mu,\tau}\Big(
\big(\Pi_{\mathfrak p \geq   N}(Id-\Pi_{N,\theta,\mu',\tau}) F\big)\circ
\Phi_\xi\Big) \, .
$$
Finally, \eqref{cojo2} follows by  applying Lemma \ref{cometichiami} to $ \big(\Pi_{\mathfrak p \geq   N}(Id-\Pi_{N,\theta,\mu'}) F\big)\circ
\Phi_\xi $ and using the fact that projections may only reduce the norm.
\end{proof}
\begin{lemma}\label{top:aa}
Let 
 $ F\in {\mathcal T}_{R/2,N,\theta,\mu',\tau}$ with $\Pi_{\mathfrak p \geq   N}F=0 $. Then $ F \circ \Phi_\xi(\cdot;\xi)\in {\mathcal T}_{s,
2r,N,\theta,\mu',\tau}$, $\forall\, \xi \in {\rro}\mathfrak K_\alpha\cup  {2\rro}\mathfrak K_\alpha$ .
\end{lemma}
\begin{proof}
Recalling Definition \ref{topa} we have $$
F = \sum_{A\in \mathcal H_N} \sum_{|m|,|n|> \theta N^{\tau_1}, \s,\s'=\pm}^*
F^{\s,\s'}(A,\s m+\s' n) u_m^\s u_n^{\s'} \ \ {\rm with} \ \ F^{\s,\s'}(A, h)
\in {\mathcal L}_{R/2}(N,\mu', h ) \, .
$$
denoting $A=[v_i;p_i]_\ell$ the apex $*$ means the sum restricted to those $n,m$ which have a cut  at $\ell$ with parameters $(\theta,\mu,\tau)$ and $m$ has associated space $A$.
 
Composing with the map $ \Phi_\xi $ in \eqref{actionangle}, since $ m, n \notin {S } $,  we get
$$
F\circ \Phi_\xi = \sum_{\s,\s'=\pm\,; |m|,|n|> \theta N^{\tau_1}}
F^{\s,\s'}(A,\s m+\s' n)\circ \Phi_\xi \, z_m^\s z_n^{\s'} \, .
$$
Each coefficient
$ F^{\s,\s'}(A,\s m+\s' n)\circ \Phi_\xi$  depends on $n,m,\s,\s'$ only
through $A,\s m+\s' n,\s,\s' $. Hence, in order to conclude that  $ F\circ \Phi_\xi
 \in {\mathcal T}_{s, 2r,N,\theta,\mu',\tau}
$
 it remains only to prove
 that  $ F^{\s,\s'}(A,\s m+\s' n) \circ\Phi_\xi\in {\mathcal L}_{s,2r}(N,\mu',\s m+\s' n) $,
 see Definition \ref{LM}.
Each monomial $ \mathfrak m_{\a,\b} $ of $F^{\s,\s'}( A , \s m+\s' n ) \in {\mathcal L}_{R/2}(N,\mu',\s m+\s' n) $
satisfies
$$
\sum_{l=1}^n (\a_{\pluto_l}+\b_{\pluto_l} )|\pluto_l |
+ \sum_{j\in \Z^d\setminus S}(\a_j+\b_j)|j| < \mu' N^3  \ \quad
{\rm and } \quad \ \mathfrak p(\mathfrak m_{\a,\b} ) < N
$$
by the hypothesis $ \Pi_{\mathfrak p \geq   N}F =0 $. Hence 
 $ \mathfrak m_{\a,\b} \circ \Phi_\xi$ (see \eqref{mabco}) is $(N,\mu')$-low momentum,
 in particular $  | \a^{(1)} - \b^{(1)} | \leq
\mathfrak p(\mathfrak m_{\a,\b} ) < N $.
\end{proof}

\smallskip

\begin{proof}{\sc of Proposition \ref{qtop0}}.
Since $  F \in {\mathcal Q}^T_{R/2,\PaPi,\theta,\mu' } $ (see Definition \ref{topbis}), for all $ N\geq \PaPi $,
there is a T\"oplitz approximation  $ \tilde F \in {\mathcal
T}_{R/2,N,\theta,\mu',\tau }$  of $ F $, namely
\be\label{Pmu1}
\Pi_{N, \theta, \mu',\tau} F= \tilde F+N^{-4d\tau}\hat F \,\quad{\rm  with }
\quad   \|X_F \|_{R/2}, \|X_{\tilde F}\|_{R/2},\|X_{\hat F}\|_{R/2} <
2\|F\|_{R/2, \PaPi ,\theta,\mu'}^T \, .
\end{equation}
In order to prove that $ f := F \circ \Phi_\xi \in  {\mathcal Q}_{s,r, \PaPi,\theta,\mu }^T  $
  we define
its candidate T\"oplitz approximation
\be\label{tildefap}
\tilde f:=
\Pi_{N,\theta,\mu,\tau}( (\Pi_{ {\mathfrak p} < N} \tilde F)
\circ \Phi_\xi) \, ,
\end{equation}
see \eqref{protang}.
Lemma \ref{top:aa} applied to
$  \Pi_{ {\mathfrak p} < N} \tilde F \in  {\mathcal T}_{R/2,N,\theta,\mu' ,\tau}  $  
implies that
$ (\Pi_{ {\mathfrak p} < N} \tilde F) \circ \Phi_\xi  \in {\mathcal T}_{s,2r,N,\theta,\mu',\tau} $ and
then, applying the projection $ \Pi_{N,\theta, \mu,\tau}$ we get
$ \tilde f \in {\mathcal T}_{s,2r,N,\theta,\mu,\tau } \subset  {\mathcal T}_{s,r,N,\theta,\mu ,\tau} $.
Moreover,  by \eqref{tildefap} and applying Lemma \ref{artaserse} to
$ \Pi_{ {\mathfrak p} < N} \tilde F $ (note that $  \Pi_{ {\mathfrak p} < N} \tilde F $ is either zero or it
is in $  {\mathcal H}^{D}_{R/2} $ with $ D \geq 2 $ because it is bilinear),
we get  
\begin{eqnarray}\label{hofame}
\|X_{\tilde f}\|_{s,r,\ro \mathfrak K_\alpha} \leq  \|X_{ (\Pi_{ {\mathfrak p} < N} \tilde F)
\circ \Phi_\xi) }\|_{s,r,\ro \mathfrak K_\alpha}
& \stackrel{\eqref{coord}} \lessdot & (8r/R)^{D-2}
\|X_{\Pi_{ {\mathfrak p} < N} \tilde F} \|_{R/2} \nonumber \\
&  \stackrel{  \eqref{Pmu1}} \lessdot  & (8r/R)^{D-2} \| F \|_{R/2, \PaPi, \theta, \mu',\tau }^T \, .
\end{eqnarray}
Moreover the T\"oplitz defect is
\begin{eqnarray*}
\hat f
&:=&
N^{4d\tau}(\Pi_{N,\theta,\mu,\tau} f-\tilde f)
\stackrel{\eqref{tildefap}} = N^{4d\tau} \, \Pi_{ N,\theta,\mu,\tau }\big(( F - \Pi_{{\mathfrak p}< N} \tilde F)\circ \Phi_\xi\big)
\\
&=&
 N^{4d\tau}\Pi_{N,\theta,\mu,\tau}\big((F-\tilde F)\circ \Phi_\xi\big) + N^{4d\tau}\Pi_{ N,\theta,\mu,\tau }\big(( \tilde F - \Pi_{{\mathfrak p} < N} \tilde F)\circ \Phi_\xi\big)
\\
& \stackrel{\eqref{Pmu1}, \eqref{protang}} =&
\Pi_{ N,\theta,\mu,\tau }(\hat F\circ \Phi_\xi)
+ N^{4d\tau}\Pi_{N,\theta,\mu,\tau} \Big( \big(F-\Pi_{N,\theta,\mu'}F\big)\circ \Phi_\xi\Big)\\&&\qquad
+ N^{4d\tau}\Pi_{ N,\theta,\mu,\tau
} \big((\Pi_{{\mathfrak p} \geq N} \tilde F)\circ \Phi_\xi\big)
\\
&\stackrel{\eqref{adamo}}=&
\Pi_{ N,\theta,\mu ,\tau}(\hat F\circ \Phi_\xi)
+ N^{4d\tau}f^*
+ N^{4d\tau}\Pi_{ N,\theta,\mu,\tau
} \big((\Pi_{{\mathfrak p} \geq N} \tilde F)\circ \Phi_\xi\big)
\,.
\end{eqnarray*}
 Lemmata \ref{cometichiami} and
\ref{ciccio}
 imply that, since $N^{4d\tau} 2^{-\frac{N}{2\kappa}+1}\leq 1$ 
  $\forall\, N\geq \PaPi$  by \eqref{fuso},
\begin{eqnarray}
\|X_{\hat f}\|_{s,r,{\ro}\mathfrak K_\alpha }
&\leq&
\|X_{\hat F\circ \Phi_\xi}\|_{s,r,{\ro}\mathfrak K_\alpha } + N^{4d\tau}
2^{-\frac{N}{2\kappa}+1}
(\|X_{F\circ \Phi_\xi}\|_{s,2r,{2\ro}\mathfrak K_\alpha }
+\|X_{ \tilde F\circ \Phi_\xi}\|_{s,2r,{2\ro}\mathfrak K_\alpha })
\nonumber
 \\
 &\lessdot &
  \|X_{\hat F\circ \Phi_\xi}\|_{s,2r,{\ro}\mathfrak K_\alpha }+
 \|X_{ F\circ \Phi_\xi}\|_{s,2r,{2\ro}\mathfrak K_\alpha }
 +\|X_{ \tilde F\circ \Phi_\xi}\|_{s,2r,{2\ro}\mathfrak K_\alpha } \nonumber
  \\
&  \stackrel{\eqref{coord}}{\lessdot} &
(8r/R)^{D-2}
(\|X_{\hat F}\|_{R/2}+\|X_{F}\|_{R/2}+\|X_{ \tilde F}\|_{R/2})
\label{ultimopassa}
  \\
& \stackrel{\eqref{Pmu1}} \lessdot & (8r/R)^{D-2} \|F\|_{R/2,\PaPi,\theta,\mu'}^T
 \label{hofamebis}
 \end{eqnarray}
(to get  \eqref{ultimopassa}  we also note that we can choose $ \hat F , \tilde F$ so that they belong to the same $ {\mathcal H}^{D}_{R/2} $ as $F$.
The bound \eqref{topstim} follows by \eqref{coord}, \eqref{hofame}, \eqref{hofamebis}.
\end{proof}

\subsection{Reduction to constant coefficients}

 For all  $k\in S^c$ set $\er(k):=\er(A)$ to be the root of the component $A$  of $\Gamma_S$ to which $k$ belongs (this is chosen once in one of the graphs in the same translation class).  We have thus associated to each $k$  an element $L(k)\in \Z^n$ , see Theorem \ref{teo1} and formula \eqref{defL}:
\begin{equation}\label{labella}
  z_k= e^{-\ii L(k).x}z_k' ,\ y=y'+\sum_{k\in S^c}  L(k)  |z_k'|^2,\ x=x'.
\end{equation} to define  a  symplectic change  of variables $\Psi:D(s,r/2)\to D(s,r)$ in which the normal form has constant coefficients. One may trivially check  that 
$$
\|X_{F\circ \Psi}\|^\lambda_{s,r}\le 4 e^{2d\kappa s}\|X_{F}\|^\lambda_{s,2r}.
$$

  We need to see what happens to   $(N,\theta,\mu,\tau)$--bilinear monomials first.
Note that  the momentum of $z_k'$  is $\er(k)$.
Take     a monomial
 $$ \mathfrak m= \mathfrak m_{\alpha,\beta,k}= e^{\ii (k,x)}y^lz^\alpha{\bar z}^\beta .$$ 

  We have $$ \mathfrak m \circ\Psi=  e^{\ii (k',x)}(y'+\sum_{k\in S^c}  L(k)  |z_k'|^2)^l{z'}^\alpha \bar {z }'{}^\beta \,,\quad k'=k-\sum_jL(j)(\alpha_j-\beta_j).$$ 
 
 Hence we obtain a sum of monomials 
 \begin{equation}
\label{monok}
 e^{\ii (k',x)} (y')^h (z')^{\alpha}(\bar z')^{\beta}|z'|^{2  g} \,,  \quad h\leq l\,,\quad 
 |h|+|  g|=|l|
\end{equation} all with momentum:
\begin{equation}
\label{mome2}\pi_\er(k',\alpha,\beta)=\pi(k') +\sum_j(\alpha_j-\beta_j)\er(j)=\pi (k',\alpha,\beta) +\sum_j(\alpha_j-\beta_j)(\er(j)-j). 
\end{equation}
As in the previous section we define a cut off parameter $$\mathfrak p(\mathfrak m_{\alpha,\beta,k}):= |k|+ 2d\kappa (|\alpha|+|\beta|), $$
and set
\be\label{protang1}
\Pi_{\mathfrak p \geq  K} F :=
\sum_{{\mathfrak p} ( \mathfrak m_{\a,\b,k} ) \geq K} F_{\a,\b,k}  \mathfrak
m_{\a,\b,k} \, , \quad \Pi_{\mathfrak p <  K} := I - \Pi_{\mathfrak p \geq  K} \, .
\end{equation}
In the following Lemmas we assume that  $s> (2d\kappa)^{-1}$.
\begin{lemma}
For all $ F \in {\mathcal H}_{s,r} $ we have
\be\label{stimaPX}
\| X_{(\Pi_{\mathfrak p \geq  K}F)\circ
\Psi}\|_{s,r}
\leq
2^{- \frac{K}{8d\kappa }+2} \|X_{F\circ \Psi}\|_{2s,2r}\,.
\end{equation}
\end{lemma}
\begin{proof}
When $\mathfrak p(\mathfrak m_{\alpha,\beta,k}):= |k|+ 2d\kappa (|\alpha|+|\beta|)>K$ we distinguish two cases:

1. $2d\kappa (|\alpha|+|\beta|)>K/4$.  We note that $\Psi$ may only increase the degree in the normal variables of monomials so the total degree in the new variables is $> K/(8d\kappa)$ and the bound follows by the degree bounds \ref{degrl}.

2. Otherwise $|k'|> K/2$ and the bound follows by the ultraviolet bounds \ref{smoothl}.
\end{proof}
Fix parameters $$ (\mu'-\mu) N^3 > N \,, (\theta-\theta')N^{\tau_1}>2d\kappa .$$

  \begin{lemma}
Take a function $F\in \mathcal H_{s,r}$, assume that 
$$\Pi_{\mathfrak p\ge N} F= \Pi_{N,\mu'}^L F=0. $$ 
Then we have
$$f^*:= \Pi_{N,\theta,\mu,\tau}(F-\Pi_{N,\theta',\mu',\tau} F)\circ \Psi =0 $$
\end{lemma}
\begin{proof}
We may assume that $F=\mathfrak m_{\alpha,\beta,k}$ is a monomial. If $F$ is $(N,\theta',\mu',\tau)$--bilinear the statement is clear. Otherwise $f^*$ is a sum a monomials described by formula \ref{monok}.   If one of these monomials is bilinear its high variables either  come from one of the new exponents $  g$  or already appear in $\alpha,\beta$. In the first case this is possible only if  $\mathfrak m$ is $(N,\mu')$--low contrary to our hypothesis. In fact suppose that $  g=\bar g+e_m$, where $m=n$ is the high variable with $|\er(m)|>\theta N^{\tau_1}$, and that
$$ \sum_j |\mathtt \er(j)|(\alpha_j+\beta_j+\bar   g_j)< \mu N^3. $$
Then since $\bar   g_j\geq 0$ and $|j-\mathtt \er(j)|<2d\kappa$ we have 
$$\sum_j|j|(\alpha_j+\beta_j)< \mu N^3 +2d\kappa (|\alpha|+|\beta|) < \mu N^3+N< \mu' N^3.$$
Finally since $\mathfrak p<N$ we have  $|k|<N$, we deduce that $\mathfrak m$ is low.

 In the other case  the two high variables $m,n$ such that  $|\er(m)|,|\er(n)|>\theta N^{\tau_1}$ already appear in $\mathfrak m_{\alpha,\beta,k}$. We claim that this implies 
$\mathfrak m_{\alpha,\beta,k}$ $(N,\theta',\mu',\tau)$-bilinear contrary to the hypothesis. In fact write $\alpha=\bar\alpha+e_m,\beta=\bar\beta+e_n$. Applying $\Psi$  the monomials  appearing in $f^*$ are of the form $ \mathfrak m_{\bar\alpha+  g,\bar\beta+  g,k'}z_m\bar z_n$ with $|k'|<N$ and $\sum_j |\mathtt \er(j)|(\bar\alpha_j+\bar\beta_j+    2g_j)< \mu N^3. $ Then  $\sum_j |  j |(\bar\alpha_j+\bar\beta_j )< \mu N^3+2d\kappa (|\alpha|+|\beta|)< \mu' N^3. $ Since $|j-\mathtt \er(j)|<2d\kappa$ we have  
$$|m|,|n|> \theta N^{\tau_1}-2d\kappa >\theta'N^{\tau_1}.$$
The fact that $m,n$ have the correct cut is trivial, see Remark \ref{varlm}.
Finally we are assuming that $\mathfrak p(m_{\alpha,\beta,k}):= |k|+ 2d \kappa(|\alpha|+|\beta|)\leq N $ hence $|k|<N$ and we have that $\mathfrak m_{\alpha,\beta,k}$ is $(N,\theta',\mu',\tau)$-bilinear.

\end{proof}
 We next analyze a function $F$ with $ \Pi_{N,\mu'}^L F=F$ and again we may assume that it is a monomial $F=\mathfrak m_{\alpha,\beta,k}$, in this case $f^*:= \Pi_{N,\theta,\mu,\tau}(F-\Pi_{N,\theta',\mu',\tau} F)\circ \Psi =  \Pi_{N,\theta,\mu,\tau} F \circ \Psi   $ is a sum of monomials $ \mathfrak m_{\alpha+\bar  g,\beta+\bar  g,k'}|z_m|^2$  arising from the terms  \ref{monok} with   $g=\bar  g+e_m$.
 \begin{lemma}\label{bir}
Given a function $F$ with $ \Pi_{N,\mu'}^L F=F$ then $f^*  =  \Pi_{N,\theta,\mu,\tau} F \circ \Psi   $ is piecewise T\"oplitz and diagonal.
\end{lemma}
\begin{proof}
By the previous remarks we may compute explicitly $f^*$  as:
$$\Pi_{N,\mu}^L\Big(\nabla_y F\circ \Psi\Big)\cdot\sum_{|m|>\theta N^{\tau_1}, \atop m\in (N,\theta,\mu,\tau)-{\rm cut}} L(m)|z_m|^2\,, $$
we have that $f^*\in \mathcal T_{(N,\theta,\mu,\tau)}$  since $L(m)$ is fixed on all the $(N,\theta,\mu,\tau)$--good points of any subspace (by Theorem \ref{Lostra}). 
\end{proof}\
 \begin{lemma}
Given a function $F\in \mathcal T_{(N,\theta',\mu',\tau)}$, then $\Pi_{N,\theta,\mu,\tau} F \circ \Psi \in \mathcal T_{(N,\theta,\mu,\tau)} $
\end{lemma}
 \begin{proof}
 Recalling Definition \ref{topa} we have
$$
F = \sum_{A\in \mathcal H_N} \sum_{|m|,|n|> \theta' N^{\tau_1}, \s,\s'=\pm}^*
F^{\s,\s'}(A,\s m+\s' n) z_m^\s z_n^{\s'} \ \ {\rm with} \ \ F^{\s,\s'}(A, h)
\in {\mathcal L}_{r,s}(N,\mu', h ) \, .
$$
denoting $A=[v_i;p_i]_\ell$ the apex $*$ means the sum restricted to those $n,m$ which have a cut  at $\ell$ with parameters $(N,\theta',\mu',\tau)$ and $m$ has associated space $A$.
 
Composing with the map $ \Psi $ in \eqref{labella}, since $ m, n \notin {S } $ and $|\er(m)|,|\er(n)|> \theta N^{\tau_1}$ implies $|m|,|n|> \theta' N^{\tau_1}$,  we get $\Pi_{(N,\theta,\mu,\tau)}F\circ \Psi =$
$$
 \sum_{A\in \mathcal H_N} \sum_{|\er(m)|,|\er(n)|> \theta N^{\tau_1}, \s,\s'=\pm}^*
\Pi_{N,\mu}^L\Big(F^{\s,\s'}(A,\s m+\s' n)\circ \Psi \, e^{-\ii(\s L(m)+\s' L(n), x)}\Big)(z'_m)^\s (z'_n)^{\s'} \, .
$$
Each coefficient
$ F^{\s,\s'}(A,\s m+\s' n)\circ \Phi$   depends on $n,m,\s,\s'$ only
through $A,\s m+\s' n,\s,\s' $, same for $\s L(m)+\s' L(n)$. 
\end{proof}

 \begin{proposition}{\bf (Quasi--T\"oplitz)} \label{qtop0bis}
Let $$\vec p= (r,s,\PaPi,\theta,\mu,\lambda, \ro \mathfrak K_\alpha) ,\ {\vec p}\,'= (2r,2s,\PaPi,\theta',\mu',\lambda, \ro \mathfrak K_\alpha) $$ be admissible parameters and
\begin{equation}\label{fusobis}
(\mu'-\mu){\PaPi}^3 > \PaPi\,,\qquad (\theta-\theta')\PaPi^{\tau_1}> 2d\kappa>s^{-1}\,,  \qquad\PaPi^{\tau_1}
2^{-\frac{\PaPi}{8\kappa d}+2}<1 \,  .
\end{equation}
If
$ F \in \mathcal Q_{\vec p'}^T ,$
then $ f := F \circ \Psi\in \mathcal Q_{\vec p}^T$
and
\begin{equation}\label{topstimbis}
 \|X_f\|_{\vec p}^T
 \lessdot e^{2d\kappa s}\|X_F\|_{\vec p'}^T \, .
 \end{equation}
\end{proposition}
 \begin{proof}
Consider $N\geq \PaPi$ and suppose that $F$ has no $N,\mu'$--low terms. In this case the proof is identical to that of Proposition \ref{qtop0} provided we use the corresponding Lemmata of this section. We conclude the proof by noting that $\Pi_{(N,\theta,\mu,\tau)}\big(\Pi^L_{N,\mu'} F\circ \Psi)\in \mathcal T_{N,\theta,\mu,\tau}$ by Lemma \ref{bir}. Hence in this case the T\"oplitz defect is zero.
\end{proof}

\subsection{The final step}  In the final step we diagonalize block by block the matrices following Theorem \ref{xixi}.
The linear change $\Xi$ has a {\em finite block structure}  in the sense that the Hilbert space  $\ell^{a,p}$  is decomposed into an orthogonal sum of subspaces  $V_{\GA,a}$ indexed by the combinatorial pairs $\GA, a$  and, if we write a vector as a finite vector with coordinates in these subspaces, the linear transformation $\Xi$ is given by the finite matrix $U_\GA=(U_{i,j}(\xi))$ with entries  depending on $\xi$ and uniformly bounded by some value $U$, see Remark\ref{normU}.  Denote by $\Xi^*:F\mapsto F\circ \Xi$  the map induced on functions.  One may trivially see that the majorant norm
$$\sup_{\xi\in \ro \mathfrak K_\alpha}\|M \Xi  z'\|_{s,r} ^2 \le  \sum_{i\in S^c} (\sum_{j:\,\er(j)=\er(i)} \sup_{\xi\in \ro \mathfrak K_\alpha} |U_{i,j}(\xi)| |z'_j| )^2 e^{2a |i|} |i|^{2p} \le 2^{2p}(d+1)^2 U^2 e^{4d\kappa|a|}|z'|^2_{s,r}\,. $$ We now restrict to thet domain $\mathcal O_0= \ro \mathfrak K_\alpha$ of measure of order $\e^{2n}$,  which is all contained in one of the connected components of Theorem \ref{xixi}. Recall that   one of the domains   $ \mathfrak K_e$   is contained in the elliptic region.

Using the bounds of  Remark \ref{normU} and passing to the majorant norm for vector fields we have
\begin{equation}
\label{ilrrr}\| X_{\Xi^*f  }\|^\lambda_{s,r}\le  \eufm A \| X_{f}\|^\lambda_{s,C r} \,,\quad \eufm A:= (d+1)2^p U \,e^{2d\kappa|a|}\,. 
\end{equation}

Next we need   to control the T\"oplitz norms. We remark that, since we are making linear transformations among variables $z_k$  which have the same root,  any monomial in these variables is replaced by a  homogeneous sum of monomials in the new variables, all of which have the same  root-momentum, so the space 
$\mathcal  L_{s,r}(N,\mu,h)$ is mapped into itself. The bilinear functions  $\mathcal B_{\underline p}$, with   $\underline p:=(N,\theta,\mu,\tau)$ are mapped in $\mathcal B_{\underline p^1}$, with  $\underline p^1:=(N,\theta^1,\mu^1,\tau)$, provided that  $\theta,\theta^1,\mu,\mu^1$ are parameters which satisfy the neighborhood lemma \ref{mah}, so that if $m$ has a $\underline p$--cut, also $m+u$  has a $\underline p^1$--cut for all possible types  $u\in \mathcal Z$.  The new estimate on parameters that we need is, using Formula \eqref{vicin} for $r-m\in\mathcal Z$ is:  
 \begin{equation}
\label{vicinz}2d\kappa<\min({\kappa}   ^{-1}(\mu^1-\mu)N^{\tau-1},\ {\kappa}   ^{-1}(\theta -\theta^1)N^{4d\tau-1} ).
\end{equation}\smallskip

We now claim that for more restricted parameters $p'$ we have  that $\Xi^*\mathcal Q^T_{\underline p}$ is contained in $\mathcal Q^T_{\underline p'}$. it
 remains to understand what happens to the space  $ {\mathcal  T}_{\underline p}$ we claim that $\Pi_{\underline p'}\Xi^* {\mathcal  T}_{\underline p}\subset   {\mathcal  T}_{\underline p'}$. 
 Take thus a function $
g  =\sum_{m,n}^{(A,\underline p)} {\mathtt g}(\s m+\s' n) z_m^\s z_n^{\s'}$ as in Formula \eqref{nzo}.
We have that $A^g_{\underline p}$  is contained in some stratum $\Sigma_{\GA,a}$  for some combinatorial pair $\GA,a$.  For the space $B$ associated to $n$ we have that $B^g_{\underline p}$  is contained  in a  stratum $\Sigma_{\mathcal B,b}$. 
Note that the pair $\mathcal B,b$ is determined by $\GA,a$ and $\s m+\s' n$.

 Now the change of variables $\Xi$ acts on $z_m$ giving a linear combination of $z_ {m-u_a+u}
$ where $u_a$ is the type of $m$ and $u$ runs over the types  appearing in $\GA$ similarly for $\mathcal B$.

Consider \begin{equation}
\label{mischia}\Pi_{\underline p'}\Xi^* g=\sum_{m,n}^{(A,\underline p)}\Pi_{\underline p'}\Xi^* {\mathtt g}(\s m+\s' n) \!\!\!  \sum_{v\in \GA, \;k_1= u_v-u_a} \!\!\!  U_\GA(\xi)_{a,v} z_{m+k_1}^\s \!\!\!  \sum_{w\in \mathcal B\,,\; k_2= u_w-u_b} \!\!\!  U_\mathcal B(\xi)_{b,w}  z_{n+k_2}^{\s'}.
\end{equation} We have already remarked that $\Xi^* {\mathtt g}(\s m+\s' n)\in \mathcal  L_{s,r}(N,\mu,h)$.
Formula \eqref{mischia}  gives a sum  $\sum_{m',n'}g_{m' ,n'}^{\s,\s'}z_{m'}^\s z_{n'}^{\s'}$  where both $m',n'$ have a $\underline p'$  cut and either the associated space of $m'$ precedes that of $n'$ or the opposite case occurs, moreover $m'\in \Sigma_ {\GA, v},\ n'\in\Sigma_{\mathcal B,w}$.   Reordering Formula \eqref{mischia} it is easily seen that the coefficient $$g_{m',n'}^{\s,\s'}= U_\GA(\xi)_{a,v}  U_\mathcal B(\xi)_{b,w} {\mathtt g}(\s (m'-k_1)+\s' (n'-k_2))=$$
$$U_\GA(\xi)_{a,v}  U_\mathcal B(\xi)_{b,w} {\mathtt g}(\s  m' +\s'  n'- \s  k_1-\s'  k_2 )$$
 depends only upon $\s m'+\s' n'$ and $v$ hence the claim. Indeed the only thing to make explicit is  how to remove the restriction $m'-k_1\in A^g_{\underline p}$. This follows from the estimate on the parameters  $\underline p'$ which ensures that, if $m',n'$ have a  $\underline p'$ cut at $\ell$ then the vectors  $m'-k_1, m'-k_2$  have a  $\underline p $ cut at $\ell$. This we do as usual by the neighborhood Lemma noticing that $|k_1|,|k_2|\leq  2d\kappa$ since they are differences of two elements in $\mathcal Z$ (cf. Remark \ref{types}). So the requirement is by \eqref{vicin}:
$$2d\kappa<\min({\kappa}   ^{-1}(\mu-\mu')N^{\tau-1},\ {\kappa}   ^{-1}(\theta'-\theta)N^{4d\tau-1} .$$  \smallskip

Summarizing we have performed 4 changes of coordinates   called $\Psi^{(1)},\Phi_\xi,$$\Psi, \Xi$. The final Hamiltonian is thus  $H\circ  \Psi^{(1)}\circ \Phi_\xi\circ \Psi\circ \Xi$, this by definition is    {\em The Hamiltonian of the NLS in the final coordinates}.   Recall that the perturbation $P$ refers to the Hamiltonian  $H\circ \Psi^{(1)}\circ \Phi_\xi=\mathcal N+P   $ (Definition \ref{puzza})  which by abuse of notation we have still called $H$.

  \begin{proposition}\label{Pfinale}
 The perturbation of the Hamiltonian of the NLS in the final coordinates is quasi-T\"oplitz for the parameters $\vec p_0=(r_0= r/(2\eufm A), s_0= s/4,\theta_0=\mathtt C/2, \mu_0= 2 \mathtt c, K_0> N_0, \lambda=2\e^2, \mathcal O= \ro \mathfrak K_\alpha)$. We have the bounds:
 \begin{equation}\label{stimefinali}
 \norma X_{  P   \circ\Psi\circ \Xi}\norma_{\vec p_0}^T \leq  C(\e r + \e^5 r^{-1})\,, 
\end{equation}
\end{proposition}
\begin{proof}
By  Corollary \ref{NLS00}  we have that $P$ is quasi-T\"oplitz with parameters $s,r, K,\theta= \mathtt C \frac34,\mu= \mathtt c \frac54,2\e^2, \ro \mathfrak K_\alpha) $. 
 Since, by Formula \eqref{itau}: 
  $$ (\frac{3}{2} \mathtt c - \frac{5}{4} \mathtt c) N_0^2> 4d\kappa^2\, , \quad (\frac34 - \frac38) \mathtt C N_0^{\tau_1}> 4d\kappa^2\,, \quad  N_0^{\tau_1}
2^{-\frac{N_0}{2\kappa}+1}<1$$  we apply Proposition \ref{qtop0bis}  and obtain the desired bounds for $P \circ \Psi$ with the parameters
$(s/2,r/2, \theta= \frac38 \mathtt C, \mu= \frac 32 \mathtt c, K> N_0, 2\e^2, \mathcal O_0)$. Then we apply the  last change of variables $\Xi$ we need to satisfy again the neighborhood Lemma  and reduce the parameters to $\theta_0=\mathtt C/2, \mu_0= 2 \mathtt c$ moreover we reduce the analyticity radius by $\frac 1{\eufm A}$. We obtain the desired result. \end{proof}

\subsection{Final conclusions: solutions of the NLS}
\begin{proposition}
The Hamiltonian of the NLS in the final coordinates is a compatible Hamiltonian in the sense of Definition \ref{buone} and satisfies the hypotheses of Theorem \ref{gacom} provided we choose $r= \e^2$ and $\e$ small.
\end{proposition}
\begin{proof}
The fact that it  satisfies the smallness condition in Theorem \ref{gacom} follows from \eqref{stimefinali}.
We need to verify all the conditions $(A1)-(A5)$.

\noindent $(A1)$\quad {\it Non--degeneracy:} The map $\xi\to
\omega(\xi)$ is $\xi\mapsto \vgot -2\xi$ so it is a lipeomorphism from  $\mathcal O_0$ to
its image with $ |\ome^{-1}|^{lip}_\infty \leq 1$. We have $ |\ome(\xi)-\vgot|\leq   \varepsilon^2$ since by assumption $|\xi|\leq \varepsilon^2$.
\bigskip

\noindent $(A2)$\quad  {\it Asymptotics of normal frequency:}
 For all $n\in S^c$ we have a decomposition:
\begin{equation}\label{asympN} \Omega _n(\xi)=\s(n)(|\er(n)|^2+2\val_n(\xi)).
\end{equation} 
In our case we start from $  \tilde
\Omega _n(\xi)=0$.
  We know that the $ \val _n(\xi)$ are   in a finite list of analytic functions which are homogeneous of degree one in $\xi$.  As for \eqref{omelip}, by homogeneity,  we can fix $M\geq 1$ so that $
 2+ 2|\val|^{lip}_\infty \leq M\,,\quad  2|\val|_\infty \leq M \e^2\,.  
 $ 
 This fixes the parameters $M,L$, then  we fix $K_0$ large enough (independently form $r_0$) and  choose $\gamma   < \min(2 M ,B)\e^2$ with $B=B(K_0,\O_0)$ given in the iterative Lemma.    
\bigskip

 \noindent$(A3)$\quad  {\it Regularity and Quasi--T\"oplitz property:} The function $ \tilde\Omega (z)=0.$ The functions $P$, $ \val(z):=\sum_j\vartheta_j |z_j|^2$    are $M$--regular, preserve momentum as in \eqref{mome},  are Lipschitz in the parameters $\xi$. Then   $P$ is  quasi-T\"oplitz
with parameters $(s_0,r_0,{\Pappa_0},\theta_0,\mu_0, \gamma/M, \ro \mathfrak K_\alpha)$ by Proposition \ref{Pfinale}, with the bounds \eqref{stimefinali}. Moreover we know that the functions $\val_i$ are constant of the strata $\Sigma_{\GA,a}$  of \S \ref{gac} hence $ \val (z) :=\sum_j\vartheta_j |z_j|^2$ is  regular, preserve momentum and is quasi--T\"oplitz and for all $N \geq K_0$ $\tau_0\leq \tau\leq \tau_1/4d$ we have
$\Pi_{(N,\theta,\mu,\tau)} \sum_j\val_j |z_j|^2 \in \mathcal T_{(N,\theta,\mu,\tau)}$.
 \bigskip

\noindent $(A4)$\quad{ \it Smallness condition}:  We compute $\Theta,|\vec \ix|\leq C \e^3 \gamma^{-1}$ by \eqref{stimefinali}. The condition
\begin{equation}
 \Theta<{1} \,, \;LM |\vec\ix|<1\,,\; \gamma^{-1}\| X_{ \tilde\Omega  }\|^{T}_{\vec p}<  1\,,\; \kappa e |\vec{\ix}| {\Pappa}_0^{4d \tau_1} \ll 1.
\end{equation}  This translates to the condition 
\eqref{smeg}  in which $\alpha=3$ by \eqref{scossa}  with $r= \e^2$. Finally we   note that the third condition in $(A4)$ is trivial since $\tilde\Ome=0$.
 \bigskip 

\noindent$(A5)$\quad {\it Non--degeneracy (Melnikov conditions):}  For all $(k,l)\neq 0$ compatible with momentum conservation the function $\langle \ome, k\rangle  +  (\Ome , l)$  is of the form $\langle \vgot, k\rangle  + ( \Vgot , l)-2\sum_ik_i\xi_i+2\theta$ where $\theta$ can be $0,\pm\val_j, \pm \val_j\pm \val_k$.    From the main result  of \cite{PP1} we know that all the  functions $\sum_ik_i\xi_i+ \theta$ are analytic, homogeneous of degree 1 and different from 0 and give distinct eigenvalues on distinct blocks. 
Hence   in each connected component $(\R_+)^n_\alpha$  of $(\R_+)^n\setminus \mathfrak A$ we can choose a  compact domain $\mathfrak K_\alpha$ which does not intersect any of the zero curves of  the functions $(\langle \ome, k\rangle  +  (\Ome, l)) 0$ for $|k|<16 \sqrt{n}$, this amunts, as explained in Remark \ref{lemme},  to taking the  $\mathfrak K_\alpha$ disjoint from finitely many hypersurfaces describing the resultants  given in that Remark. Now fix $\mathfrak K_\alpha$,  for each $k,l$ as above  such that $\langle \vgot, k\rangle  + ( \Vgot , l)=0$ we have that $\langle \ome, k\rangle  +  (\Ome, l) $ is homogeneous of degree one  and non zero. Then for all positive $\rho$ such that $\rho\xi\in \mathfrak K_\alpha$ we have
$$
\frac{|\langle \ome(\rho\xi), k\rangle  +  (\Ome(\rho\xi), l)-(\langle \ome(\xi), k\rangle  +  (\Ome(\xi), l))|}{ (|1-\rho||\xi|}=
\frac{\langle \ome, k\rangle  +  (\Ome, l)}{|\xi|}\geq a>0
$$
for some positive $a:=a(\alpha)$.  We repeat the same argument for all $\alpha$ and we choose $\mathcal O_0= \e^2 \mathfrak K $. 

\end{proof}
For compatible Hamitonians we  have proved in the previous section a general Theorem \ref{gacom}  which ensures the existence  of KAM--tori, now we can apply this Theorem to the NLS  and have as final result Theorem \ref{ebbe} .

\bibliographystyle{plain}

\bibliography{/Users/procesi/Documents/scienza/lavori-in-corso/bibliografia}
\end{document}